\documentclass[12pt]{amsart} % use larger type; default would be 10pt
\usepackage[T1]{fontenc}
\usepackage{amsmath}
\usepackage{amsthm}
\usepackage{amssymb}
\usepackage{mathrsfs}
\usepackage{pdfpages}
\usepackage{tikz-cd}

\usepackage{geometry} % See geometry.pdf to learn the layout options. There are lots.
\usepackage{graphicx} % support the \includegraphics command and options
\usepackage{extarrows}

\usepackage[colorlinks=true,citecolor=blue]{hyperref}

\newtheorem{theorem}{Theorem}
\newtheorem{proposition}[theorem]{Proposition}

\newtheorem{lemma}[theorem]{Lemma}
\newtheorem{corollary}[theorem]{Corollary}
\theoremstyle{definition}
\newtheorem{definition}[theorem]{Definition}

\theoremstyle{remark}

\newcommand{\G}{\mathbf{{G}}}

\newcommand{\R}{\mathbf{R}}
\newcommand{\Q}{\mathbf{Q}}
\newcommand{\Z}{\mathbf{Z}}

\renewcommand{\L}{\mathbf{{L}}}
\newcommand{\Gal}{\mathscr{{G}}}
\newcommand{\Hgp}{\mathbf{{H}}}
\renewcommand{\H}{\mathbf{{H}}}
\newcommand{\g}{\mathfrak{g}}
\newcommand{\h}{\mathfrak{h}}
\renewcommand{\l}{\mathfrak{l}}
\newcommand{\w}{\mathfrak{w}}

\newcommand{\Rat}{\mathscr{R}}
\newcommand{\RatQ}{\mathscr{R}_\Q}

\newcommand{\Lscr}{\mathscr{R}_\Gamma}
\newcommand{\Supp}{\textbf{Supp}}
\newcommand{\Hom}{\textbf{Hom}}

\newcommand{\Ad}{\textbf{Ad}}
\newcommand{\Cji}{\text{Conj}^{-1}}
\newcommand{\Stab}{\text{Stab}}
\renewcommand{\det}{\text{det}}
\newcommand{\tens}{\otimes}

\newcommand{\lie}[1]{\mathfrak{#1}}

\newcommand{\ciso}{\simeq}

\newcommand{\leftexp}[2]{{\vphantom{#2}}^{#1}{#2}}
\newcommand{\Nm}[1]{{\left\|{#1}\right\|}}
\newcommand{\abs}[1]{{\left|{#1}\right|}}

\newcommand{\eps}{\varepsilon}

\newcommand{\Lpp}{L^{\scalebox{.5}{$\boldsymbol{++}$}}}

\title[Limit distributions for translations in arith. homog. spaces]{Limit distributions of translated pieces \\of possibly irrational leaves \\in~$S$-arithmetic homogeneous spaces}
\author{Rodolphe Richard, Thomas Zamojski}
\numberwithin{equation}{section}
\begin{document}

\begin{abstract}
The purpose of this article is to describe and characterize the limit distributions of translates of a  bounded open ``piece of orbit'' of a reductive subgroup on a space of $S$-arithmetic lattices. This is accomplished under a mild assumption of ``analytic stability'' on the sequence of translates. It is important to note however that it is not necessary to assume that the reductive subgroup or its centralizer are algebraic. Moreover, it is also not necessary to assume that the initial orbit is of finite measure or even closed. This article thus provides important generalizations on previously known results and opens ways to new applications in number theory, two of which are briefly mentioned. 
\end{abstract}

\maketitle
%\begin{abstract}
%This work fits in the context of “homogeneous dynamics” of “algebraic measures”. Namely the study of closed orbits, inside lattices spaces, and of homogeneous measures on it. We bring a general answer to questions about the processes at work when a sequence of translates of an homogeneous measure converge to some limit, continuing a tradition going back to Dani-Margulis work on Raghunathan’s conjecture (now a theorem of Ratner). In the work of Eskin, Mozes and Shah, the term “focusing" has been coined to describe the convergence of algebraic orbits to another one. 
%
%We are able to study the focusing of “pieces” of  closed orbits, instead of full orbits, and even pieces of not necessarily closed leaves of non necessarily
% algebraic Lie subgroups. We furthermore generalise previous theory to the S-adic setting. We can evade the main hypothesis of the cited authors under a technical assumption from an earlier work of Richard and Shah we rely upon. Our method is also expected to cover non arithmetic lattices. 
% 
% Thus relaxing several hypothesis needed in existing work on these questions, this opens the way to new results in a range of applications, such as: counting problem, diophantine approximation, equidistribution in arithmetic geometry.
%
%\end{abstract}
\setcounter{tocdepth}{2}
\setcounter{secnumdepth}{4}
\tableofcontents

%\subsection*{Summary}\phantom{.}
%
%Preface by an associated author ? 
%
%\paragraph{To do:}
%
%\paragraph{Credits}
%This work could not exist without earlier work of (notably) Dani, Eskin, Kleinbock, Margulis, Moz{e}s, Ratner, Shah, Tomanov.
%
%Present work stems from a collaboration visiting Pr Shah to at ICTP (Mumbai) in january 2005 (producing \cite{Lemma}), helped by the  \emph{\'{E}cole normale sup\'{e}rieure de Paris}. Tools were developped in the thesis~\cite{Lemma,Lemmap,LemmaA} at IRMAR (Rennes). These were deepened with a visit of Pr Shah at Irchel Universit\"{a}t (Z\"{u}rich) in june 2011 and OSU (Columbus OH) in september 2011. Through collaboration with T. Zamojski at \'{E}PFL (Lausanne) in the academic year 2011-2012, this gave the preprint~\cite{Lemmanew}, which updates~\cite{Lemma}, and the present developments. Writing occured at \'{E}PFL, Plouisy (France), ETHZ, and Mahina (Tahiti). The preview Section~\ref{SecShimura} about Shimura varieties is the development which was the implicit perspective of the thesis~\cite{These,These3}.
%
%\newpage

\section{Introduction}
Several problems in number theory involve understanding the limiting distribution of translates of certain orbits of subgroups in a space $G/\Gamma$ of $S$-arithmetic lattices. For example, \cite{DukRudSar93} suggests such an approach to the study of the density of integral points on affine homogeneous varieties  under a semisimple Lie group, an approach which was further pursued in \cite{EskMcm93,EMSAnn,GorMauOh08,GorOh09}. Another example pertains to Galois action on Hecke orbits in Shimura varieties and a conjecture of Pink~\cite{These3}. Set aside applications to number theory, the problem of classifying and characterising these limit measures is an interesting problem in homogeneous dynamics in its own right, leading to the development of new results, notably in~\cite{Lemmanew}.   

The setting for the problem is as follows: let $S$ be a finite set of places, and~$\Q_S$ the product of the corresponding completions of~$\Q$. Denote~$G$ the group of $\Q_S\text{-}$points of a semisimple $\Q$-algebraic group, $\Gamma$ an $S$-arithmetic lattice in $G$, and $H$ a connected reductive $\Q_S$-subgroup in $G$ (see~\S\ref{secintro}). In~\cite{EMSAnn}, it is assumed furthermore that $\Q_S=\R$, that the orbit $H\Gamma/\Gamma$ in $G/\Gamma$ supports a globally $H\text{-}$invariant probability measure $\mu_H$ and that $H$ is defined over $\Q$. One question is then to describe weak limits of sequences of translates $(g_n\cdot\mu_H)_{n\geq0}$, with~$g_n\in G$. Although an answer is obtained for any non-divergent sequences, every application to counting integral points on affine homogeneous varieties work under the more stringent assumption that $H'$, the centralizer of $H$ in $G$, be defined over $\Q$ and be $\Q$-anisotropic. This avoids the trivial issue of pushing all the mass to infinity using elements of $H'$. One wish is to encompass this phenomena in the ergodic method to counting and relax the hypothesis on~$H^\prime$.

The purpose of this article is to continue the study of the limit measures of translates of orbits of reductive subgroups in $G/\Gamma$, removing several of the assumptions of \cite{EMSAnn}. We continue a tradition dating back to Raghunathan’s conjecture, who was motivated by Oppenheim conjecture (now theorems of Ratner and Margulis respectively), and to the linearisation method initiated by Dani and Margulis for unipotent trajectories, pursued in work of Eskin, Mozes, Shah, together with the more recent property of $(C,\alpha)$-good functions, and its adaptation to the~$S$-adic setting by Kleinbock and Tomanov. 

Considering the two applications to number theory mentioned previously, cases arise when $H$ is not necessarily rational, $H'$ is not $\Q$-anisotropic and when $H\Gamma/\Gamma$ is not finite volume, or more drastically, not even closed.\footnote{For instance, if~$A$ is an abelian variety over a field of zero characteristic,~$S$ is made of a $\ell$-adic place and the archimedean place, and~$G$ is te Mumford-Tate group of~$A$. After maybe a finite extension of the base field, the image of the Galois representation on the~$\ell$-adic Tate module sits in~$G(\Q_\ell)$. This is a compact~$\ell$-adic Lie subgroup which is an open in an $\Q_\ell$-algebraic subgroup~$H$ of~$G$. If in the study~\cite{These3} we wish not to rely on Mumford-Tate conjecture (for a finitely generated base field), which asserts one should have~$H=G$, we are to consider general subgroups~$H$ defined over~$\Q_\ell$ and maybe not over~$\Q$.} It is then necessary to focus on a part of the orbit: let $\Omega$ be a Zariski dense open bounded subset of $H$ with zero measure boundary, and let $\mu_\Omega$ be the push forward under $G\to G/\Gamma$ of the restriction to $\Omega$ of an invariant measure on $H$. Under a mild assumption on the sequence $(g_n)_{n\geq0}$, we describe the weak limits of translates $g_n\cdot\mu_\Omega$ and we furthermore give a sufficient and necessary condition on the sequence to converge to this limit.  Following~\cite{EMSAnn},  we name the latter the~\emph{focusing criterion}. 

%The purpose of this article is to describe the weak limit measures of translates of a ``piece of orbit" in the ``$S$-arithmetic setting".  Furthermore, given such limit measure, we give a sufficient and necessary condition on the sequence to converge to this limit.  After~\cite{EMSAnn},  we name it the~\emph{focusing criterion}. 

\subsection{Two particular cases}
Before making our problem precise in section~\ref{secintro}, after having described at length our setting, we provide straight away two simplified cases of our main result~Theorem~\ref{Theorem}. We deem these case more useful to some readers than our most general statement.

\subsubsection{A particular case} This first simplified version contains little novelty with what was known to experts, but exhibits the kind of results we will present, while introducing the way our main theorem is formulated.
\addtocounter{theorem}{-1}
\begin{theorem} \label{Theointro} Let~$G$ be a real linear semisimple algebraic group defined over~$\Q$, and let~$\Gamma$ be an arithmetic lattice relative to the~$\Q$-structure on~$G$. Let~$H$ be a real algebraic connected subgroup of~$G$ defined over~$\Q$ which is reductive as an algebraic subgroup\footnote{The centre of~$H$ is semisimple in~$G$.}. Assume the centraliser~$H^\prime$ of~$H$ in~$G$ is $\Q$-anisotropic. 

Let~$\Omega$ be a non-empty open bounded subset of~$H$, let~$\mu|_\Omega$ be a probability on~$\Omega$ which is the restriction of a Haar measure of~$H$, and demote~$\mu_\Omega$ be the image probability measure on~$G/\Gamma$.

Fix any sequence~$\left(g_i\cdot\mu_\Omega\right)_{i\geq0}$ of translated probabilities on~$G/\Gamma$ and let $\mu_\infty$ be any weak limit measure. We have the following.
\begin{enumerate}
\item{\textbf{Tightness.}} The limit $\mu_\infty$ is a probability measure.
\item{\textbf{Limit probabilities.}} There are
\begin{itemize}
\item  a~$g_\infty$ in~$G$ and
\item an algebraic subgroup~$L$ of~$G$ defined over $\Q$, normalised by~$H$ and generated as a~$\Q$-algebraic group by its real unipotent elements, with analytic connected component of the identity~$\Lpp$,
\end{itemize}  such that
\begin{equation}\label{limitformulasimplified}
\mu_\infty= g_\infty\int_\Omega (\omega\cdot\mu_{\Lpp})~~d\mu(\omega),
\end{equation}
where $\mu_{\Lpp}$ is the~$\Lpp$-invariant probability measure on~$\Lpp\Gamma/\Gamma$.
%writes~$\mu_\infty$ as integral sum of translates of the~$L^+$-invariant probability~$\mu_{L^+}$ on~$L^+\Gamma/\Gamma$;
\item{\textbf{Focusing criterion.}} %If the sequence~$\left(g_i\cdot\mu_\Omega\right)_{i\geq0}$ converges to the probability~\eqref{limitformulasimplified}, for some~$L$ as above, then
The sequence~$\left(g_i\right)_{i\geq0}$ is of the class
\begin{equation}\label{focusingcriterionsimplified} g_\infty\cdot o(1)\cdot L\left(\Gamma\cap N(L)H^\prime\right),\end{equation}
where~$o(1)$ denotes the class of sequences in~$G$ converging to the identity, $N(L)$ is the normaliser of~$L$ in~$G$, and $H^\prime$ is the centraliser of~$H$.
\end{enumerate}
\end{theorem}

%The first conclusion follows from earlier work (\cite{EMSGAFA}, \cite{Lemma,LemmaA}). The second and third conclusions are proved together, and are an instance of Theorem~\ref{Theorem} in the simplified case above.
Theorem~\ref{Theointro} is a slight generalisation of~\cite[Statement 1.13]{EMSAnn}. Moreover, our main theorem will generalise further this statement in several aspects (see~\S\ref{seccontext}). Most importantly, $H$ will no longer be required to be defined over $\Q$, nor to have a closed orbit $H\Gamma/\Gamma$. We will however need to restrict to sequences having a geometric stability property (see \S\ref{analytic stability}).
Here this hypothesis was evaded thanks to~\cite{Lemma}, which is related to~\cite{EMSGAFA} (see~{\S}\ref{RemarkcasGAFA}).
Nonetheless, this restriction is flexible enough to encompass all applications of~\cite{EMSAnn} to counting integral points on varieties, and even more.

\subsubsection{Another particular case}
Our theorem also opens the way to applications to Galois actions on Hecke orbits in Shimura varieties. Measures on bounded pieces of orbits of irrational groups $H$ arise in this context. To illustrate this, we give the following version of our theorem, requiring much of the power of our main Theorem~\ref{Theorem} (and also Theorem~\ref{thm21bis}). For simplicity we consider here only one ultrametric place. The reader might adapt at wish to finitely many ultrametric places.

\begin{theorem}\label{Theointro2} Let~$G$ be a semi-simple linear algebraic group over~$\Q$. Fix a prime~$p$ and a compact subgroup~$K$ of~$G(\Q_p)$ Let~$\Gamma$ be an~$S$-arithmetic lattice\footnote{With respect to the  couple~$S:=\left\{\Q\to\R;\Q\to\Q_p\right\}$ of places. See~\cite[{\S}(3.1.2)]{Margulis} or~{\S}\ref{secmuomega} for the definition of a $S$-arithmetic lattice we use.} of~$G(\R)\times G(\Q_p)$. Consider the Haar probability measure~$\mu$ on~$K$ and define~$\mu_K$ to be its direct image in the quotient space~$\left.G(\R)\times G(\Q_p)\middle/\Gamma\right.$.

Assume the following concerning~$K$:
\begin{enumerate}
\item[(H1.)] \label{H1intro} that the centraliser of~$K$ in~$G(\Q_p)$ is nothing more than the centre of~$G(\Q_p$);
\item[(H2.)] \label{H2intro} that~$K$ is reductive in~$G(\Q_p)$ (equivalently: its Lie algebra is semisimple);
\item[(H3.)] \label{H3intro} that the Zariski closure of~$K$ over~$\Q_p$ is Zariski connected.
\end{enumerate}

Then the following conclusions hold.
\begin{enumerate}
\item The familly of translated probabilities
\begin{equation}\left(g\cdot \mu_K\right)_{g\in G(\Q_p)}\end{equation}
 is tight.
\item The limit points of the family~$\left(g\cdot \mu_K\right)_{g\in G(\Q_p)}$ are all translate of one of the measures of the form
\begin{equation}\label{theointropadiclimitformula}
\mu_{K\star L} = \int_{\kappa\in K} \kappa\cdot \mu_{\Lpp}~d\mu(\kappa).
\end{equation}
\begin{itemize}
\item where~$L$ is a subgroup of~$G$ defined over~$\Q$ without~$\R\times\Q_p$-anisotropic $\Q$-rational factor (some say “of non compact type”), generated over~$\Q$ by its unipotent elements defined over~$\R\times\Q_p$,  (of \emph{Ratner class} in~{\S}\ref{secnotations})
\item where~$\Lpp$ is an explicit (defined in~{\S}\ref{secnotations}) open subgroup of finite index in~$L(\R\times\Q_p)$,
\item and where~$\mu_{\Lpp}$ is the~$\Lpp$-invariant probability on~$\Lpp\Gamma/\Gamma$.
\end{itemize}
\item The measure~$\eqref{theointropadiclimitformula}$ can occur as a limit point for a subgroup~$L$ distinct from~$G$ and~$\{e\}$ if and only if there exists an unbounded algebraic subgroup of~$G(\Q_p)$ which is normalised by~$K$ and not Zariski dense over~$\Q$ in~$G$.
\item Assume that~$(g_i\cdot \mu_\rho)$ converges to a translate of~$\mu_{K\star G}$ for every sequence~$(g_i)_{i\geq 0}$ in~$G(\Q_p)$ without bounded infinite subsequence. Then the Zariski closure over~$\Q$ of~$K$ in~$G$ contains the~$\Q_p$-isotropic factors of~$G$.
\end{enumerate}
\end{theorem}
The first point is a particular case of~\cite{LemmaA}. The second point is a particular case of the limit formula of our result Theorem~\ref{Theorem}.  The third point is derived from the focusing criterion of Theorem~\ref{Theorem}, and the last point is a consequence of the third.
These last two points serve as a criterion for
``indiscriminate'' equidistribution, for all reasonable sequences: there is no focusing phenomenon.

\subsubsection{}
Both applications discussed above shall be the subject of future work.
 
The proof follows the strategy of~\cite{EMSAnn}, thus relies on Ratner's measure classification theorem and on Dani and Margulis' linearisation method. However, a significant difference in our proof is that the geometric stability properties from~\cite{Lemmanew} provides a major simplification to the linearisation method (see~{\S}\ref{subsection-proof-structure}). It is responsible for the strong conclusions obtained in our main Theorem~\ref{Theorem}.

%Our proof follows the strategy of~\cite{EMSAnn} but is a new in certain key technical aspects of its realisation. In particular, our method, which can be thought as more "orthodox", gives interesting and detailed information on the way the convergence occurs, from the~$g_i\cdot\mu_\Omega$ to~$\mu_\infty$.

%\subsection*{Structure of the article} Section~\ref{secintro} introduces the setting for the problem, subgroups relevant to the proof and the stability assumption mentioned above. Section~\ref{secthm} consists of the commented statement of the result, whereas its proof lies in Section~\ref{secproof}. Appendices... 

\subsection{Acknowledgements}
This work could not exist without earlier work of (notably) Dani, Eskin, Kleinbock, Margulis, Moz{e}s, Ratner, Shah, Tomanov.

Present work stems from a collaboration of the first named author visiting Pr~Shah at ICTP (Mumbai) in January 2005 (producing \cite{Lemma}), helped by the  \emph{\'{E}cole normale sup\'{e}rieure de Paris}. Tools were developed in the thesis~\cite{Lemma,Lemmap,LemmaA} at IRMAR (Rennes). These were deepened with a visit of Pr Shah at Irchel Universit\"{a}t (Z\"{u}rich) in June 2011 and to him at OSU (Columbus OH) in september 2011. At \'{E}PFL (Lausanne) in the academic year 2011-2012, this gave the preprint~\cite{Lemmanew}, which updates~\cite{Lemma} with the cornerstone result
allowing the present developments, a collaboration with T. Zamojski. Writing occurred at \'{E}PFL, Plouisy (France), ETHZ, and Mahina (Tahiti), 
and were lately supported by Universiteit Leiden.% The preview Section~\ref{SecShimura} about Shimura varieties is the development which was the implicit perspective of the thesis~\cite{These,These3}.

The second named author was supported by the European Research Council Grant 228304.

Both authors first met at a summer school in Pisa organised by the Clay institute about homogeneous dynamics.

%\section{Local Focusing for transcendent algebraic measures}

\section{Setting} \label{secintro}
Our result, Theorem~\ref{Theorem}, is set up in the $S$-arithmetic setting, which we introduce in~\ref{Ssetting} below.
This Theorem involves translates of quite general probability measures~\(\mu_\Omega\), the alluded “pieces of orbits”, which we describe in~\ref{secmuomega}. For full generality, we rely on a technical hypothesis detailed in~\ref{analytic stability}. We then emphasize a particular case
where this hypothesis is void, a particular case which encompass the works of ~\cite{EMSAnn, EMSGAFA}. Lastly we review some of Ratner theory and introduce some notations, especially the measures~$\mu_{\Lpp}$ our limit measures will be made of.
\subsection{The~$S$-arithmetic setting} \label{Ssetting}
 This setting is analogous to \cite{Borel-Prasad}, \cite{MargulisTomanov}, \cite{TomanovOrbits}, \cite{KT}, \cite{LemmaA}.

Let~$S$ be a finite set of places~$v:\Q\to\Q_v$. We consider the  notation
\begin{equation}\label{notationQS}
\Q_S:=\prod_{v\in S}\Q_v.
\end{equation}
Let~$\G$ be a semisimple algebraic group over~$\Q$. Let~$G$ be the topological group~$\G(\Q_S)$, which we will identify with~$\prod_{v\in S}\G(\Q_v)$. The \emph{Zariski topology on~$G$}, and its subsets, will mean the product topology of the Zariski topologies induced on the~$\G(\Q_v)$.

We consider a subgroup~$H$ of~$G$ such that
\begin{equation}
\label{H reductif dans}
\text{``$H$ is a connected~\emph{$\Q_S$-subgroup} which is \emph{reductive in~$G$}''.}
\end{equation}
 Namely~$H$ can be written~$\prod_{v\in S}H_v$ where
\begin{itemize}
\item if~$\Q_v\ciso\R$ is archimedean,~$H_v$ is a connected real Lie subgroup of~$\G(\R)$ whose Lie algebra has semi-simple adjoint action on the Lie algebra of~$\G(\R)$,
\item if~$\Q_v$ is ultrametric,~$H_v$ is the group~$\Hgp_v(\Q_v)$ of $\Q_v$-rational points of a Zariski connected reductive algebraic subgroup~$\Hgp_v$ of~$\G$ defined over~$\Q_v$.
\end{itemize}
%TODO
%A \emph{norm}~$\Nm{-}$ on a~$\Q_S$ module (...)

\subsection{Sequences of probabilities}\label{secmuomega} Given this setting, we consider a non-empty open bounded subset~$\Omega$ of~$H$, which is necessarily~$\Q_S$-Zariski dense in~$H$.
%; if not, we may replace~$H$ by the~$\Q_S$-Zariski closure of~$\Omega$ in~$H$. %The density of~$\Omega$ always holds if~$H$ is made of a connected real Lie group and Zariski connected algebraic groups~$\Hgp_v$.

From assumption~\eqref{H reductif dans}, we note that~$H$ is a unimodular group (\cite[p.15, Exemple de modules a)]{vigneras1996representations}, \cite[Corollary 8.31 (d)]{Knapp}); its Haar measures are both left and right Haar measures. Let~$\mu$ be the Haar measure on~$H$ normalised so that~$\mu(\Omega)=1$. Fix a lattice~$\Gamma$ of~$G$. Denote~$\mu_\Omega$ the image probability measure on~$G/\Gamma$ of the restriction of $\mu$ to $\Omega$. We will assume here for simplicity that~$\Gamma$ is an $S$-arithmetic lattice, in the sense of~\cite[{\S}(3.1.2)]{Margulis} (compare~\cite{TomanovOrbits}, confer~\cite[\S\,1,\S\,6]{Borelihes}).

%%%%%%% Variant for mu de f au lieu de mu omega %%%%%%%%%%

%\paragraph{A variant} Following the setting of~\cite[Theorem~1.3]{LemmaA}, we could have replaced~$\mu_\Omega$ with the following generalisation. Take~$f$ a bounded positive~$\mu$-integrable function on~$H$. For example the characterisitc function of~$\Omega$. Define~$\mu_f$ to be the direct image of~$f\cdot\mu$ in~$G/\Gamma$. We replace~$\mu_\Omega$ with~$\mu_f$ provided the latter is non zero.

We are interested in the asymptotic behaviour of a sequence~$\left(g_i\cdot\mu_\Omega\right)_{i\geq0}$ of probabilities on~$G/\Gamma$ made of translates of~$\mu_\Omega$ by elements~$g_i$ in~$G$. We ask:
\begin{enumerate}
\item   what are the limit measures (in particular are they also probabilities);
\item   how to characterise conveniently the sequences converging to a given limit measure.
\end{enumerate}
We will completely answer these questions provided the translating elements~\(g_i\) are constrained to any subset~\(Y\) of~$G$ satisfying the following stability property.

\subsubsection{Variation and Full orbits.} Let us mention a possible variation to our setting. As in~\cite{LemmaA}, we could consider more general measures of the form~$\mu_f=f\cdot\mu_\Omega$ where~$f\geq0$ is of class~$L^1(\mu_\Omega)$ and such that~$\int f\cdot\mu_\Omega=1$. We will not
delve here into such a generalisation. It could be useful nonetheless, for instance in order to rigorously relate our results on the measure~$\mu_\Omega$ on a piece of orbit, and analogous statements about probability measure on a full closed orbit (assuming there is such, which may not be the case in our general setting). We refer to~\cite{LemmaA} for such considerations. (See also~\cite{Lemma} for a variant.) A nice treatment would also study closed orbits
of maybe infinite homogeneous measures; such study is for a later work.

\subsubsection{On general Lattices} There is a more general notion of~$S$-arithmetic lattices (see~\cite{TitsMargulis}), and semi-simple groups can feature interesting non arithmetic lattice.
We limit ourselves to a restricted notion of lattice for commodity: we rely technically on quite a number of references. This will might hamper applications
to Shimura varieties, which maybe associated with general arithmetic lattices (cf~\cite{These3}.) We expect nevertheless our 
method to adapt for general lattices without serious trouble. It deals mostly with Dani-Marulis’ linearisation, which has already been applied
to more general lattices, and rely on Ratner’s theorems and its variants in comparable generalities. Nonetheless, the exigency of rigour led us
at technical points to require the use of explicit structure properties of the groups~$L$ of Ratner class (see~Appendix~\ref{AppRatner}). 

%We will consider a weakly converging sequence~$\left(g_i\cdot\mu_\Omega\right)_{i\geq0}$ with limit~$\mu_\infty$ and assume that~$\mu_\infty$ is non zero\footnote{The non vanishing of~$\mu_\infty$ will actually follow from~\cite{LemmaA} under the hypothesis~\eqref{Hypo} that will be introduced shortly.}.
%Our main result Theorem~\ref{Theorem} p.\pageref{Theorem} identifies the possible forms of the limit~$\mu_\infty$, and its Corollary~\ref{Corollary} describes the sequences converging to a given non zero limit. %Assuming the above-mentioned hypothesis.
%

%We will use the linearisation techniques~\cite[{\S}2]{DM} of Dani-Margulis, following the method of~\cite[{\S}3]{EMSGAFA} of Eskin-Mozes-Shah, or rather its $S$-arithmetic extensions appendix~\ref{AppA} is devoted to. We will also use the classification theorem~\cite{Rat} of Ratner~\cite[?]{Ratnerp}, and more precisely the~$S$-arithmetic version as in~\cite[Theorem~2]{MargulisTomanov}, together with refinements from~\cite{TomanovOrbits}. We refer to appendix~\ref{AppRatner} for precisions. 

\subsection{\emph{Analytic stability} hypothesis}\label{analytic stability}
% We will also need the geometric stability lemmas~\cite[Theorem~1, Theorem~2]{Lemma,Lemmanew} of Richard-Shah and~\cite{Lemmap} of Richard.
%This leads to the following technical hypothesis.

The \emph{analytic stability} property on a subset~$Y$ of~$G$ of \cite{Lemma,Lemmanew,Lemmap} is defined as follows:
\begin{equation}\tag{An.S.}\label{AnS}
\begin{array}{l}
\text{For any $\Q_S$-linear representation~$\rho:G\to GL(V)$, for any norm~$\Nm{-}$ on~$V$,}\\
\exists c>0,~\forall y\in Y,~\forall v\in V,~\sup_{\omega\in\Omega}\Nm{y \cdot \omega\cdot v }\geq \Nm{v}/c.
\end{array}
\end{equation}
 Note that the property can be made independent of~$\Omega$, as long as~$\Omega$ is bounded and Zariski dense in~$H$.
The hypothesis under which our proof works is that
\begin{equation}\tag{H}\label{Hypo}
 \text{ the set~$Y:=\{g_i~|~i\geq0\}$ satisfies~\eqref{AnS}.}
\end{equation}

We explain why such hypothesis is reasonable. On the one hand, it implies that the sequence~$\left(g_i\cdot \mu_\Omega\right)_{i\geq0}$ is tight (\cite[Theorem~1.3]{LemmaA}). One the other hand, this hypothesis is not far from being equivalent to tightness. More precisely, thanks to Mahler's criterion and \cite{KT}, one can closely relate tightness of the sequence~$\left(g_i\cdot \mu_\Omega\right)_{i\geq0}$ to the following~\emph{(uniform) Arithmetic (semi)stability} statement.
\begin{equation}\tag{Ar.S.}\label{ArS}
\scalebox{0.92}{
$\begin{array}{l}
\text{For any $\Q$-linear representation~$\rho:G\to GL(V)$, for any norm~$\Nm{-}$ on~$V\tens\Q_S$,}\\
\text{for any $\Q$-rational $S$-arithmetic lattice~$\Lambda$ in~$V\tens\Q_S$,}\\
\exists C>0,~\forall y\in Y,~\forall v\in \Lambda\smallsetminus\{0\}, \max_{\omega\in\Omega}\Nm{y \cdot \omega\cdot v }\geq C.
\end{array}
$}
\end{equation}
One deduce~\eqref{ArS} from~\eqref{AnS} as follows. We note that~$\Lambda$ is discrete in~$V\tens\Q_S$. Thus its systole~$\inf \{\Nm{\lambda}~|~\lambda\in\Lambda\smallsetminus{0}\}$ is bounded below. Let~$\sigma>0$ be a bound. Then~\eqref{AnS} for~$V\tens\Q_S$ implies~\eqref{ArS} with~$C=\sigma/c$.

Last but not least, references~\cite{Lemma,Lemmap,LemmaA,Lemmanew} show that hypothesis~\eqref{AnS} can be achieved for explicit and quite large subsets~$Y$ (see~Theorem~\ref{thm21bis}).

Came lately to the notice of the authors the work of Grayson~\cite{Grayson1,Grayson2} where he produces related subsets of a symmetric space~$G/K$
defined by a related (more precise) property of \emph{arithmetic stability} (coined by Stuhler~\cite{Stuhler}), in relation to reduction theory for arithmetic groups.

\subsection{An important special case}\label{RemarkcasGAFA}
One can avoid the extra hypothesis~\eqref{Hypo} in the following special case which is important considering the applications\footnote{Notably~\cite{EMSAnn}, and its wanted generalisations; but also Galois actions on Hecke orbits in Shimura varieties, see~\cite{These3}.}.
\begin{equation}\tag{R.nD.}\label{RnD}
\text{The centraliser~$H^\prime$ of~$H$ in~$G$ is defined over~$\Q$ and is~$\Q$-anisotropic.}
\end{equation}
Here, by~``$\Q$-anisotropic" we mean satisfying the Godement compacity criterion~\cite[Théorème~8.7\footnote{Note that here,~$H^\prime$ is necessarily reductive, and has trivial unipotent radical.}]{BorelIntro}. On top of the rationality property for~$H^\prime$, the anisotropy translates in 
a non divergence property for sequences of translates~$g_i\cdot \mu_\Omega$ (see~\cite{LemmaA} following~\cite{EMSGAFA}).

In the case~\eqref{RnD}, the hypothesis~\eqref{Hypo} can always be achieved by modifying the sequence of translators~$\left(g_i\right)_{i\geq0}$ by right multiplying by elements of~$H^\prime\cap\Gamma$, such modification leaving the sequence of probabilities~$\left(g_i\cdot\mu_\Omega\right)_{i\geq0}$ unchanged. 

This special case is enough to encompass the setting needed to apply~\cite{EMSGAFA} (and its~$S$-arithmetic generalisation~\cite{LemmaA}.)

The hypothesis~\eqref{RnD} is notably satisfied in the case where
\begin{equation}\label{centraliserless}
\text{$H^\prime$ is the centre of~$G$.}
\end{equation}

\subsection{Notations from Ratner theory}\label{secnotations}
%In order to state our result, we introduce some of the notations we will use later using Ratner's theory. 

A one parameter algebraic unipotent subgroup of~$G$ will mean a subgroup of the form~$U(\Q_S)$ where~$U$ is an algebraic subgroup of~$\G$ defined over~$\Q_S$ of dimension\footnote{The dimension for~$\Q_S$-algebraic groups is naturally a function on~$\mathbf{Spec}(\Q_S)$. The latter can be identified with~$S$. This is a constant function for an algebraic group which comes from~$\Q$. The group~$H$ is not assumed of constant dimension, as well as groups depending on it. It may furthermore contain an non-algebraic factor at the real place, at which place one considers the real analytic dimension.} at most one at every place in~$S$. 

For a subgroup~$L$ of~$G$ (resp. an algebraic subgroup~$L$ of~$\G$ defined over~$\Q$), we denote~$L^+$ the subgroup of~$L$ (resp. of~$L(\Q_S)$) generated by its one parameter algebraic unipotent subgroups defined over~$\Q_S$. (cf.~\cite[{\S}1.5 and Theorem~2.3.1]{Margulis}, \cite[{\S}6]{BorelTits} for interesting properties of such subgroups. See also~\cite{GilleKneserTits} for more actual questions. We recall some properties in Appendix~\ref{AppRatner}.)

We call the \emph{Ratner class}, the set~$\Rat_\Q$ of 
\begin{equation}\label{defRatclass}\text{algebraic subgroups~$L$ of~$\G$ over~$\Q$ such that~$L^+$ is $\Q$-Zariski dense in~$L$.}\end{equation}
This is the class~$\mathscr{F}$ of~\cite{TomanovOrbits}. Namely the property~\eqref{defRatclass} is equivalent to~$L$ being an algebraic subgroup of~$G$ over~$\Q$ such that any nontrivial quotient~$L/P$ defined over~$\Q$ contains a nontrivial unipotent element. Such groups are such that their reductive Levi factors are actually semisimple, and are of non compact type. They satisfy the Borel-Wang density Theorem:~$\Gamma\cap L$ is~$\Q_S$-Zariski dense in~$L$ (and a lattice).

We define a \emph{Ratner type} as a $\Gamma$-conjugacy class of groups of Ratner class. We denote~$[L]$
the Ratner type to which belongs a group~$L$ of Ratner class.

  For any~$L$ in~$\Rat_\Q$,
\begin{itemize}
\item we use the nonstandard notation~$\Lpp$ for the topological closure
\begin{equation}\Lpp=\overline{L^+\cdot(\Gamma\cap L(\Q_S))},\end{equation}
which is a subgroup of finite index in~$L(\Q_S)$ (see~Proposition~\ref{++finiteindex}) 
\item and let~$\mu_{\Lpp}$ be the~$\Lpp$-invariant probability measure on~$G/\Gamma$ supported over~$\Lpp\Gamma/\Gamma$.
\end{itemize}
Note that, though~$L^{+}$ does not depend on~$\Gamma$, the group~$\Lpp$ does. The probability~$\mu_{\Lpp}$ exists and is unique provided~$L$ belongs to~$\Rat_\Q$. The support~$\Lpp\Gamma/\Gamma$ of~$\mu_{\Lpp}$ is also the orbit closure~$\overline{L^{+}\Gamma/\Gamma}$.

Once a group~$L$ of Ratner class is fixed, consider its normaliser~$N(L)$ in~$\G$, and the action of~$N(L)$ on the Lie algebra of~$L$.  We will call \emph{unitary normaliser} of~$L$ in~$\G$ the algebraic subgroup~$N^1(L)$ of~$N(L)$ which is the kernel of the determinant of this action. We will write~$N=N^1(L)(\Q_S)$ for the group of it $\Q_S$-rational points, and~$\Gamma_N$ for~$\Gamma\cap N$. We shall benefit from the abbreviations
%\begin{equation}\label{defLF}
%M= {\Lpp}^\Omega:=\bigcap_{\omega\in\Omega} \omega\cdot \Lpp\cdot\omega^{-1}\text{ and }F:=\bigcap_{\omega\in\Omega} \omega\cdot N\cdot\omega^{-1}=\bigcap_{\omega\in H} \omega\cdot N\cdot\omega^{-1}.
%\end{equation}
\begin{equation}\label{defLF}
L^H:=\bigcap_{h\in H} h\cdot L\cdot h^{-1},~M:=\bigcap_{h\in H} h\cdot \Lpp\cdot h^{-1}~\text{ and }~F:=\bigcap_{h\in H} h\cdot N\cdot h^{-1}.
\end{equation}
The following easy observations might be useful while reading our main statement. The proof is left to the reader. (Recall~\(H^\prime\) denotes the centraliser of~\(H\) in~\(G\).)
\begin{lemma}\label{lemdebut}
\begin{enumerate}
\item The groups~$H$ and~$F$ both normalise~$M$. \label{lem1}
\item We have the identity~$N\cap H^\prime=F\cap H^\prime$. \label{lem2}
\end{enumerate}
\end{lemma}
\noindent We also note that~$F$ is a normal subgroup of~$HF$, and that~$M$ is a normal
subgroup of~$F$ and of~$HF$. We will see that~$M$ is open and of finite index in~$L^H$. Our proof will introduce the subquotient~$\widehat{G}=HF/M$ of~$G$.

\section{Statements}\label{secthm}
\subsection{Main Theorem with Comments}Here is our main theorem. We use the notation and setup of section~{\S}\ref{secintro}, and make direct reference to these notations inside the statement. Recall that~$o(1)$ stands for the class of sequences in~$G$ converging to the identity~$e$.

\begin{theorem}[Local focusing of translated measures]\label{Theorem} 
Let~$\G$ be a semisimple $\Q$-algebraic group, $S$ a finite set of places, $G=\G(\Q_S)$ (see~\eqref{notationQS}), $\Gamma$ an $S$-arithmetic lattice, $H$ a connected~$\Q_S$-subgroup reductive in~$G$ (as in~\eqref{H reductif dans}) and $\Omega$ a Zariski dense open bounded subset of~$H$. Consider the probability on~$\Omega$ which is the restriction of a Haar measure~$\mu$ on~$H$, and let~$\mu_\Omega$ be its direct image on~$G/\Gamma$ (as in~{\S}\ref{secmuomega}). Write~$H^\prime$ for the centraliser of~$H$ in~$G$.

 Consider a sequence~$\left(g_i\cdot\mu_\Omega\right)_{i\geq0}$ of translates, with non zero weak limit~$\mu_\infty$, and assume the sequence~$\left(g_i\right)_{i\geq0}$ in~$G$ satisfies the hypothesis~\eqref{Hypo} p.\pageref{Hypo}.%~\cite[Theorem~2.1, part~2. of conclusion]{LemmaA} (confer~\eqref{})

To~$\mu_\infty$ one can associate a Ratner type~$[L_0]$, depending only on~$\mu_\infty$, such that the following holds. We can decompose~$\left(g_i\right)_{i\geq0}$ into finitely many subsequences such that, if we substitute~$\left(g_i\right)_{i\geq0}$ with anyone of these subsequences, then there exists a group~$L$ in the Ratner type~$[L_0]$, and, writing~$N,F,{\Lpp},M$ as in~{\S}\ref{secnotations}  (see~\eqref{defLF}), there exist~$g_\infty$ in~$G$ and~$n_\infty$ in~$N$ such that
\begin{enumerate} 
\begin{subequations}
\item{\textbf{Limit measures:}} the limit measure can be written as
\begin{equation}\label{limitformula}
\mu_\infty = g_\infty\cdot\int_{\omega\in\Omega} \left(\omega\cdot n_\infty\cdot\mu_{\Lpp}\right)~~\mu(\omega),
\end{equation}

\item{\textbf{Focusing criterion:}} the sequence~$(g_i)_{i\geq 0}$ is of class
\begin{equation}\label{focusingcriterion}
O(1)\cdot (H^\prime\cap F)\cdot L^H,
\end{equation}
and there is a subsequence of~$(g_i)_{i\geq0}$ which is of class 
\begin{equation}\label{finefocusing} g_\infty\cdot o(1)\cdot n_\infty\cdot M.\end{equation}
\end{subequations}
\end{enumerate}

Moreover, we can choose~$n_\infty$ in the topological closure of~${\left(H^\prime\cap F\right)\cdot\Gamma_N}$.
\end{theorem}
Formula~\eqref{limitformula} can be read as follows. For any continuous function~$f:G/\Gamma\to\R$ with compact support (a posteriori, boundedness of~$f$ will suffice), its integral against~$\mu_\infty$ can be computed via the double integral
$$\int_{x\in G/\Gamma} f(x)~~\mu_\infty(x)=\int_{\omega\in\Omega} \left(\int_{y\in G/\Gamma} f\left({g_\infty}\cdot\omega\cdot n_\infty\cdot y\right)~~\mu_{\Lpp}(y)\right)~~\mu(\omega).
$$
Note that on the right hand side, the integral is in~\(G\) against $\mu|_\Omega$ and \emph{not} against~$\mu_\Omega$ (in~\(G/\Gamma\)). The latter would not make sense because~$L$ is not necessarily normalised by~$\Omega$.
 Let us note that in the context of~\cite{EMSAnn}, the fact that~$L$ is normalised by~$\Omega$ holds is a key technical fact in the proof; this is the subject of \cite{EMSCorr}. Our more general context and our method do not rely on this.

We also note as a consequence of formula~\eqref{limitformula} that~$\mu_\infty$ is a probability: it has mass one; there is no escape of mass
at infinity. Actually this follows from hypothesis~\eqref{Hypo} as proved in the generalisation~\cite{LemmaA} of~\cite{EMSGAFA}. The hypothesis~$\mu_\infty$ being non zero is even superfluous. With this hypothesis removed, our statement do
encompass the main result of~\cite{LemmaA}, and of~\cite{EMSGAFA} (we remind~{\S}\ref{RemarkcasGAFA}).

Formula~\eqref{finefocusing} is to be understood as a convergence in the quotient~$G/M$: the image sequence~$\left(g_i\cdot M\right)_{i\geq0}$ of cosets in~$G/M$ is such that 
$$\lim_{i\to\infty}g_i\cdot M \to g_\infty\cdot n_\infty\cdot M\text{~in~}G/M.$$

Let us warn that in~\eqref{limitformula}, neither~$g_\infty$, nor~$n_\infty$, nor even~$L$ are in general determined by~$\mu_\infty$. For example, for any $n\in H'\cap N$, one can replace the pair $(g_\infty,n_\infty)$ by $(g_\infty n,n^{-1}n_\infty)$. Nevertheless, the Ratner type~$[L_0]$ can be defined canonically in terms of~$\mu_\infty$, although implicitly (cf the begining of the proof, notably footnote~\ref{anticipating} at p.\pageref{anticipating}). %Furthermore, given~$\left(g_i\right)_{i\geq0}$, there is a finite subset of this type in which one can find~$L$.

The class~\eqref{finefocusing} describes the ``typical" behaviour of~$(g_i)_{i\geq0}$ in the following sense: nested inside any infinite subsequence~$(a_i)_{i\geq 0}$ of~$(g_i)_{i\geq0}$ one can find an infinite subsequence which is of class~\eqref{finefocusing}, although maybe for another triple~$(L,g_\infty,n_\infty)$ describing the same measure~$\mu_\infty$ via formula~\eqref{limitformula}. Indeed one can apply Theorem~\ref{Theorem} to the subsequence~$(a_i)_{i\geq 0}$ instead of~$(g_i)_{i\geq0}$.

%A word about convergence. We will use \emph{weak convergence} of measures. We understand by weak convergence the pointwise convergence as functionals on the space of continuous functions with compact support (which is actually a kind of~\emph{weak-$\star$} convergence from the viewpoint of functional analysis). One should note that, as far as we are interested with probability measures converging to probability measures, tight, vague, and weak convergence\footnote{Test functions are respectively: continuous and bounded (tight convergence); continuous and vanishing at infinity (vague convergence); continuous with compact support (weak convergence).} are equivalent. In above Theorem~\ref{Theorem}, the limit measure~\eqref{limitformula} is indeed a probability measure.

\subsection{Context}\label{seccontext}
We compare our statement to \cite{EMSAnn,EMSGAFA}. 

Theorem~1.7 and Corollary~1.13 in \cite{EMSAnn} concern limits of translates of a globally $H$-invariant probability measure supported on $H\Gamma/\Gamma$, in the case $\Q_S=\R$, such that~$H$ is defined over~$\Q$ and under the non-divergence hypothesis, i.e. the limit is a probability measure. The latter is intimately related to arithmetic stability~\eqref{ArS}.

Every application to counting integral points on varieties met in~\cite{EMSAnn} and the results in~\cite{EMSGAFA} assumes furthermore that $H'$ is defined over $\Q$ and is $\Q$-anisotropic, named~\eqref{RnD} in \S2.4. As noted there, we can assume hypothesis~\eqref{Hypo} in that case, and our statement generalises previous results: to the $S$-arithmetic setting; to a piece of measure $\mu_\Omega$ on $H\Gamma/\Gamma$; to a non-necessarily $\Q$-defined $H$ and finally to a non-necessarily closed orbit $H\Gamma/\Gamma$. (In order to recover statements about full closed orbits form statement from a piece of orbit we refer to the technique in~\cite{LemmaA}.)

However, without~\eqref{RnD}, the relation is not that simple. Hypothesis~\eqref{Hypo} is not vacuous. It implies non-divergence by \cite{LemmaA}, but not the other way around (consider a divergent sequence~$(g_i)_{i\geq0}$ in~$H^\prime\cap \Gamma$). A side effect of assuming the stronger analytic stability is that our focusing criterion is also stronger than the corresponding~\cite[Corollary~1.13]{EMSAnn}. Nonetheless, there are examples where~\eqref{Hypo} can not be satisfied, yet convergence to a probability measure occurs. Again, we insist though that our method does not require the additional assumptions of rationality of~$H^\prime$ and even closedness of the $H$-orbit where~$\mu_\Omega$ is supported.

We note also that~\cite[Theorem 1.7]{EMSAnn} concludes that $\mu_\infty$ is homogeneous, whereas our classification seems more complicated, being represented as an integral of homogeneous measures. However, when $H$ is defined over $\Q$, so is the group~$\bigcap_{\omega\in\Omega} \omega·L·\omega^{-1}$, and one should actually replace~$L$ with the $\Q$-Zariski closure of~$\bigcap_{\omega\in\Omega} \omega\cdot L^+\cdot\omega^{-1}$. The latter is normalised by~$\Omega$, hence by $H$, and we retrieve~\cite[Remark~4.1]{EMSAnn}.

As for the proof of Theorem~\ref{Theorem}, it is similar to~\cite{EMSAnn}, with one fundamental difference: the lack of rationality led us to introduce analytic stability to take further advantage of the linearisation technics.

\subsection{Complements}
\subsubsection{Rational case} Let us discuss some simplifications which occurs when~$H$ is rational over~$\Q$.
This case is arguably the one of importance in most applications to date.

$L=L^H$ is normalised by~$H$, $L$ $N$ et $F$ sont rationnels; $M=\Lpp$

The focusing criterion becomes $O(1) L$.

\subsubsection{Addenda from the proof: lifted convergence} The way our proof proceeds (see {\S}\ref{subsection-proof-structure}) allows us to state finer properties on the way~$(g_i\mu_\Omega)_{i\geq0}$
does converge to~$\mu_\infty$. Actually, writing
\[
\Gamma_N=\Gamma\cap N,
\]
all convergence (and the ergodic decomposition) occur at the higher~``level''~$G/\Gamma_N$.

Moreover in this statement we give a precise meaning to the property that equidistribution of
homogeneous sets is ``inner'': from the inside to the outside. In the typical case, one expect the 
support of measures in an equidistributing sequence to be eventually contained in the support
of the limit probability.  In the general case, one has to consider we may also be translating
by a convergent sequence: the support of the measures in an equidistributing sequence are
contained in the support of a closer and closer translate of the limit probability. Here, roughly speaking, 
the equidistribution phenomenon actually occurs inside right~$\Lpp$-orbits. Passing to 
the right quotient by~$\Lpp$, we are left with little dynamic: that of a bounded sequence of
translators.

\begin{proposition}[Addenda to~Theorem~\ref{Theorem}] Let us consider the situation of the Theorem~\ref{Theorem}, once~$(g_i)_{i\geq0}$ is  substituted with one of the finitely many considered subsequences. We can conclude moreover, 
\begin{enumerate}
\begin{subequations}
\item \label{stat1} denoting~$\tilde{\mu}_\Omega$ the direct image of~$\mu|_\Omega$ into~$G/\Gamma_N$, the sequence of translated measures~$g_i\cdot\tilde{\mu}_\Omega$ has limit
\begin{equation}\label{limitformulatilde1}
\widetilde{\mu}_\infty := g_\infty\cdot\int_{\omega\in\Omega} \left(\omega\cdot n_\infty\cdot\tilde{\mu}_{\Lpp}\right)~~\mu(\omega),
\end{equation}
in the space of probabilities on~$G/\Gamma_N$, where~$\tilde{\mu}_{\Lpp}$ is the left~$\Lpp$-invariant probability on~$\Lpp\Gamma_N/\Gamma_N$.

\item \label{stat2} the sequence~$(\mu_i)_{i\geq 0}$ of ``corresponding right-$\Lpp$ invariant" measures
\begin{equation}
\mu_i:= g_i\cdot\int_{\omega\in\Omega} \left(\omega \cdot\tilde{\mu}_{\Lpp}\right)~~\mu(\omega)
\end{equation}
has as a limit the measure~$\widetilde{\mu}_\infty$ from~\eqref{limitformulatilde1}.

\item \label{stat3} 	for every~$i$, the probabilities~$g_i\cdot\tilde{\mu}_\Omega$ and~$\mu_i$ have a same direct image~$\widehat{\mu}_i$ into~$G/(\Lpp\Gamma_N)$, and the sequence~$(\widehat{\mu}_i)_{i\geq 0}$ has limit
\begin{equation}\label{limitformulatilde2}
\widehat{\mu}_\infty = g_\infty\cdot\int_{\omega\in\Omega} \left(\omega\cdot n_\infty \cdot \delta_{\Lpp\Gamma_N}\right)~~\mu(\omega).
\end{equation}
\end{subequations}
\end{enumerate}
\end{proposition}
The more stringent statement is actually statement~\eqref{stat1}, which is proved at the end of the proof of Theorem~\ref{Theorem}. The statement~\eqref{stat3} is deduced by direct image into~$G/(\Lpp\Gamma_N)$. One can also deduce the identity~\eqref{limitformula} from statement~\eqref{stat3} by direct image, but into~$G/\Gamma$ instead. The statement~\eqref{stat2} is deduced from statement~\eqref{stat3} by the correspondence  between invariant measures
 and measures on the quotient\footnote{See \cite{Integration}[VIII-{\S}2 Prop.~4] about quotient measures. This correspondence is compatible with the vague topologies, by \emph{loc. cit}. We are concerned with probabilities to probabilities, so the
 convergence in the statement can be interpreted as vague convergence.}.

\subsubsection{Uniformity with respect to~$\Omega$}
We have the following strengthening of our main Theorem.
\begin{theorem} The situation is that of Theorem\ref{Theorem} and of its conclusions~\eqref{}. Notably, we consider a particular~$\Omega$ and have introduced a group~$L$ in the Ratner class~$\Rat$, and replaced~$(g_i)_{i\geq0}$ by a subsequence satisfying the focusing criterion~\eqref{finefocusing}.
\begin{itemize}
\item Suppose~$\Omega^\prime$ is another bounded non empty open subset of~$H$, which is included in~$\Omega$ and which \emph{is~$\Lpp$-saturated to the right inside~$\Omega$}, by which we mean the property
\[
\Omega^\prime=(\Omega^\prime\cdot \Lpp)\cap\Omega.
\]
We consider the normalised probability~$\frac{\mu(\Omega)}{\mu(\Omega^\prime)}\mu|_{\Omega^\prime}$ and denote~$\tilde{\mu}_{\Omega^\prime}$ its direct image on~$G/\Gamma_N$. Then the limit formula~\eqref{limitformulatilde1} holds with~$\Omega^\prime$ instead of~$\Omega$: the sequence of~$(g_i\cdot\tilde{\mu}_{\Omega^\prime})_{i\geq0}$ of probabilities on~$G/\Gamma_N$ converges to
\begin{equation}\label{limitformulaprime}
{\tilde{\mu}_\infty}^\prime = g_\infty\cdot\int_{\omega\in\Omega^\prime} \left(\omega\cdot n_\infty\cdot\tilde{\mu}_{\Lpp}\right)~~\mu(\omega).
\end{equation}
In particular the Theorem~\ref{Theorem} holds with~$\Omega$ replaced by~$\Omega^\prime$, for the same~$L$
 and the same subsequence of~$(g_i)_{i\geq0}$.
\item Let assume that the group~$L$ we are given by the Theorem~\ref{Theorem} is of minimal dimension among
all the groups in the Ratner class~$\Rat$ which occurs when applying Theorem~\ref{Theorem} to any infinite subsequence of~$(g_i)_{i\geq0}$ and any non empty bounded subset~$\Omega^\prime$ of~$H$. 
Then the Theorem~\ref{Theorem} holds with~$\Omega$ replaced by~$\Omega^\prime$, for the same~$L$
 and the same subsequence of~$(g_i)_{i\geq0}$.
\item Still under the previous minimality assumption about~$L$ we have the following. Let~$f$ be any ``density'' function~$f\in L^1(H,\mu)$. Write~$\tilde{\mu}_f$ for the direct image of~$f\cdot \mu$ on~$G/\Gamma_N$ then the sequence~$(g_i\cdot\tilde{\mu}_f)$ of probabilities on~$G/\Gamma_N$ converges to
 \begin{equation}\label{limitformuladensite}
{\tilde{\mu}_\infty}^\prime = g_\infty\cdot\int_{\omega\in H} \left(\omega\cdot n_\infty\cdot\tilde{\mu}_{\Lpp}\right)~~f\mu(\omega).
\end{equation}
\end{itemize}
\end{theorem}

We can proceed with the first point without further ado.
\begin{proof} Let~$\Omega^\prime$ be another bounded non empty open subset of~$H$ included in~$\Omega$ which is~$\Lpp$-saturated to the right inside~$\Omega$. Then we have the domination of measures
\[g_i\cdot\tilde{\mu_{\Omega^\prime}}\leq C\cdot g_i\cdot\tilde{\mu_{\Omega}}\] for every index~$i\geq 0$, with a uniform constant~$C=\mu(\Omega)/\mu(\Omega^\prime)$. As a consequence, the sequence~$(g_i\cdot\tilde{\mu_{\Omega^\prime}})_{i\geq 0}$ is tight, and any adherence value~${\mu_\infty}^\prime$
will satisfy the domination
\[{\mu_\infty}^\prime\leq C\cdot {\mu_\infty}.\]
For further reference, we note that, regarding supports of the measures, we deduce
\[\Supp({\mu_\infty}^\prime)\subseteq \Supp({\mu_\infty}).\]

We use now the saturation hypothesis...
\end{proof}

\begin{proposition}[On orbits which are closed (to each other)]
Consider two groups~$L$ and~$L^\prime$ in the Ratner class~$\Rat$, and two translated closed orbits~$O=g\cdot \Lpp\Gamma/\Gamma$
and~$O^\prime=g^\prime\cdot {\Lpp}^\prime\Gamma/\Gamma$ in~$G/\Gamma$.

\end{proposition}

\subsubsection{Reciprocal of the focusing criterion}
Once Theorem~\ref{Theorem} is proven, we can note, in the reciprocal direction, the following observation. This
details to what extent the focusing criterion is a sufficient condition to characterize the group~$L$ of Ratner
class involved in the limit formula~\eqref{limitformula} for~$\mu_\infty$.
 REDO
\begin{proposition} Let~$(g_i)_{i\geq 0}$ be a sequence of the form
\[g_i=b_i z_i l_i\text{ with }(b_i)_{i\geq 0}\in O(1),~z_i\in H^\prime\cap F,~l_i\in L^H,\]
for some subgroup~$L$ of Ratner class.

Then the support of~$g_i\mu_\Omega$ is contained in~$C\overline{\Omega} z_i\Lpp\Gamma/\Gamma$.
It diverges if and only~$z_i\Gamma_N$ does so. If~$z_i\Gamma_N$ converges, then so does ...
\end{proposition}
 
\begin{proposition}[Reciprocal] Let~$G,H,\Gamma$ and~$\Omega$ be as in the theorem. Let~$L$ be a group in the Ratner class~$\Rat$, and let~$(g_i)_{i\geq0}$ be a sequence in~$G$ satisfying~\eqref{focusingcriterion}. More precisely, let~$C$ be a compact subset of~$G$ such that
$$\left\{g_i\middle| i\geq0\right\}\subseteq C\cdot\bigcap_{\omega\in\Omega}\omega\cdot L^+\cdot\omega^{-1}.$$

\begin{enumerate}
\item \label{statement reciproque un}
Then~$C\cdot\overline{\Omega} L^+\Gamma/\Gamma $ is a closed subset of~$G/\Gamma$ which contains the support of each of ~$g_i\cdot\mu_\Omega$ for~${i\geq0}$.
\item
Assume now that the sequence~$(g_i)_{i\geq0}$ satisfies~hypothesis~\eqref{Hypo} p.\pageref{Hypo},
%\cite[Theorem~2.1, part~2. of conclusion]{LemmaA},
so that we can apply~Theorem~\ref{Theorem}.

Then any adherence value~$\mu_\infty$ of the sequence~$(g_i\cdot\mu_\Omega)_{i\geq0}$ is of the form
$$\mu_\infty=g_\infty\cdot\int_\Omega \left(\omega\cdot n_\infty\cdot\mu_{\tilde{L}^{++}}\right)~~\mu(\omega),$$
with some~$\tilde{L}$ in~$\Rat$ strictly included in~$L$.
\end{enumerate}
\end{proposition}
\begin{proof} The statement numbered~\eqref{statement reciproque un} on supports is merely an observation. It implies any weak limit~$\mu_\infty$ will have support in the same 
encompassing closed set. 

\end{proof}

\begin{corollary}[Algebraic focusing criterion] \label{Corollary} Consider a sequence~$\left(g_i\cdot\mu_\Omega\right)_{i\geq0}$ of translates, with non zero weak limit~$\mu_\infty$, and we assume the sequence~$\left(g_i\right)_{i\geq0}$ satisfies the theorem~1 of~\cite{LemmaA}.

Then we have
$$\mu_\infty=g_\infty\cdot\int_\Omega \left(\omega\cdot\mu_{L^+}\right)~~\mu(\omega) $$
if and only if the two conditions are satisfied
\begin{enumerate}
\item the necessary condition~\ref{focusingcriterion} of Theorem~\ref{Theorem},
\item no subsequence of~$\left(g_i\right)_{i\geq0}$ is of class~$O(1)\cdot\bigcap_{\omega\in\Omega}\omega\cdot \tilde{L}^+\cdot\omega^{-1}$ with some~$\tilde{L}$ in~$\Rat$ strictly included in~$L$.
\end{enumerate}
\end{corollary}

\subsubsection{On Eligible groups of Ratner class}
Let us proceed with studying, for a given~$\Omega$, which subgroups~$L$ of Ratner class are eligible for occurring in~\eqref{limitformula} for some
sequence of translators (depending on~$L$) for which Theorem~\ref{Theorem} applies. The following statement limits the class of eligible~$L$.

Recall~\eqref{defRatclass} that subgroup of Ratner class are generated over~$\Q$ by unipotent $\Q_s\text{-}$e\-le\-ments. Similarly, the subgroups~$L$ in the proposition below are generated over~$\Q$ by conjugacy $H$-orbits of unipotent $\Q_s$-elements. For reference in further works, we might call these subgroups~\emph{$H$-eligible} subgroups of Ratner class.
\begin{proposition} Let~$L$ be a group of Ratner class which actually occurs in Theorem~\ref{Theorem} for
some sequence of translators~$(g_i)_{i\geq0}$.
\begin{enumerate}
\item \label{propo1}  Then~$(L^H)^+$ is~$\Q$-Zariski dense in~$L$.
\item \label{propo2} Equivalently, $L$ is the~$\Q$-Zariski closure of an $H$-invariant~$\Q_S$-subgroup generated by unipotents.
\item \label{propo3} If~$H$ is defined over~$\Q$, this amounts to:~$L$ is normalised by~$H$.
\end{enumerate}
\end{proposition}
\begin{proof} Clearly~\eqref{propo1} implies~\eqref{propo2}. The latter is a consequence of property~\eqref{propo3} because~$L$ is
of Ratner class: it has no anisotropic factor and its radical is unipotent. If~$H$ is defined over~$\Q$ and it normalises a~$\Q_S$-subgroup, it
will normalise the~$\Q$-Zariski closure, hence~\eqref{propo1} gives~\eqref{propo3} in this case. Conversely~\eqref{propo3} gives~\eqref{propo1} by the defining properties of subgroups of Ratner class. Finally \eqref{propo2} implies~\eqref{propo1}
as, by construction,~$(L^H)^+$ is the biggest~$H$-invariant subgroup generated by unipotents.

We now prove why~\eqref{propo1} is implied by the hypothesis. Let~$L^\prime$ be the $\Q$-Zariski closure
of~$(L^H)^+$. Then, 
\end{proof}

\paragraph{}
A notable particular case is then when 
\begin{quote}
 every non zero Lie~$\Q_S$-subalgebra of~$\g$ generated by nilpotents is not contained in a proper~$\Q$-subalgebra.
\end{quote}
In such a case, the only eligible~$L$ are~$G$ and its $0$-dimensional subgroup reduced to the neutral element.
% (In particular this implies~$G$ is of non compact type.)
We deem the above criterion, stated in terms of linear rationality properties, shall be a manageable one.
We hope for instance it can prove, conversely to~\cite{These3}, that equidistribution properties have consequences on Galois image of~$\ell$-adic representations (a~$\Q$-Zariski weakening of Mumford-Tate conjecture).

%TODO ces deux sous-sections
\subsubsection{Remarks: generic pieces~$\Omega$} Let us give a few words on part~\eqref{finefocusing} of the focusing criterion.
In the corresponding part of the proof, we were reduced to consider a converging sequence of translates~$b_i\cdot \widehat{\mu}_\Omega$ 
of a measure with compact support~$\widehat{\mu}_\Omega$ by a bounded sequence~$(b_i)_{i\geq 0}$. What can be concluded
at this point may depend on~$\Omega$. In general, we know that~$(b_i)_{i\geq0}$ will be convergent modulo the stabiliser of~$\widehat{\mu}_\Omega$.
This stabiliser will always contain~$\Gamma\cap N \cap H^\prime$, which stabilises the image~$\widetilde{\mu}_\Omega$ of~$\mu|_\Omega$ in~$\G/\Gamma_N$.
But it is possible for specific~$\Omega$ that this stabiliser is bigger: the compact set~$\Omega$ may have a finite set of symmetries 
modulo~$\Gamma$; more drastically, we can take for~$\Omega$ a compact fundamental set for~$\Gamma\cap N\cap H$ in~$H$ (for a 
well chosen~$H$). In such a case, The whole group~$H$ will stabilise~$\widetilde{\mu}_\Omega$ and hence~$\widehat{\mu}_\Omega$.

Generically, say for a very small piece~$\Omega$ with no particularly symmetric shape, even when projected to the subquotients~$\widehat{G}$,
one expects that the stabiliser of~$\widehat{\mu}_\Omega$
will be reduced to~$\Gamma\cap N \cap H^\prime$, and we expect we can add to the conclusion of the focusing criterion in the
Theorem~\ref{Theorem} that the sequence
\[g_i(\Gamma\cap N\cap H^\prime)M\]
is convergent, and that the sequence 
\[g_iM\]
can be divided into finitely many converging subsequences.

%\subsubsection{Notable particular cases}
%
%\paragraph{The centraliser~$H^\prime$ reduces to the centre of~$G$}
%
%\paragraph{The centraliser~$H^\prime$ has a compact orbit modulo~$\Gamma$}
%
%\paragraph{The centraliser~$H^\prime$ has a closed orbit modulo~$\Gamma$}
%
%\paragraph{The group~$H$ admits $\Omega$ as a fundamental domain for~$\Gamma\cap H$}

\section{Proof of Theorem~\ref{Theorem}}\label{secproof}
We will now turn to the proof of~Theorem~\ref{Theorem}, after introducing some notations, mostly from the Dani-Margulis linearisation method of~\cite[\S\,{\S}2-3]{DM} and Ratner theory.

Firstly for convenience,
 \begin{subequations}
\begin{equation}\text{let $O(1)$ denote the class of bounded sequences in~$G$, and}\end{equation}
\begin{equation}\text{let $o(1)$ be the class of sequences converging to the neutral element~$e$ of~$G$.}\end{equation}
\end{subequations}
and let us write
\begin{subequations}
\begin{align}
{}^x{A}&:=x·A·x^{-1}\text{ for a left conjugation,}\\
{A}^x&:=x^{-1}·A·x\text{ for right conjugation,}\\
\text{ and~}A^\Omega&:=\bigcap_{\omega\in\Omega} \omega·A·\omega^{-1}.
\end{align}
\end{subequations}
For instance,~$M={\Lpp}^H$ is the group involved in Theorem~\ref{Theorem}.

\subsection{Notations and Singular sets from Linearisation and from Ratner's theory}\label{RatnerNotations}
Recall from~\ref{secnotations} that a \emph{Ratner type} was defined to be the~$\Gamma$-conjugacy class~$[L]$ of a subgroup~$L$ in the Ratner class~$\Rat_\Q$. Let~$W$ be a subgroup of~$G$ such that~$W^+=W$. Dani-Margulis define the subset
\begin{subequations}
\begin{equation}\label{X1}
X(L,W) = \Cji (W,L) = \left\{g\in G\middle| g^{-1}Wg \subseteq L \right\}\subseteq G
\end{equation}
of inverse of elements in~$G$ that conjugate~$W$ into~$L$. If~$[L]$ is the Ratner type of~$L$, we denote
\begin{equation}\label{X2}
X([L],W) = X(L,W)\Gamma/\Gamma \subseteq G/\Gamma
\end{equation}
the \emph{singular locus of type~$[L]$ for the action of~$W$ on~$G/\Gamma$}. This locus depends indeed on~$L$ only via~$[L]$, by 
\begin{equation}\label{gammaX} X(L^\gamma ,W)=X(L ,W)\cdot\gamma,\end{equation}
Denote
\end{subequations}
\begin{subequations}
\begin{equation}\label{X3}
X^*(L,W) = X(L,W) \smallsetminus \bigcup_{\tilde{L}\in \Rat_\Q, \dim(\tilde{L})<\dim(L)} X\left(\tilde{L},W\right).
\end{equation}
Note that this last union is right-invariant under~$\Gamma$\footnote{\label{rightgammainv}	The set~$\Rat_\Q$  is invariant under the action of~$\Gamma$ by conjugation. Conjugation preserve the dimension. We conclude by~\eqref{gammaX}. } .
We then denote
\begin{equation}\label{X4}
X^*([L],W)=X^*(L,W)\Gamma/\Gamma=X([L],W) \smallsetminus \bigcup_{\tilde{L}\in \Rat, \dim(\tilde{L})<\dim(L)} X\left([\tilde{L}],W\right)
\end{equation} 
\end{subequations}
the \emph{exclusive} singular locus of type~$[L]$.

Then~$X([L],W)$ is the set of points in~$G/\Gamma$ belonging to the support of a~$W$-invariant translates of~$\mu_{\Lpp}$. The set~$X(L,W)$ is the set of elements~$g$ in~$G$ translating~$\mu_{\Lpp}$ into a~$W$-invariant measure~$g\cdot\mu_{\Lpp}$. Using Ratner's theorem about orbits closure (\cite{Ratnerp}, used only in the context of its alternative treatment in~\cite[Theorem~2]{MargulisTomanov}, in which context we also use the complement~\cite[Theorems~1 and~2, cf. comments between Thms~2 and~3]{TomanovOrbits}), one knows that~$X^*([L],W)$ is the set of points~$x$ in~$G/\Gamma$ such that the topological closure~$\overline{W·x}$ of the~$W$-orbit at~$x$ is of the form~$g\Lpp\Gamma/\Gamma$ for some~$g$ in~$G$. The no self-intersection statement~\cite[Proposition~3.3]{DM} then says that~$X^*([L],W)$ is the set of points in~$G/\Gamma$ belonging to the support of \emph{exactly one}~$W$-invariant translate of~$\mu_{\Lpp}$. 

%The set~$X(L,W)\cdot\mu_{L^+}$ is the set of translates of~$\mu_{L^+}$ which are~$W$-invariant, and 

\begin{subequations}
For any~$W$-invariant probability measure~$\mu_\infty$, we denote
\begin{equation}\label{notationrestriction}{\mu_\infty}^{[L]}\text{ the restriction~}\mu_\infty|_{X^*([L],W)}.\end{equation}
By Ratner classification theorem (and some precision from~\cite{TomanovOrbits} in the~$S$-adic case), the measure~${\mu_\infty}^{[L]}$ is made of the $W$-ergodic components of~$\mu_\infty$ that are translates of~$\mu_{\Lpp}$.

Observe that the subsets~$X^*([L],W)$, for a fixed~$W$, but where~$[L]$ ranges through the Ratner types, induce a countable (confer~\cite[Proposition~2.1]{DM}) partition of~$G/\Gamma$ into~$W$-invariant subsets. Correspondingly, one has a countable decomposition, for any~$W$-invariant probability measure~$\mu_\infty$,
\begin{equation}\label{Ratnerdecomposition} \mu_\infty=\sum_{[L]}{\mu_\infty}^{[L]}.\end{equation}
\end{subequations}

\subsection{Structure of the proof}\label{subsection-proof-structure} The main strategy is classical in the context of linearisation methods.
The key new inputs are the two stability results in linearised dynamics from~\cite{Lemmanew}, at the expense of hypothesis~(\ref{Hypo}).
\footnote{The first named author strongly advocate a rapprochement between Homogenous dynamics, especially linearisation, and arithmetic stability, in particular in Arakelov geometry context (\cite{Seshadri,Burnol,ACLT,Zhang,Bost}).}

The proof will begin, until~\eqref{eq0}, with introducing a suitable~$W$ and an application of Dani-Margulis linearisation method (Appendix~\ref{AppA}), as used in~\cite{EMSGAFA} Prop.~3.13. Then a simplification will occur as, from~\eqref{eq1} to~\eqref{eq5}, we will apply Theorem~1 of~\cite{Lemma,Lemmanew}. It will allow us to apply Theorem~2 of~\cite{Lemmanew}, and when interpreted in the linearisation setting, in~\eqref{keyvariant},~\eqref{keypoint}, and~\eqref{purity}, this allows us to lift the measures in~\eqref{lift} to~$G/\Gamma_N$ and initiate an induction process in~\eqref{induction}. By maximality of~$W$, we will, at last, be able to conclude.

As a guide through our induction, we assemble in a diagram the arrows through which our proof will proceed.
\[ 
\begin{tikzcd}[row sep=large, column sep=small]
G\arrow{dr}\arrow{rr}	&   				&	G/\Gamma_N\arrow{dl}\arrow{dr}	&\\
								& G/\Gamma	&											  		& 
								G/(\Lpp\Gamma_N) & X\arrow[hookrightarrow,swap]{l}\arrow[loop right]{}
															&\widehat{G}
\end{tikzcd}
\]
We will end up working inside~$X=\overline{HN/(\Lpp\Gamma_N)}$ on which~$\widehat{G}=HF/M$ acts 
(see~\eqref{Ghat} and~\eqref{defX}).

Working in our general setting leads us to regularly cumulate technical tools. We also experienced that the process of the proof run on quite a length
and that this flow allow very little change: steps insist to occur in some succession. We tried to evade most technicalities from the main proof
and separate them in a postponed section. A treatment of standard linearisation statements, adapted to our setting has been gathered in Appendix~\ref{AppA}.
We encourage the reader to focus of the main part of the proof, in~{\S}\ref{mainproof} below, and first consider a simpler setting, as for instance the one in~\ref{Theointro}, where most of the postponed technicalities can be ignored (and for example where Appendix~\ref{AppA} reduces to results well known to experts) for quite some time now).

\subsection{Proof of Theorem~\ref{Theorem} -- Main part of the proof}\label{mainproof}
\begin{proof}[Proof of Theorem~\ref{Theorem}]
%The beginning of our proof will require to work with $\Omega$ small enough in order to apply Proposition~\ref{AProp313}. To distinguish this initial choice, we denote it $\Omega_0$. 
 %TODO ref proposition
 
Let~$W=\left(\Stab_G(\mu_\infty)\right)^+$,  the group generated by the $\Q_S$-algebraic unipotent subgroups of~$G$ that stabilise~$\mu_\infty$. The group~$W$ is normal in~$\Stab_G(\mu_\infty)$, and by Corollary~\ref{corliftuni}  no quotient of~$\Stab_G(\mu_\infty)/W$ contains a $\Q_S$-algebraic unipotent subgroup. \label{defW}%Furthermore, it is possible to assume that $W$ is nontrivial. This can be argued as \cite[Proposition~2.2]{EMSAnn}, but we prefer to give an alternative argument requiring only hypothesis~\eqref{Hypo}, as it fits more succinctly in our context and is used again at the end of the proof.

\subsubsection{Ratner decomposition}
By hypothesis, the weak limit~$\mu_\infty$ is non zero. From Ratner decomposition~\eqref{Ratnerdecomposition}, there must be at least one group~$L_0$ in the Ratner class~$\Rat_\Q$ such that
\begin{equation*}{\mu_\infty}^{[L_0]}\neq 0.\label{Rdec}\end{equation*}
Let us choose\footnote{ \label{anticipating}
Anticipating on the proof, we will see that the Ratner decomposition will actually involve a unique non zero term (a unique Ratner type). 
Anticipating furthermore, we will actually find the group~$L$ fulfilling the statement among the conjugates of~$L_0$ by an element of~$\Gamma$.
} a group~$L_0$ of minimal dimension such that~${\mu_\infty}^{[L_0]}\neq 0$.
The minimality of~$L_0$ gives the identity (see notation~\eqref{X4})
\begin{equation}\label{eqmini}
{\mu_\infty}|_{X([L_0],W)}={\mu_\infty}|_{X^*([L_0],W)}.
\end{equation}

\subsubsection{Linearisation method and focusing}\label{subsection linearisation}
We now apply the linearisation method, and more precisely \cite[Prop.~3.13]{EMSAnn}. We first introduce the setting of~\cite[Section~3, paragraph~3]{DM}, adapted to our $S$-arithmetic setting.
\begin{itemize}
\item  Let~$V=\bigwedge^{\dim(L_0)}\g$, as a $\Q_S$-linear representation of~$G$: it is an exterior power of the adjoint representation.
The ``linearisation” method involves the study of dynamics of~$G$ on~$G/\Gamma$ by working inside this kind of linear representation.
\item Denote~$p_{L_0}$ for a generator of the $\Q_S$-line~$\det(\l_0):=\Lambda^{\dim(L_0)}\l_0$ in~$V$ (by a~\emph{$\Q_S$-line}, we mean a free~$\Q_S$-module of rank~$1$.)
\item We write
\[
A_{L_0}:=\left< X(L_0,W)\cdot p_{L_0}\right>\]
 the $\Q_S$-sub\-mo\-du\-le generated by~$X(L_0,W)\cdot p_{L_0}$.
\end{itemize} 
Let~$\w$ denote the Lie algebra of~$W$ and define the sub\-module
\[ V_W := \left\{v\in V~\middle|~\forall\, w\in\w,\ v\wedge w={0}\text{ inside }{\bigwedge}^\bullet\g \right\}.\]
We observe (cf.~\cite[Prop.~3.2]{DM}) that~$X(L_0,W)\cdot p_{L_0}$ is the intersection of~$V_W$ with the orbit~$G\cdot p_{L_0}$. Note that~$G\cdot p_{L_0}$ depends on~$L_0$ only up to conjugation. Hence, so do both~$G\cdot p_{L_0}\cap V_W$ and the linear space~$A_{L_0}$ it generates.

%Let $\Omega_0$ be a neighbourhood of the identity in $H$ contained in $\Omega$ and small enough so that~Proposition~\ref{AProp313} (or~\cite[Prop. 3.13]{EMSAnn} in the archimedean case) applies:
By~Proposition~\ref{AProp313} (or~\cite[Prop. 3.13]{EMSAnn} in the archimedean case), there exists a compact subset~$D$ of~$A_{L_0}$, and a sequence~$(\gamma_i)_{i\geq 0}$ in~$\Gamma$, such that, for any neighbourhood~$\Phi$ of~$D$ in~$V$, one has\footnote{The quantifier~$\forall i \gg 0$ being the usual substitute for the pair of quantifiers~$\exists i_0,\forall i > i_0$.}
\begin{equation} \label{eq0}
\forall i \gg 0,~g_i\cdot\Omega\cdot\gamma_i\cdot p_{L_0}\subset \Phi.
\end{equation}

In other words, the sequence~$\left( g_i·\Omega·\gamma_i·p_{L_0}\right)_{i\geq 0}$ of the subsets of~$V$ is uniformly bounded, and any limit point, for the Hausdorff topology, is contained in~$D$.

One knows that the orbit~$\Gamma\cdot p_{L_0}$ is discrete inside~$V$ (cf.~\cite[Theorem~3.4.]{DM}; this is immediate here for an arithmetic lattice). Choose, for each place~$s$ in~$S$, a~$s$-adic norm~$\Nm{-}_s:V\tens_{\Q_S}\Q_s\to\R_{\geq0}$, and introduce the product norm~$\Nm{-}:V\to\R_{\geq0}, (x_s)\mapsto\max_{s\in S} \Nm{x_s}_s$. We recall that this norm is a proper map. Hence the previous discreteness property translates into
\begin{equation}\label{eq1}
\lim_{x\in\Gamma\cdot p_{L_0}} \Nm{x} =+\infty,
\end{equation}
(confer~\cite[(2.6), {\S}8.1]{KT} for such a setting.)

\subsubsection{Linearised nondivergence and consequences} For any subset~$A$ of~$V$, we abbreviate~$\Nm{A}=\sup_{a\in A}\Nm{a}$.
We now use hypothesis~\eqref{Hypo} (see~\cite[Theorem~2.1 under the form~Theorem~\ref{thm21bis}]{LemmaA}) about the~$g_i$.  In particular
\begin{equation}\label{eq2}
\exists c>0, \forall \gamma \in G, \forall i\geq 0, \Nm{g_i\Omega\gamma p_{L_0}}\geq c·\Nm{\gamma·p_{L_0}}.
\end{equation}
Combining~\eqref{eq1} with~\eqref{eq2} results in~
$$\lim_{x\in\Gamma·p_{L_0}}\inf_i \Nm{g_i\Omega\gamma x} \geq c\cdot \lim_{x\in\Gamma·p_{L_0}} \Nm{x} =+\infty$$

For any bound~$\Lambda<+\infty$, the set 
\begin{equation}\label{eq3}
\left\{\gamma·p_{L_0}\in\Gamma p_{L_0}
~\middle|~\exists i, \Nm{g_i·\Omega\gamma·p_{L_0}}\leq\Lambda\right\}
\end{equation}
is finite. If~$\Lambda>\Nm{D}$, there exists a neighbourhood~$\Phi$ of~$D$ such that~$\Nm{\Phi}\leq\Lambda$. Fix such a~$\Phi$. Applying~\eqref{eq0} to this~$\Phi$, we get that, for~$i\gg0$, the element~$\gamma_i p_{L_0}$ belongs to the finite set~\eqref{eq3}. Consequently, the set~$\left\{\gamma_ip_{L_0}~\middle|~i\geq0\right\}$ described by the sequence~$(\gamma_ip_{L_0})_{i\geq0}$ is finite.

%Consequently, for any subset~$\Phi$ of~$V$ such that~$\Nm{\Phi}\leq\Lambda$, 
%$$\left\{\gamma_i·p_{L_0}~\middle|~i\in\Z_{\geq0},~g_i·\Omega\gamma_i·p_{L_0}\subseteq\Phi \right\},$$
%which is a subset of~\eqref{eq3}, is finite. Take~$\Lambda>\Nm{D}$, so that the set of~$\Phi$ such~$\Nm{\Phi}\leq\Lambda$ contains a basis of neighbourhoods of~$D$. Note that the finite set~\eqref{eq3} is independant of such~$\Phi$.

%
%Consequently,
%\begin{equation}\label{eq3}
%\left\{\gamma·p_{L_0}\in\Gamma·p_{L_0} \middle| \exists i, \Nm{g_i·\Omega\gamma·p_{L_0}}\leq\Nm{\Phi} \right\}
%\end{equation}
%is finite as soon as~$\Phi$ is bounded. Even more: choose a bound~$\Lambda$. Then
%$$\left\{\gamma_i·p_{L_0}\in\Gamma·p_{L_0} \middle|~i\geq0,~g_i·\Omega\gamma_i·p_{L_0}\subseteq\Phi \right\}$$
%is contained in a finite subset of~$\Gamma\cdot p_{L_0}$, independent of~$\Phi$, provided~$\Nm{\Phi}$ is bounded by~$\Lambda$. Take~$\Lambda>\Nm{D}$, so that the set of such~$\Phi$ contain a basis of neighbourhood of~$D$. 

%Using~$\eqref{eq0}$ for such~$\Phi$, we deduce that only finitely many~$\gamma_i·p_{L_0}$ are involved in~\eqref{eq0}.

Let us decompose~$(g_i)_{i\geq0}$ into finitely many subsequences according to the value of~$\gamma_i·p_{L_0}$. Passing to anyone of these subsequences, we can assume~$(\gamma_i·p_{L_0})_{i\geq0}$ is constant. We set~$L={}^{\gamma_0}L_0:={\gamma_0}L_0{\gamma_0}^{-1}$. It belongs to the Ratner type of~$L_0$, and is such that~$p_L=\gamma_0·p_{L_0}=\gamma_i·p_{L_0}$ for all~$i\geq 0$. We will show that this group~$L$ is the one involved in the statement of the theorem.

We now have deduced the following from~\eqref{eq0}. For any neighbourhood~$\Phi$  of~$D$,
\begin{equation}\label{eq5}
 \forall i\gg0, g_i\cdot\Omega\cdot p_L \subseteq \Phi.
\end{equation}
For any such~$\Phi$, the translated subsets~$g_i·\Omega·p_L$ are bounded by~$\Nm{\Phi}$, independently of~$i$. In particular, we can choose a bounded~$\Phi$, and the sequence~$\left( g_i·\Omega·p_L\right)_{i\geq0}$ is uniformly bounded.

\subsubsection{Linearised focusing and consequences}\label{section-lin-focusing} Recall the definitions~\eqref{defLF}.
The stabiliser of~$p_L$ is~$N$, and the fixator\footnote{We name \emph{fixator} the pointwise stabiliser.} of~$\Omega·p_L$ is~$F$.
Applying \cite[Theorem~2]{Lemmanew}, we deduce the following important ingredient of our proof
$$\text{the sequence~}\left( g_i\right)_{i\geq 0}\text{ is of class }O(1)·F.$$

%Theoreme invariant par O(1) ? contexte de la preuve
Without loss of generality, translating by a bounded sequence, we may assume that~$g_i$ belongs to~$F$.

We observe that each conjugate~${g_i}^\omega$ belongs to~$N$. We have, by definition of~$F$,
\begin{equation}\label{stabilitybygi}
g_i·\omega·N=\omega·{g_i}^\omega·N=\omega N.
\end{equation}
This holds for any~$\omega$ in~$\Omega$; hence for any~$\omega$ in the $\Q_S$-Zariski closure of~$\Omega$; in particular for any~$\omega$ in the topological closure~$\overline{\Omega}$.
It follows~$g_i·\overline{\Omega}·N=\overline{\Omega}·N$, and~$g_i\overline{\Omega}·p_L=\overline{\Omega}·p_L$.

%But recall that for any~$\Phi$ in some basis of neighbourhood of~$D$,
%$$\forall i\gg0, g_i\cdot\Omega\cdot p_L \subseteq \Phi.$$
Consequently, in~\eqref{eq5}, we have actually~$g_i\cdot\overline{\Omega}\cdot p_L=\overline{\Omega}\cdot p_L\subseteq \overline{\Phi}$. As we can vary~$\Phi$ through arbitrary small neighbourhoods of~$D$,
\begin{equation}\label{keyvariant}
\overline{\Omega}\cdot p_L\subseteq D.
\end{equation}
 Hence~$\overline{\Omega}\cdot p_L$ is in~$\left(G\cdot p_L\right)\cap V_W$. In other words,
\begin{equation}\label{keypoint}
\text{$\overline{\Omega}\cdot N$ is contained in~$X(L,W)$.}
\end{equation}
It follows that~$\overline{\Omega}\cdot N\Gamma/\Gamma$ is contained in~$X([L],W)$. 

The set~$N\Gamma/\Gamma$ is closed in~$G/\Gamma$, as a consequence of~$\Gamma·p_L$ being closed (actually discrete) in~$V$ (confer~\cite[Argument in~{\S}8.1]{BorelIntro}). It follows, as~$\overline{\Omega}$ is compact, that~$\overline{\Omega}N\Gamma/\Gamma$ is closed inside~$G/\Gamma$.  But~$\overline{\Omega}N\Gamma/\Gamma$ contains the support of~$\mu_\Omega$. From~\eqref{stabilitybygi}, it is also stable under the~$g_i$, and will contain the support of each of the~$g_i\mu_\Omega$. Finally, it contains the support of~$\mu_\infty$, and the support of~${\mu_\infty}^{[L]}$. 

Hence the support of~$\mu_\infty$ is contained in~$X([L],W)$. Consequently (recall~\eqref{notationrestriction} and~\eqref{eqmini}),
\begin{equation}\label{purity}
\mu_\infty=\mu_\infty|_{X([L],W)}={\mu_\infty}^{[L]}={\mu_\infty}^{[L]}|_{X^*([L],W)}.
\end{equation}
We have proved that Ratner decomposition~\eqref{Ratnerdecomposition} of~$\mu_\infty$ has a unique Ratner type!

We will henceforth drop the exponent~${}^{[L]}$ and write~$\mu_\infty$ instead of~${\mu_\infty}^{[L]}$.

We know that the~$W$-ergodic components of~${\mu_\infty}$ are all translates~$c\cdot\mu_{\Lpp}$. 
If~$c\cdot\mu_{\Lpp}$ is one of the~$W$-ergodic components of~$\mu_\infty$, we necessarily have
\begin{itemize}
\item firstly~$c\Gamma\in\overline{\Omega}N\Gamma/\Gamma$, as the support of~$\mu_\infty$ is contained in~$\overline{\Omega}N\Gamma/\Gamma$;
\item secondly,~$c\in X^*(L,W)$, because~$W$ acts ergodically on~$c\cdot\mu_{\Lpp}$.
\end{itemize}
In other words
\begin{equation}\label{cbelongs}
\text{
the element~$c$ belongs to the intersection~$\overline{\Omega}N\Gamma\cap X^*(L,W).$
}
\end{equation}

\subsubsection{Linearisation and self-intersection}
We now use~\cite[Prop.~3.3, Corollary~3.5]{DM}\footnote{The role played by~$H$ in~\cite[pp.99-100]{DM} is played by~$L$ here. The Ratner class~$\Rat_\Q$ here is a substitute for~$\mathscr{H}$ in~\cite[pp.99-100]{DM}. The~$\Gamma_N$ we are about to define is denoted~$\Gamma_L$ in~\cite[pp.99-100]{DM}.}.
Let~$\Gamma_N=\Gamma\cap N$.
After~\cite[Prop.~3.3]{DM}, the set~$X^*(L,W)$ ``does not have~$(L,\Gamma_N)$-self-intersection" in the sense of~\cite[{\S}3, p.100]{DM}. Namely 
\begin{subequations}
\begin{equation}\label{self1}
\forall x\in X^*(L,W), \forall \gamma\in\Gamma, x\gamma\in X^*(L,W)\Rightarrow \gamma\in\Gamma_N.
\end{equation}
Equivalently the quotient map~$\varphi:G/\Gamma_N\to G/\Gamma$ induces a bijection\footnote{A similar geometric interpretation of~\cite[p.100]{DM} ``self-intersetion" property can be found at the end of \cite[Corollary~3.5]{DM} "The quotient map~[\ldots] is injective. that is~[\ldots]".}
\begin{equation}\label{self2}
\left.X^*(L,W)\,\Gamma_N\middle/\Gamma_N \right.
\xrightarrow{}
 \left.X^*(L,W)\,\Gamma\middle/\Gamma\right. .
\end{equation}
\end{subequations}

Denote
\begin{subequations}
\begin{equation}\label{starnotation}\left(\overline{\Omega}N\Gamma\right)^*:=X^*(L,W)\cap\left(\overline{\Omega}N\Gamma\right).\end{equation}
It can be considered as a lifting from~$G/\Gamma$ to~$G$ of
\begin{equation}\label{eqa}\left.\left(\overline{\Omega}N\Gamma\right)^*\Gamma\middle/\Gamma\right.=X^*([L],W)\cap\left(\left.\overline{\Omega}N\Gamma\middle/\Gamma\right.\right).\end{equation}
\end{subequations}
On the right hand side of equality~\eqref{eqa}, the first factor is of full measure for~${\mu_\infty}^{[L]}$, and the second factor contains the support of~${\mu_\infty}^{[L]}$. As a consequence,~$\left.\left(\overline{\Omega}N\Gamma\right)^*\Gamma\middle/\Gamma\right.$ is of full measure for~${\mu_\infty}^{[L]}$.

Recall that~$\overline{\Omega}N$ is contained in~$X(L,W)$, by~\eqref{keypoint}. Therefore, using~\eqref{self1},
\begin{equation}\label{omegaself1}
\left(\overline{\Omega}N\right)^*:=X^*(L,W)\cap\left(\overline{\Omega}N\right)
\end{equation}
is such that
\begin{equation}\label{omegaself2}
\left(\overline{\Omega}N\right)^*\Gamma\cap X^*(L,W)=\left(\overline{\Omega}N\right)^*\Gamma_N\cap X^*(L,W).
\end{equation}
Let us prove that
\begin{equation}\label{gammaoops}
\left(\overline{\Omega}N\right)^*=\left(\overline{\Omega}N\Gamma\right)^*.
\end{equation}
In details:
\begin{proof} We prove the identity by double inclusion. The inclusion~$\left(\overline{\Omega}N\right)^*\subseteq\left(\overline{\Omega}N\Gamma\right)^*$ is immediate. Let us prove the reverse inclusion.

Let~$x$ be in~$\left(\overline{\Omega}N\Gamma\right)^*$. We can write~$x=y\gamma$ with~$y\in\overline{\Omega}N$ and~$\gamma\in\Gamma$. If~$y$ belongs to~$\left(\overline{\Omega}N\right)^*$, then~\eqref{omegaself2} implies~$\gamma\in\Gamma_N$, and~$x$ would belong to~$y\Gamma_N\subseteq yN\in \overline{\Omega}N N=\overline{\Omega}N$ and we would be done.

What if~$y$ belongs to~$\overline{\Omega}N$ but not to~$X^*(L,W)$. It then belongs to~$X(L,W)\smallsetminus X^*(L,W)$, by~\eqref{keypoint}. But the right hand side of~\eqref{X3} is right-invariant under~$\Gamma$, by footnote~\ref{rightgammainv}. So~$x=y\gamma$ can't belong to~$X^*(L,W)$. This case doesn't occur.
\end{proof}

\subsubsection{About~$W$ and left-invariance of~$\mu_\infty$}\label{subsubsection-W}
Recall~\eqref{cbelongs}, that the~$W$-ergodic components of~$\mu_\infty$ are of the form~$c\cdot\mu_{\Lpp}$ with~$c\in \left(\overline{\Omega}N\Gamma\right)^*$ (cf~{\S}\ref{RatnerNotations}). By~\eqref{gammaoops}, one gets~$c\in \left(\overline{\Omega}N\right)^*\subseteq\overline{\Omega}N$. 
Write~$c=hn$ with~$h\in H$ and~$n\in N$. Then~$c\cdot\mu_{\Lpp}$ is left invariant under~$hn\Lpp n^{-1}h^{-1}$. Consequently~$\mu_\infty$ is left invariant under
$$ M:= \bigcap_{h\in H} h{\Lpp}h^{-1}.$$
In particular,~$W$ contains~$M^+$, by maximality of~$W$.

Recall that~$\overline{\Omega}N$ is contained in~$X(L,W)$. Namely,~$Wx\subseteq xL^+$ for any $x$ in~$\overline{\Omega}N$. It follows that~$W$ acts trivially on the quotient~$\overline{\Omega}N/L^+$. Equivalently,
$$W\subseteq 
\bigcap_{\omega \in \overline{\Omega}N} {}^\omega{L^{+}}
\subseteq \bigcap_{\omega \in \overline{\Omega}} {}^\omega{L^{+}}
 \subseteq \bigcap_{\omega\in\overline{\Omega}} {}^\omega{\Lpp} = {\Lpp}^\Omega.$$
In particular, this give the reverse inclusion of the identity~$W=\left({\Lpp}^\Omega\right)^+=M^+$.

\subsubsection{Lifting along~$G/\Gamma_N\xrightarrow{\varphi}G/\Gamma$ of the measures}
Finally,~\eqref{self2} induces a bijection
\begin{equation}\label{finalbij}
\left.\left(\overline{\Omega}N\right)^*\Gamma_N\middle/\Gamma_N\right.
=
\left.\left(\overline{\Omega}N\Gamma\right)^*\Gamma_N\middle/\Gamma_N\right.
\xrightarrow{\psi}
\left.\left(\overline{\Omega}N\Gamma\right)^*\Gamma\middle/\Gamma\right.
\end{equation}
We will lift the limit measure~$\mu_\infty$ and the sequence~$g_i\cdot\mu_\Omega$, and show that we have still
convergence of measures and an ergodic decomposition.

Recall that the target space is of full measure for~$\mu_\infty$. The map~$\varphi$ is a local homeomorphism, and hence an open map, and hence maps a Borel subset to a Borel subset. Its restriction to~$\left.\left(\overline{\Omega}N\right)^*\Gamma_N\middle/\Gamma_N\right.$, which is easily checked to be borelian, will also map a Borel subset to a Borel subset. In particular, the inverse bijection~$\psi^{-1}$ is a Borel measurable map. Let us lift the measure~$\mu_\infty$ along the bijection~\eqref{finalbij}. It results in a measure
\begin{equation}\label{lift}
{\tilde{\mu}_\infty}:=
{\psi^{-1}}_\star\left({\mu_\infty}\right)
%|_{ \left.\left(\overline{\Omega}N\right)^*\Gamma_N\middle/\Gamma_N\right. }
\end{equation} 
whose support is contained in the closed set~$\left.\overline{\Omega}N\Gamma_N\middle/\Gamma_N\right.=\left.\overline{\Omega}N\middle/\Gamma_N\right.$, and whose direct image satisfies
$${\varphi}_\star\left({{\tilde{\mu}}_\infty}\right)={\mu_\infty}.$$
There is two notable consequences:~$\mu_\infty$ is determined by its lifting~${\tilde{\mu}_\infty}$; any left translation by an element~$g$ of~$G$ that leaves~${\tilde{\mu}_\infty}$ invariant also leaves~$\mu_\infty$ invariant. 

We will furthermore lift the sequence~$g_i\cdot\mu_\Omega$. Mimicking the definition of~$\mu_\Omega$, we define~$\tilde{\mu}_\Omega$ to be the direct image in~$G/\Gamma_N$ (instead of~$G/\Gamma$) of the probability on~$\Omega$ that is a restriction a Haar measure on~$H$. We can then consider the sequence
\begin{equation}\label{naivelift}
\left( g_i\cdot\tilde{\mu}_\Omega\right)_{i\geq 0}
\end{equation}
of probabilities on~$G/\Gamma_N$, whose projection on~$G/\Gamma$ gives back~$\left( g_i\cdot\mu_\Omega\right)_{i\geq 0}$. Note that the~$g_i\cdot\tilde{\mu}_\Omega$ are supported inside the closed subset~$\overline{\Omega}N/\Gamma_N$.
 (cf. the argument between~\eqref{keypoint} and~\eqref{purity}.)

In order to apply Proposition~\ref{Propoliftcvg} to the map~$\overline{\Omega}N/\Gamma_N\xrightarrow{\pi}\overline{\Omega}N\Gamma/\Gamma$, we need to check a few things.
\begin{lemma}
\begin{itemize}
\item  The locus of points~$\left.\overline{\Omega}N\Gamma\middle/\Gamma\right.$ where the fiber of the projection~$\pi$ is a singleton, form an open subset.
\item The projection~$\pi:\left.\overline{\Omega}N\Gamma_N\middle/\Gamma_N\right.\to\left.\overline{\Omega}N\Gamma\middle/\Gamma\right.$
is proper and, above the considered locus, an homeomorphism.
\end{itemize}
\end{lemma}
\begin{proof}[Proof of the openness statement]
 Let~$C$ be an arbitrary large compact subset of~$N$. For~$c$ in~$C$, we consider the fiber~$c\Gamma/\Gamma_N$ of~$\varphi$ at~$c$, and its intersection with~$\overline{\Omega}N/\Gamma_N$, namely
 $$c\Gamma/\Gamma_N\cap \overline{\Omega}N/\Gamma_N.$$
 an element of this intersection can be written
 $$ c\gamma\Gamma_N= \omega n \Gamma_N$$
 with~$\gamma\in\Gamma$, with~$\omega\in\overline{\Omega}$ and with~$n$ in~$N$.

For such~$\gamma,\omega$ and~$n$, we have
$$\gamma\Gamma_N=c^{-1}\omega n \Gamma_N.$$
The coset~$\gamma N$ belongs to~$C^{-1}\overline{\Omega}N/N$, which is the image of a compact subset in the quotient space~$G/N$. It is a compact subset. But the orbit~$\Gamma N/N$ is closed, and discrete. Hence the coset~$\gamma N$ belong to the finite set
$$ F_C:=C^{-1}\overline{\Omega}N/N\cap \Gamma N/N.$$

Consider the function
\begin{equation}\label{semicontinu}
\chi:c\mapsto \# \left( c\Gamma/\Gamma_N \cap \overline{\Omega}N/\Gamma_N \right)
\end{equation}
over~$C$. It is the sum over the cosets~$f\Gamma_N$ in~$F_C$ of the characteristic function~$\chi_{f\Gamma_N}$ of the set
$$\{c\in G~|~cf\Gamma_N\in \overline{\Omega} N/\Gamma_N \}= \overline{\Omega} N f^{-1}.$$
The latter is a closed subset. The function~$\chi_{f\Gamma_N}$ is upper semi-continuous. Then
$$\chi=\sum_{f\Gamma_N\in F_C}\chi_{f\Gamma_N}$$
is upper semi-continuous.

As a consequence, the locus~$\chi\geq 2$ is closed in~$C$. As~$C$ is arbitrarily large in~$N$, the locus~$\chi\geq 2$ is closed in~$N$.

As the function~$\chi$ is right~$\Gamma$-invariant, we get the conclusion.
\end{proof}
\begin{proof}[Proof of the openness of the projection] Let~$K$ be a compact subset of~$\overline{\Omega}N\Gamma/\Gamma$. It is contained in~$\overline{\Omega}C$, for a sufficiently large compact subset~$C$ of~$N$. As above, its inverse image~$\stackrel{-1}{\pi}(K)$ in~$\overline{\Omega}N/\Gamma_N$ is contained in~$\overline{\Omega}C F_C$, which is compact in~$G/\Gamma_N$. The intersection with the closed subset~$\overline{\Omega}N/\Gamma_N$ is a compact~$\overline{\Omega}C F_C\cap\overline{\Omega}N/\Gamma_N$.

The inverse image~$\stackrel{-1}{\pi}(K)$ is contained in a compact of~$\overline{\Omega}N/\Gamma_N$. But it is closed, as~$K$ is closed and~$\pi$ is continuous. Thus~$\stackrel{-1}{\pi}(K)$ is compact. As~$K$ is arbitrary, the application~$\pi$ is proper.

To conclude, it suffices to remark that a bijective continuous proper map between metrisable locally compact spaces is necessarily an homeomorphism.
\end{proof}

We now apply Proposition~\ref{Propoliftcvg}. The sequence~\eqref{naivelift} has a weak limit, and this limit is~\eqref{lift}.
\begin{equation}\label{liftlimit}
\lim_{i\geq 0} g_i\cdot \tilde\mu_\Omega=\tilde\mu_\infty
\end{equation}

\subsubsection{Lifting the ergodic decomposition}
Note\footnote{By definition~$N$ normalises~$L$, and hence normalises~$L^+$ (which is invariant in~$L$ under algebraic automorphisms). Hence~$\Gamma\cap N$ normalises both~$L\cap\Gamma$ and~$L^+$, hence normalises~$\Lpp=\overline{L^+(\Gamma\cap L)}$.} that~$\Gamma_N$ normalises~$\Lpp$ in~$G$. Consequently, the quotient space~$G/\Gamma_N$ admits a \emph{right} $\Lpp$-action, with correponding quotient space~$G/(\Lpp\Gamma_N)$. 
We know that the~$W$-ergodic components of~${\mu_\infty}^{[L]}$ are all translates~$c\cdot\mu_{\Lpp}$. 
If~$c\cdot\mu_{\Lpp}$ is one of the~$W$-ergodic components of~$\mu_\infty$, recall by~\eqref{cbelongs} that
%\begin{itemize}
%\item firstly~$c\Gamma\in\overline{\Omega}N\Gamma/\Gamma$, as the support of~$\mu_\infty$ is contained in~$\overline{\Omega}N\Gamma/\Gamma$;
%\item secondly,~$c\in X^*(L,W)$, because~$W$ acts ergodically on the translate~$c\cdot\mu_{\Lpp}$.
%\end{itemize}
%In other words~
$c$ belongs to~$\left(\overline{\Omega}N\right)^*$.

As~$\left(\overline{\Omega}N\right)^*\Gamma/\Gamma$ is of full measure for~$\mu_\infty$, it is of full measure for almost all ergodic component. Consequently almost all ergodic component~$c\cdot\mu_{\Lpp}$ has a unique lift to~$\left(\overline{\Omega}N\right)/\Gamma_N$, which is
\begin{equation}\label{ergodiclift}
{\psi^{-1}}_\star(c\cdot\mu_{\Lpp}).
\end{equation}

Denote~$\tilde{\mu}_{\Lpp}$  the probability on~$\Lpp\Gamma_N/\Gamma_N$ which is left~$\Lpp$-invariant. Consider the lift~$c\cdot\tilde{\mu}_{\Lpp}$ of~$c\cdot\mu_{\Lpp}$ from~$G/\Gamma$ to~$G/\Gamma_N$. In order to prove that this actually the unique lift to~$\left(\overline{\Omega}N\right)^*/\Gamma_N$, it suffices to prove that~$\left(\overline{\Omega}N\right)/\Gamma_N$ is of full measure for~$c\cdot\tilde{\mu}_{\Lpp}$. But~$c$ belongs to~$\left(\overline{\Omega}N\right)^*$, and~$\left(\overline{\Omega}N\right)$ is right~$\Lpp$ invariant (as~$N$ contains~$L$.) Thus the support of~$c\cdot\tilde{\mu}_{\Lpp}$ is contained in~$\left(\overline{\Omega}N\right)/\Gamma_N$.

In the same way that~$\mu_\infty$ is made of components of the type~$c\cdot{\mu}_{\Lpp}$, its lift~${\tilde{\mu}_\infty}$ is made of the corresponding lifts~$c\cdot\tilde{\mu}_{\Lpp}$. In particular, we have proved that
\begin{equation}\label{claimLinvariant}\text{``${\tilde{\mu}_\infty}$ is right $\Lpp$-invariant''.}\end{equation}

\subsubsection{Reduction to a subquotient~$\widehat{G}$}\label{induction} We will start our induction process.

Define the measure~${\widehat{\mu}_\infty}$ as the direct image of~${\tilde{\mu}_\infty}$ along
\[G/\Gamma_N\xrightarrow{}G/\left(\Gamma_N\cdot \Lpp\right).\] 
We define also the direct image~$\widehat{\mu}_\Omega$ of~$\tilde{\mu}_\Omega$. Thanks to~\eqref{claimLinvariant}, we know that~${\widehat{\mu}_\infty}$ and~${\tilde{\mu}_\infty}$ determine each other completely.
 We note that to establish formula~\eqref{limitformula} identifying~$\mu_\infty$ in the Theorem, it will suffice to identify~$\widehat{\mu}_\infty$.

More precisely direct image induces a~$G$-equivariant bijection between right~$\Lpp$ invariant probabilities on~$G/\Gamma_N$ and probabilities on~$G/(\Lpp\Gamma_N)$.  Consequently, the stabiliser of~${\widehat{\mu}_\infty}$ in~$G$ equals the stabiliser of~${\tilde{\mu}_\infty}$ in~$G$.

Recall from~{\S}\ref{subsubsection-W} that~$W=M^+\subseteq M$. 
Moreover~$M$ stabilises~$\tilde{\mu}_\infty$. It hence stabilises~${\widehat{\mu}}_\infty$ also. 

By~Lemma~\ref{lemdebut},~$H$ and~$F$ both normalise~$M$ and we can consider the quotient group
\begin{equation}\label{Ghat}
\widehat{G}:=HF/M.
\end{equation}
Denoting~$\widehat{g}_i$ the image of~$g_i\in F$ in~$\widehat{G}$, then the image of~$g_i\tilde{\mu}_\Omega$ is~$\widehat{g}_i\widehat{\mu}_\Omega$. Note that the measure~$\widehat{\mu}_\infty$ is the limit of the sequence~$\left(\widehat{g_i}\widehat{\mu}_\Omega\right)_{i\geq0}$.

\subsubsection{Open and finite index}
Before proceeding furthermore, let us check that the subgroup~$M$ in~${L}^H$ is open of finite index.
\begin{proof} First observe The group~$L^H$ is algebraic over~$\Q_S$.

As~$\Lpp$ is open and of finite index in~$L$, the subgroup~$\ell=\Lpp\cap L^H$ of~$L^H$ is open and of finite index~$n\in\Z_{>0}$. The group~$L^H$ is normalised by~$H$, and~$H$ acts by continuous algebraic automorphism. Each conjugate~$\ell^h=h^{-1} \ell h$ of~$\ell$ by some~$h$ in~$H$ is another open subgroup of~$L^H$ of the same finite index~$n\in\Z_{>0}$. Their intersection is~$M$. It would suffice that they are finitely many distinct~$\ell^h$.

Actually, see Lemma~\ref{Lemma(F)},~$L$ has finitely many open subgroups of given finite index~$n\in\Z_{>0}$ (The profinite continuous quotient of~$L$ is of type~$(F)$ as in~\cite[{\S}6.4]{PR} or~\cite[{\S}4.1]{SerreLNM5}; actually it is topologically finitely generated).
\end{proof}

As a consequence the comparison map
\[\widehat{G}=HF/M\to \check{G}:=HF/L^H\]
is surjective with finite kernel: one could say an isogeny. Let~$\underline{H}$,~$\underline{F}$ and~$\underline{L^H}$
denote the~$\Q_S$-algebraic group associated with~$H$,~$F$ and~$L^H$.
Then $\check{G}$ is open of finite index (\cite[{\S}3.1 Prop.~3.3 Cor.~1, {\S}6.4]{PR}) in the group of
rational points~$\underline{\check{G}}(\Q_S)$ of 
\[\underline{\check{G}}=\underline{H}\,\underline{F}\,/\underline{L^H}.\]
In particular,~$\widehat{G}$ is locally isomorphic to the linear algebraic group~$\underline{\check{G}}(\Q_S)$
over~$\Q_S$. We can then define the Lie algebra~$\widehat{\lie{g}}$ of~$\widehat{G}$ over~$\Q_S$, which we will identify
with the Lie algebra of~$\check{G}$. The adjoint representation of~$\widehat{G}$ on~$\widehat{\lie{g}}$ factors through~$\check{G}$
via the comparison map above.

\subsubsection{Proving boundedness~\eqref{focusingcriterion} by contradiction}\label{provingbdd}
Let~$\widehat{H}$ be the image of~$H$, 
and~$\widehat{\lie{h}}$ the corresponding Lie subalgebra of~$\widehat{\lie{g}}$.
We write~$\widehat{Z}$ for the centraliser of~$\widehat{\lie{h}}$ in~$\widehat{G}$. 

%Let us denote~$\left(\hat{g_i}\right)_{i\geq 0}$ the image sequence, in~$\widehat{G}$ of our sequence~$\left({g_i}\right)_{i\geq 0}$.
We will now prove the following claim
\begin{equation}\label{claim1}
\text{``the sequence~$\left(\widehat{g_i}\right)_{i\geq 0}$ is of class~$O(1)\widehat{Z}$".}
\end{equation}
Let~$\widehat{\lie{h}}$ denote the~$\Q_S$-Lie algebra of~$\widehat{H}$. By virtue of Proposition~\ref{Propobddcrit} (applied to~$\check{G}$), the claim~\eqref{claim1} is actually equivalent to
\begin{equation}\label{claim2}
\text{``the sequence~$\left(\Ad_{\widehat{g_i}}\right)_{i\geq 0}$ is uniformly bounded on~$\widehat{\lie{h}}$".}
\end{equation}
We will prove~\eqref{claim2} by contradiction. We first show that the failure of~\eqref{claim2} implies that~$\widehat{\mu}_i$ is invariant under non trivial unipotent one parameter subgroups of~$\widehat{G}$. This part of the argument is well known. We stress this argument is the main reason 
why Ratner theory of unipotent flows is relevant to our subject, and is responsible for the strategy of proof we are following since the start.
\begin{proof}[Proof of~\eqref{claim2} by contradiction]
Define the closed subset~$X$ of~$G/(\Gamma_N\cdot \Lpp)$ as the closure of~$HN/(\Gamma_N\cdot \Lpp)$. 
\begin{equation}\label{defX}
X=\overline{HN/(\Lpp\Gamma_N)}\subseteq G/(\Gamma_N\cdot \Lpp)
\end{equation}
We note that~$HN/(\Gamma_N\cdot \Lpp)$ is stable under the action of~$HF$; and that it is acted upon trivially under~$M$. Consequently~$X$ is a locally compact space under which~$\widehat{G}$ acts continuously.

Assume by contradiction that, maybe extracting a subsequence, there is a sequence~$(X_i)_{i\geq0}$ in~$\widehat{\lie{h}}$ converging to~$0$ in~$\widehat{\lie{h}}$ such that the sequence~$\left(\Ad_{\widehat{g_i}}(X_i)\right)_{i\geq 0}$ has a nonzero limit~$X_\infty$ in~$\widehat{\g}$.

We note that
 \begin{itemize}
\item for big enough~$i$, the element~$X_i$ will be close enough to~$0$ so that the exponential~$\exp_{\widehat{H}}$ of the group~$\widehat{H}$ converges at~$X_i$, and~$\widehat{h}_i:=\exp_{\widehat{H}}(X_i)$ belongs to~$\widehat{H}$;
\item for such~$i$, one has~$\exp_{\widehat{G}}\left(\Ad_{\widehat{g_i}}(X_i)\right)=\widehat{g_i}\exp_{\widehat{H}}(X_i)\widehat{g_i}^{-1}=\widehat{g_i}\widehat{h}_i\widehat{g_i}^{-1}$;
\item the measure~$\widehat{\mu}_\infty$ is supported inside~$X$, this is the limit of~$(\widehat{g}_i\cdot\widehat{\mu}_\Omega)_{i\geq0}$, and each~$\widehat{g}_i\cdot\widehat{\mu}_\Omega$ is also supported inside~$X$;
\item the sequence~$(\widehat{h_i})_{i\geq0}$ converges to the neutral element, and the sequence~$(\widehat{h}_i\cdot\widehat{\mu}_\Omega)_{i\geq0}$ converges strongly to~$\widehat{\mu}_\Omega$;
\item the sequence~$\left(\widehat{g_i}\widehat{H}_i\widehat{g_i}^{-1}\right)_{i\geq0}$ converges to a non neutral element~$\widehat{g}_\infty:=\exp_{\widehat{G}}(X_\infty)$ of~$\widehat{G}$.
\end{itemize}

Applying Proposition~\ref{unipotescence},~$\widehat{g}_\infty$ fixes~$\widehat{\mu}_\infty$. Consequently, for any lift~$g_\infty$ of~$\widehat{g}_\infty$,
$$g_\infty\cdot\widetilde{\mu}_\infty=\widetilde{\mu}_\infty$$
and finally
$$g_\infty\cdot{\mu}_\infty={\mu}_\infty.$$
Consequently,~$g_\infty$ is an element of~$\Stab_G(\mu_\infty)$ which does not belong to~$W$.

Renormalising the sequence~$(X_i)_{i\geq0}$ by a constant non zero factor, we can assume~$X_\infty$
is in a chosen neighbourhood of~$0$ in~$\g$. We can choose such a neighbourhood such that~$\exp_{\widehat{G}}$
 and~$\exp_{\check{G}}$ induce (compatible) local  homeomorphisms between~$\widehat{G}$, $\check{G}$ and
 this neighbourhood. As a consequence~$\widehat{g}_\infty$ will also correspond to (will map to) a non neutral element~$\check{g}_\infty$
 in~${\check{G}}$.
 
Note that~$\check{g}_\infty$ is unipotent\footnote{We needed to work with~$\widehat{G}$ in order to use its action on measures such as~$\mu_\infty$,
and we need to pass to~$\check{G}$ which is algebraic, in order to use the algebraic definition for a unipotent element.} in~$\check{G}$ (its conjugacy class is adherent to the neutral element).

%TODO Porpoliftuni
But, by {\S}\ref{subsubsection-liftuni} this contradicts the definition of~$W$ (see p.~\pageref{defW}).
\end{proof}

We now have proved~\eqref{claim1}. This actually is a form of the focusing criterion, from which 
we will derive~\eqref{focusingcriterion}.

\subsubsection{Conclusion} We will now extract information from our previous work.

 Let~$\check{H}$ be the image of~$H$ in~$\check{G}$. It is Zariski connected
with Lie algebra~$\widehat{\lie{h}}$. The centraliser~$\check{Z}$ of~$\check{H}$ in~$\check{G}$ is that of~$\widehat{\lie{h}}$.
Let~$Z$ be its inverse image in~$HF$; this is also the inverse image of~$\widehat{Z}$. It contains~$(H^\prime\cap HF)\cdot L^H$. By~\ref{}, 
the Zariski closure of~$(H^\prime\cap HF)\cdot L^H$ is open of finite index in that of~$Z$. As~$(H^\prime\cap HF)\cdot L^H$ of open of finite
index in its Zariski closure, it is open of finite index in~$Z$. % It is also normal in~$Z$.
It follows that the following are equivalent
\begin{itemize}
\item the sequence~$(g_i)_{i\geq0}$ is of class~$O(1)\cdot Z$ in~$HF$;
\item the sequence~$(g_i)_{i\geq0}$ is of class~$O(1)\cdot (H^\prime\cap HF) L^H$ in~$HF$;
\item the sequence~$(\widehat{\mu}_i)_{i\geq0}$ is of class~$O(1)\cdot \widehat{Z}$ in~$\widehat{G}$;
\item the image sequence in~$\check{G}$ is of class~$O(1)\cdot \check{Z}$.
\end{itemize}
We have proved 
\begin{equation}
\text{``the sequence~$\left(g_i\right)_{i\geq 0}$ is of class~$O(1)\cdot (H^\prime\cap HF)\cdot L^H$".}
\end{equation}
We recall from~{\S}\ref{section-lin-focusing} that the sequence~$\left(g_i\right)_{i\geq 0}$ also evolves in~$F$. 
Let us write
\[g_i=b_i\cdot z_i\cdot m_i\]
with a bounded sequence~$(b_i)_{i\geq0}\in O(1)$ in~$G$, a sequence~$(z_i)_{i\geq0}$ in~$(H^\prime\cap HF)$
and~$(m_i)_{i\geq0}$ in~$L^H$. Reducing in~$HF/F$, the~$m_i$ maps to the neutral element. It follows 
that the sequence of cosets~$(z_iF)_{i\geq0}$ is bounded in~$(H^\prime\cap HF)F/F\simeq(H^\prime\cap HF)/(H^\prime\cap F)$: it is inverse
to the sequence~$(b_iF)_{i\geq0}$.
As a consequence we have actually
\begin{equation}\label{claim1}
\text{``the sequence~$\left(g_i\right)_{i\geq 0}$ is of class~$O(1)\cdot (H^\prime\cap F)\cdot L^H$".}
\end{equation}
This is the first part~\eqref{focusingcriterion} of the focusing criterion from the Theorem.

Actually, as~$L^H$ is normalised by~$F$, and~$M$ is open of finite index in~$L^H$, we know by Lemma~\ref{Lemma(F)}, that~$M$
contains~$M^F=\bigcap_{f\in F} fM f^{-1}$ as an open subgroup of finite index normalised by~$F$. We have identity of classes of sequences in~$F$
\[
O(1)\cdot(H^\prime\cap F)\cdot L^H=O(1)\cdot(H^\prime\cap F)\cdot M=O(1)\cdot(H^\prime\cap F)\cdot M^F.
\]

We now move to proving~\eqref{finefocusing}. We can write
\[g_i=b_i\cdot z_i\cdot m_i\]
with a bounded sequence~$(b_i)_{i\geq0}\in O(1)$ in~$G$, a sequence~$(z_i)_{i\geq0}$ in~$(H^\prime\cap F)$
and~$(m_i)_{i\geq0}$ in~$M$. Then
\[\widehat{g}_i=\widehat{b}_i\cdot \widehat{z}_i\]
is the corresponding decomposition in~$\widehat{G}$.
From~\eqref{liftlimit} we deduce
\begin{equation}
\lim_{i\to\infty} \widehat{g_i}\cdot \widehat{\mu}_\Omega=\widehat{\mu}_\infty.
\end{equation}

Contrasting with the initial situation, the translating elements are made of a bounded part and a part~$\widehat{z}_i$ 
``commuting with~$\Omega$''. More precisely, the support of~$ \widehat{g_i}\cdot \widehat{\mu}_\Omega$ is
\[\overline{\widehat{b}_i\cdot \widehat{z}_i \cdot\Omega\cdot (\Gamma_N \Lpp)/(\Gamma_N \Lpp)}=\widehat{b}_i\cdot \overline{\Omega}\cdot \widehat{z}_i\cdot (\Gamma_N \Lpp)/(\Gamma_N \Lpp).\]
Both the~$\widehat{b_i}$ and~$\Omega$ are bounded. If the sequence of cosets~$\left(\widehat{z}_i\cdot (\Gamma_N \Lpp)\right)_{i\geq0}$ is not bounded, then one cannot have convergence of measures, but to the zero measure, which is not the case. 
Hence~$(\widehat{z}_i)_{i\geq0}$ is both~$O(1)\cdot (\Gamma_N \Lpp)$ and in~$H^\prime\cap F$.

We stress that the orbit of~$H^\prime\cap F$ on~$X=\overline{HF/(\Gamma_N\Lpp})$ need not be closed in general. Let~$x_\infty$ be an 
adherence value of the bounded sequence~$(\widehat{z}_i\cdot (\Gamma_N \Lpp))_{i\geq0}$ of cosets. Then~$x_\infty$ belongs to the 
closed orbit~$N/(\Gamma_N\Lpp)$. It can be lifted to an element~$n_\infty$ of~$N$ in the topological closure of~$(H^\prime\cap F)$. 

Let us pass to a subsequence corresponding to this adherence value. Then we have
\[\lim_{i\to\infty} \widehat{z}_i\cdot\widehat{\mu}_\Omega=\mu_{\Omega\cdot n_\infty\Gamma_N\Lpp}\]
where the right hand side is the image measure of~$\mu|_\Omega$ via~
\[\Omega\to \overline{HF/(\Gamma_N\Lpp)}:\omega\mapsto\omega \cdot n_\infty(\Gamma_N\Lpp).\]

We can deduce that 
\[\lim_{i\to\infty}\widehat{b_i}\cdot \mu_{\Omega\cdot n_\infty\Gamma_N\Lpp}=\widehat{\mu}_\infty.\]
This last part of the proof may depend on~$\Omega$.

The sequence~$(b_i)_{i\geq0}$ is bounded in~$G$ and contains a subsequence of the form~$g_\infty o(1)$ 
where~$g_\infty$ is an adherence value in~$G$. We conclude to~\eqref{finefocusing}.

Moreover we obtain
\[\widehat{\mu}_\infty=g_\infty\mu_{\Omega\cdot n_\infty\Gamma_N\Lpp}.\]
Which, recalling how~$\widetilde{\mu}_\infty$ and hence~$\mu_\infty$ are characterised by~$\widehat{\mu}_\infty$,
gives~\eqref{limitformula}.

We claim the statement of Theorem~\ref{Theorem} reflects what we have proved so far.
\end{proof}
This concludes the proof of Theorem~\ref{Theorem}, up to the statements we postponed to the next section.

\subsection{Proof of Theorem~\ref{Theorem} -- Postponed statements}

\subsubsection{A Variant of~\cite[Theorem~2.1]{LemmaA}}
Here we provide a variant of~\cite[Theorem~2.1]{LemmaA} which differs in two ways. Firstly we refer to the (product) norm
$$\left(x_v\right)_{v\in S}\mapsto\max_{v\in S}\Nm{x_v}$$
rather that the product function
$$\left(x_v\right)_{v\in S}\mapsto\prod_{v\in S}\Nm{x_v},$$
the latter which may not be proper. Secondly we will state it for any bounded open subset~$\Omega$, instead of a product~$\prod_{v\in S}\Omega_{v}$.

This Theorem is not used in the proof of Theorem~\ref{Theorem}. It is used to remove hypothesis~\eqref{Hypo} in the Theorems~\ref{Theointro} and~\ref{Theointro2} from the Introduction (see~{\S}\ref{analytic stability}).
\begin{theorem}\label{thm21bis} The Theorem~2.1 from~\cite{LemmaA} holds with the following modifications\footnote{Other than the syntactically erroneous~"given".} applied together. We use the set~$Y=Y_S$ (minding the typo)
\begin{itemize}
\item Let~$\Omega$ be any bounded open subset of~$H$;
\item the concluding formula~\cite[(9)]{LemmaA} is replaced by
\begin{equation}
\forall y\in Y,~\forall\left(x_v\right)_{v\in S}\in V,
\sup_{\omega=\left(\omega_v\right)_{v\in S}\in\Omega}
\max_{v\in S}\Nm{\rho_v(y\cdot\omega_v)(x_v)}_v
\geq
\max_{v\in S}\left.\Nm{x_v}_v\right/c.
\end{equation}
\end{itemize}
\end{theorem}
\begin{proof}[Proof of Theorem~\ref{thm21bis}]
As~$\Omega$ is nonempty open in~$H$, it contains a basic nonempty open subset of the form~$\prod_{v\in S}\Omega_v$, as these form a basis of the product topology. As each~$\Omega_v$ is nonempty and open, it is Zariski dense in~$H_v$.

The proof then goes forward analogously as in~\cite[p.VI-6/124]{LemmaA}. We have at each place
\[\forall y\in Y_v,~\forall x_v\in V_v, \sup_{\omega\in\Omega_v}\Nm{y_v \omega_y x_v}_v\leq\Nm{x_v}_v/c_v.\] 
We deduce the inequality
\[
\max_{v\in S}\sup_{\omega\in\Omega_v}\Nm{y_v \omega_y x_v}_v\geq \max_{v\in S}\left(\Nm{x_v}_v\middle/c_v\right)
\]
and finally
\[
\sup_{\omega\in\Omega_v}\max_{v\in S}\Nm{y_v \omega_y x_v}_v\geq \left.\left(\max_{v\in S}\Nm{x_v}_v\right)\middle/\left(\max_{v\in S}c_v\right)\right.	
\]
which gives the conclusion with~$c=\max_{v\in S}c_v$.
\end{proof}

\subsubsection{Lifting weak convergence of probabilities outside a negligible closed set.}
This property is used in the main proof to lift convergence~$g_i\mu_\Omega\to\mu_\infty$ 
inside~$G/\Gamma$ to a convergence~$g_i\widetilde{\mu}_\Omega\to\widetilde{\mu}_\infty$ in~$G/\Gamma_N$,
using the map~\eqref{finalbij}.

\begin{proposition}\label{Propoliftcvg}
Let~$X\xrightarrow{\pi} Y$ be a continuous map between locally compact spaces inducing an homeomorphism from an open subset~$U_X$ of~$X$ to an open subset~$U_Y$ of~$Y$.

Let~$(\mu_i)_{i\geq0}$ be a sequence of probabilities on~$Y$ with weak limit a probability~$\mu_\infty$, such that~$U_Y$ is of full measure for~$\mu_\infty$.
Then for any sequence~$(\widetilde{\mu}_i)_{i\geq0}$ of probabilities on~$X$ such that
$$\forall i\geq 0, \pi_\star\tilde{\mu}_i={\mu}_i,$$
the weak limit~$\tilde{\mu}_\infty:=\lim (\tilde{\mu}_i)_{i\geq0}$ exists and satisfies
$$\pi_\star\tilde{\mu}_\infty={\mu}_\infty.$$
\end{proposition}
\begin{proof} 
As~$\mu_\infty$ is a probability, the sequence~$(\mu_i)_{i\geq0}$ is tight: for any~$\epsilon>0$, there exists a compact~$K$ in~$Y$ such that
$$\limsup_{i\geq0} \mu_i (K) >1-\epsilon $$
Denote~$Z_X$ and~$Z_Y$ the complementary closed subsets of~$X$  and~$Y$ respectively. As
$$\limsup_{i\geq0} \mu_i (K\cap Z_Y) =0,$$ we can erase from~$K$ a small enough neighbourhood~$V$ of~$K\cap Z_Y$ in~$K$ and get, for the compact~$K\setminus V$,
\begin{equation}\limsup_{i\geq0} \mu_i (K\setminus V) >1-2\epsilon \label{tightout}\end{equation}

Take a sequence~$(\tilde{\mu}_i)_{i\geq0}$ as in the statement. Let us prove that this is a tight sequence. 
From
$$\tilde{\mu_i} \left(\stackrel{-1}{\pi}(K\setminus V)\right)=\mu_i (K\setminus V),$$
we can lift~\eqref{tightout} to~$X$:
$$\limsup_{i\geq0} \tilde{\mu_i} \left(\stackrel{-1}{\pi}(K\setminus V)\right) >1-2\epsilon.$$
Note that~$\stackrel{-1}{\pi}(K\setminus V)$ is a compact because~$\pi$ is an homeomorphism on~$U_X$.
As~$\epsilon$ is arbitrarily small, this proves the tightness of~$(\tilde{\mu}_i)_{i\geq0}$.

Let~$\mu$ be any limit point of the sequence~$(\tilde{\mu}_i)_{i\geq0}$. This is necessarily a probability. Then, by continuity of~$\pi_\star$ from the space of probabilities of~$X$ towards the space of probabilities of~$Y$, we have
$$\pi_\star(\mu)=\mu_\infty.$$

Consequently~$\mu$ is of full measure on~$U_X=\stackrel{-1}{\pi}(U_Y)$, and~$\mu$ is characterised by its restriction to~$U_X$, which is necessarily~$\stackrel{-1}{\pi}_\star({\mu_\infty}|_{U_Y})$.

Finally, the sequence~$\tilde{\mu}_\infty$ is tight, and has a unique limit point. This is hence a limit of the whole sequence~$(\tilde{\mu}_i)_{i\geq0}$. This concludes.

\end{proof}

\subsubsection{Lifting unipotents}\label{subsubsection-liftuni} This statement explain the behaviour of the property of being a unipotent elements in relation with passing to a quotient group, for algebraic linear group over a field of characteristic~$0$. This amounts to the Jordan decomposition in algebraic groups.
We refer to~{\S}\ref{secnotations} for the~${}^+$ notation.

We then give a consequence which is used in our main proof to relate the maximality property satisfied by~$W$ with the induction step~{\S}{\S}\ref{induction}--\ref{provingbdd}.

%Algebrique de trop. %TODO Hypothèse en trop ?
\begin{proposition}\label{propoliftuni}
 Let~$G$ be an linear algebraic group over a characteristic~$0$ local field~$k$, with a subgroup~$H$, and a normal subgroup~$N$ of~$H$, with subquotient group~$Q=H/N$. 

Then 
\begin{enumerate}
\item \label{propplus1} we have~${H(k)}^+\subseteq{G(k)}^{+}\cap H(k)$;
\item \label{propplus2} write~$\varphi:H\to Q$ the quotient map, we have~${Q(k)}^+\subseteq\varphi\left({H(k)}^{+}\right)$.
\end{enumerate}
\end{proposition}
\noindent(cf~\cite[{\S}III 4.3]{SerreLNM5} and~\cite[{\S}3.18-19]{BorelTits} \emph{à propos} local fields of  non zero characteristic.)
\begin{proof}
Let~$u$ be unipotent element of~$H(k)$. It is in particular a unipotent element of~$G(k)$ which belongs to~$H(k)$. Thus~${H(k)}^{+}\cap G(k)$ contains a generating set of~${H(k)}^+$. It finally contains~${H(k)}^+$. (NB: The converse inclusion may not hold!)

Let~$u$ be a unipotent element of~$Q(k)$. Note that~$\varphi\left({H}(k)\right)$ is a finite index subgroup of~$Q(k)$ (\cite[{\S}6.4: Proposition~6.13 p316; Corollary~2 {\S}6.4 p.319]{PR}).  By~\cite[Corollaire~6.7]{BorelTits},~${H}(k)\cdot N/N$ contains~${Q(k)^+}$. Thus there exists~$g$ in~$H(k)$ such that~$\phi(g)=u$.

Let~$g=g_s\cdot g_u$ the Jordan decomposition of~$g$ (\cite[I.4.4 Theorem,~(1)]{BorelLAG}). By~\cite[I.4.4 Theorem,~(4)]{BorelLAG},~$\phi(g_s)\phi(g_u)$ is the Jordan decomposition of~$\phi(g)$. But~$\phi(g)=u$ has manifestly decomposition~$u=e\cdot u$, where~$e$ is the neutral element of~$Q(k)$. We get~$\phi(g_u)=u$ with a unipotent element~$g_u$ of~$H(k)$  (and we get~$\phi(g_s)=e$). In particular,~$g_u$ belongs to~${H(k)}^+$. 

As~$u$ was arbitrary, the image of~${H(k)}^+$ in~$Q(K)$ contains the unipotent elements of the latter, hence generates~${Q(k)}^+$.
\end{proof}
We deduce immediately the following.
\begin{corollary}\label{corliftuni}
Let~$W$ be a subgroup of~$H(k)$ containing~$N(k)$ such that~$W^+\subseteq N(k)$. Then~$\varphi(W)^+$ is reduced to the neutral element~$e$ of~$Q(k)$.
\end{corollary}
\begin{proof}Assume by contradiction that~$\varphi(W)$ contain a non trivial unipotent one parameter group~$U$. For~$u$ in~$U$, let~$g$ be a preimage of~$u$ in~$H(k)$. Let~$g=g_s\cdot g_u$ be the Cartan decomposition of~$g$ in~$H(k)$. Then as above,~$\varphi(g_u)=u$ and~$\varphi(g_s)=e$. Consequently~$g_s$ belongs to~$N(k)$ hence to~$W$. Thus~$g_u={g_s}^{-1}\cdot g$ belongs to~$W$ too. It is contained in~$W^+$ and so in~$N(k)$. Hence~$u=\phi(g_u)$ is trivial. Arguing for each~$u$ in~$U$, we see that~$U$ is trivial. Contradiction.
\end{proof}

\subsubsection{Creating unipotent invariance} Here is a version of a standard argument to obtain unipotent groups stabilising a limit of translates of homogeneous measures.
\begin{proposition}\label{unipotescence} Let a locally compact group~$G$ acts continuously on a locally compact topological space~$X$. Let~$(g_i)_{i\geq0}$ be a sequence in~$G$, let~$\mu$ be a probability on~$X$, and assume the sequence of translated measures~$(g_i\cdot\mu)$ is weakly converging toward some limit~$\mu_\infty$. Assume there is a sequence~$h_i$ converging to the neutral element such that
\begin{itemize}
\item the sequence~$h_i\cdot \mu$ converges strongly to~$\mu$;
\item the sequence~$g_i h_i {g_i}^{-1}$ has a limit point~$g_\infty$.
\end{itemize}

Then~$\mu_\infty$ is fixed under translation by~$g_\infty$.
\end{proposition}
\begin{proof} Let~$\varphi$ be a test function: continuous with compact support. We write
$$g_i h_i {g_i}^{-1} g_i\mu (\varphi)=g_i h_i \mu (\varphi) = g_i \mu (\varphi) +o(1)\Nm{\varphi}$$
by strong convergence. But~$g_i \mu (\varphi)=\mu_\infty(\varphi)+o(1)$, by weak convergence. We get
\begin{equation}\label{combi1} g_i h_i {g_i}^{-1} g_i\mu (\varphi)=\mu_\infty(\varphi)+o(1).\end{equation}
On the other hand, we can write~$g_i h_i {g_i}^{-1} =o(1)g_\infty$. Note that, by uniform continuity of~$\varphi$, one has 
$\Nm{o(1)\cdot\varphi-\varphi}=o(1)$. Hence, for any probability~$\nu$, we have
$$o(1)\nu(\varphi)=\nu(o(1)\varphi)=\nu(\varphi)+o(1),$$
with error term independant of~$\nu$. Applying to~$\nu=g_i\mu$, we get
$$g_i h_i {g_i}^{-1} g_i\mu (\varphi)=g_\infty o(1) g_i\mu (\varphi)=g_\infty g_i\mu (\varphi) + o(1).$$
But~$g_i\mu$ converges weakly to~$\mu_\infty$, and
\begin{equation}\label{combi2}
g_\infty g_i\mu (\varphi) =g_i\mu ({g_\infty }^{-1}\varphi)= \mu_\infty ({g_\infty }^{-1}\varphi)+o(1)=g_\infty\mu_\infty (\varphi) +o(1).
\end{equation}
Combining~\eqref{combi1} and~\eqref{combi2} yields
$$\mu_\infty(\varphi)+o(1)=g_\infty\mu_\infty (\varphi) +o(1).$$
This concludes.
\end{proof}

\subsubsection{Finite generation and finite index subgroups} We know that~$\Lpp$ is a open finite index subgroup of~$L$.
The following argument is used in the main proof to establish the same relationship between~$\bigcap_{h\in H} h\Lpp h^{-1}$ 
with~$\bigcap_{h\in H} hLh^{-1}$. It follows from structure properties of linear group over real or $p$-adic fields. (see references given below
for the positive characteristic local fields).

\begin{lemma} \label{Lemma(F)}
Let~$M$ the topological group of rational points of an algebraic group over~$\Q_S$. For every integer~$n$, there are only finitely many open subgroups of index~$n$ in~$M$.
\end{lemma}
\begin{proof} We can argue for each place separatly. For an archimedean place, we know that the neutral component is a minimal open subgroup 
and is of finite index (\cite[{\S}3.3]{PR}). 

Let us turn to ultrametric places. We may assume~$M$ is (Zariski) connected. We start with three
cases.

By~\cite[Corollaire~6.7]{BorelTits}, every subgroup of finite index of~$M$ will contain~$M^+$.
(It settle the case of groups of non compact type whose radical is unipotent using~\cite[Proof of Proposition~6.14]{BorelTits})

The case of a split torus is easily directly checked.

Let us note that some neighborhood of the origin is finitely generated, using for instance the 
Lie group exponential which has a positive radius of convergence. In the case of an anisotropic
reductive group, which is compact (\cite[3.18-19]{BorelTits},\cite[{\S}3.1 Theorem~3.1]{PR}), the whole group will be finitely generated. 
It also profinite (locally profinite and compact). Hence we can apply~\cite[Proposition9]{SerreLNM5}.

To conclude in the general case, we use the structure theory from~\cite{BorelTits}: Proposition~6.14, its proof, and Proposition~3.19 (in characteristic~$0$,
the inseparable degree~$q$ is~$1$).
\end{proof}

\subsubsection{Boundedness criterion and Geometric stability} This kind of results is related to the work of Richardson:~\cite{RichardsonConjugacy,Richardsontuple}. It is an instance of a phenomenon coined as stability
by Mumford (see~\cite{Richardsontuple} for discussion of Mumford stability in this context; see also Corollary~\ref{corostabilityrichardson} below).
We give provide a method which relies on~\cite{Lemmanew}, thus emphasising the meaning of~\cite{Lemmanew} as stability properties.

\begin{proposition}\label{Propobddcrit} Let~$ G $ be a Zariski connected linear~$\Q_S$-group, and let~$ H $ be a Zariski connected~$\Q_S$ subgroup which is reductive in~$ G $,
following~\eqref{H reductif dans}. For any sequence~$\left( g _i\right)_{i\geq0}$ in~$ G $, the following are equivalent.
\begin{enumerate}
\begin{subequations}
\item \label{P1} If~$ H ^\prime$ denotes the centraliser of~$ H $ in~$ G $, then
\begin{equation}\label{eqp1}
\left( g _i\right)_{i\geq0}\text{ is of class }O(1) H ^\prime.
\end{equation}
\item \label{P2} If~$\lie{h}$ (resp.~$\lie{g}$, resp.~$\Ad$) denotes the Lie algebra of~$ H $ (resp. of~$ G $, resp the adjoint representation of~$ G $ on~$\lie{g}$), then for any element~$X$ in~$\lie{h}$, the sequence
\begin{equation}\label{eqp2}
\left(\Ad_{ g _i}(X)\right)_{i\geq0}
\end{equation}
is bounded in~$\lie{g}$.
\item \label{P3} (with the same notations) For any bounded sequence~$\left( X_i\right)_{i\geq0}$ in~$\lie{h}$, the sequence
\begin{equation}\label{eqp3}
\left(\Ad_{ g _i}(X_i)\right)_{i\geq0}
\end{equation}
is bounded in~$\lie{g}$.
\item \label{P4} There does not exist a sequence~$\left( X_i\right)_{i\geq0}$ in~$\lie{h}$, converging to~$0$, and such that
\begin{equation}\label{eqp4}
\left(\Ad_{ g _i}(X_i)\right)_{i\geq0}
\end{equation}
has a non zero limit point in~$\lie{g}$.
\end{subequations}
\end{enumerate}
\end{proposition}
\begin{proof} 
The implication~$\eqref{P1}\Rightarrow\eqref{P2}$ is immediate. Assume~\eqref{P1}. As~$ H ^\prime$ fixes~$\lie{h}$, and in particular, fixes~$X$, we may replace~$\left( g _i\right)_{i\geq0}$ by a bounded sequence. Thus the sequence~\eqref{eqp2} wil be bounded.

The equivalences~$\eqref{P2}\Leftrightarrow\eqref{P3}\Leftrightarrow\eqref{P4}$ are merely reformulations. Here are some details.

Firstly,~$\eqref{P2}$ is a specialisation of~$\eqref{P3}$ to constant sequences~$(X_i)_{i\geq0}=(X)_{i\geq0}$.

Conversely,~$\eqref{P3}$ reduces to~$\eqref{P2}$, by decomposing~$X_i$ into a fixed base:~$X_i=\sum_j a_{i,j} e_j$, using finitely many bounded sequences~ $(a_{i,j})_{i\geq0}$. We then apply~$\eqref{P2}$ to each of the case~$X=e_j$, then deduce that each sequence~$\left(\Ad_{ g _i}(e_j)\right)_{i\geq0}$ is bounded. Consequently, each sequence~$\left(\Ad_{ g _i}(a_{i,j}e_j)\right)_{i\geq0}$ is bounded. Finally, taking a finite sum~$\left(\Ad_{ g _i}(\sum_j a_{i,j} e_j)\right)_{i\geq0}$ of bounded sequences yields a bounded sequence.

\paragraph{NB:}
 The equivalence~$\eqref{P2}\Leftrightarrow\eqref{P3}$ falls merely under an (easy) instance of Banach-Steinhaus uniform boundedness principle.
The former~$\eqref{P2}$ states pointwise boundedness, the latter~$\eqref{P3}$ states a sequential form of boundedness.

The negation of~$\eqref{P3}$ claim there is a sequence~$(X_i)_{i\geq0}$ such that the sequence~\eqref{eqp3} is unbounded. Choose a $\Q_S$-norm on~$\lie{g}$. We can find a sequence of scalars~$(\lambda_i)_{i\geq0}$ in~$\Q_S$ such that~$0<\limsup\Nm{\lambda_i\cdot\Ad_{ g _i}(X_i)}<\infty$ .
Consequently, the sequence
\[\left(\Ad_{ g _i}(\lambda_i\cdot X_i)\right)_{i\geq0}\] takes infinitely many value in the compact set
$$\left\{X\in\lie{g}~\middle|~1/2\leq\frac{\Nm{X}}{\limsup\Nm{\lambda_i\cdot\Ad_{ g _i}(X_i)}} \leq 2\right\}.$$
The sequence~$\left(\Ad_{ g _i}(\lambda_i\cdot X_i)\right)_{i\geq0}$ has a limit point in this compact. This limit point may not be~$0$. This negates~$\eqref{P4}$ for the sequence~$\left(\lambda_i\cdot X_i)\right)_{i\geq0}$.

Assume the negation of~$\eqref{P4}$. We can find a sequence of non invertible scalars such that~$(\lambda_i)_{i\geq0}$ with limit~$0$ and such that~$\Nm{X_i}=O(\Nm{\lambda_i})$. One can deduce the negation of~$\eqref{P3}$ for the sequence~$\left( X_i/\lambda_i)\right)_{i\geq0}$. 

The remaining implication~$\eqref{P1}\Leftarrow\eqref{P2}$ is proved below.
\end{proof} 
%The proof of the implication~$\eqref{P2}\Leftarrow\eqref{P1}$ will rely on the use of ...

We first start with a special case. (See also~{\cite[Theorem 2.16]{PR}}.)
\begin{proof}[Proof of the implication~$\eqref{P1}\Leftarrow\eqref{P2}$, for semisimple~$ G $. ]
~

Assume that~$ G $ is semisimple.

Applying \cite[Theorem~2.1]{LemmaA}, we get be a subset~$Y$ such that~$ G =Y\cdot  H ^\prime$. Consequently any class~${g_i}  H ^\prime$ can be written~$ y_i  H ^\prime$ with~$y_i$ in~$Y$. We can replace~$( g _i)_{i\geq0}$ by~$(y_i)_{i\geq0}$ without changing the validity of~\eqref{P1}.

As~$ H ^\prime$ acts via identity on~$\lie{h}$, we have
$$ \forall X \in\lie{h}, \Ad_{ g _i}(X)=\Ad_{{y}_i}(X).$$
Thus we can replace~$( g _i)_{i\geq0}$ by~$(y_i)_{i\geq0}$ without changing the validity of~\eqref{P2}.

We assume from now on the identity~$( g _i)_{i\geq0}=(y_i)_{i\geq0}$.

Consider the space
\[{V}:=\Hom_{\Q_S}\left(\lie{h},\lie{g}\right)\]
of~$\Q_S$-linear maps~$\lie{h}\to\lie{g}$, with adjoint action of~$ G $ on the values on a $\Q_S$-linear homomorphisms:~$( g \phi)(X):=\Ad_{ g }(\phi(X))$. 

Note that the stabiliser of the identity embedding~$\iota:\lie{h}\xrightarrow{X\mapsto X}\lie{g}$ is the centraliser of its image~$\lie{h}$. This centraliser is~$ H ^\prime$, as~$H$ is Zariski connected and characteristics are~$0$. For any~$\omega$ in~$ H $, the map~$\omega\cdot\iota$ has same image~$\lie{h}$, and same stabiliser~$ H ^\prime$.

Consider a open bounded subset~$\Omega$ in~$ H $. 
Remark that, for~$\omega\in\Omega$, the map~$g_i\cdot\omega\cdot \iota$ is the composite~$\left.\Ad_{g}\right|_\h \circ\Ad_\omega:\h\to\h\to\g$. Fix some norm on~$\h$ and on~$\g$. For the operator norm, 
$$\Nm{g_i\cdot\omega\cdot \iota}=\Nm{\left.\Ad_{g}\right|_\h \circ\Ad_\omega}\leq\Nm{\left.\Ad_{g_i}\right|_\h}\cdot\Nm{\Ad_\omega}= \Nm{g_i\cdot\iota}\cdot\Nm{\Ad_\omega}.$$
From hypothesis~\eqref{P2}, we know that~$\Ad_{y_i}(\iota)$ is bounded in~${V}$. But~$\Omega$ is also bounded.
Consequently,~$\Nm{g_i\cdot\Omega\cdot \iota}$ is uniformly bounded.
We may apply~\cite[Theorem~2]{Lemmanew}. Its conclusion uses the fixator of~$\Omega\cdot \iota$, which is~$ H ^\prime$, and is normalised by~$\Omega$. This conclusion is that~$(y_i)_{i\geq0}$ is of class~$O(1) H ^\prime$.
\end{proof}
We will need a slight extension of this special case.
\begin{proof}[Proof of the implication~$\eqref{P2}\Leftarrow\eqref{P1}$, assuming~$ G $ is reductive.] Assume that~$ G $ is reductive.

Let us write~$ G =C\cdot D$ as an almost direct product of its center~$C$ by its derived group~$D$. Decompose correspondingly~$ g _i=c_i\cdot d_i=d_i\cdot c_i$.

Note that~\eqref{eqp1} does not depend on~$c_i$ as~$C$ is contained in~$ H ^\prime$. Neither do the sequence~\eqref{eqp2}, has~$C$ is contained  in~$ H ^\prime$ which fixes~$\lie{h}$, to which~$X$ belongs.

We can replace~$ g _i$ by~$d_i$. The result can be deduced from the preceding case, applied to the semisimple group~$D$, its subgroup~$( H  C)\cap D$ and the sequence~$(d_i)_{i\geq0}$.
\end{proof}

We will now reduce the general case to the preceding special case.
\begin{proof}[Proof of the implication~$\eqref{P1}\Leftarrow\eqref{P2}$, in general]
Let~$U$ be the unipotent Radical of~$ G $. By~\cite[Theorem~7.1]{Mostow56}, we can decompose~$ G =M\cdot U$ as a semi product of a maximal reductive subgroup~$M$, which we assume to contain~$ H $, with~$U$.

%As~$ H $ is reductive, there is a Levi subgroup~$L$ of~$ G $ such that~$LT$ contains $ H $. (cf.~\cite[{\S} 6.8, Corrolaire~2]{BBKLie1}

Let~$\bar{G}=G/U$ be the quotient of~$ G $ by its unipotent radical.  We denote~$\bar{\lie{g}}$ the corresponding Lie algebra. Let us denote the image of the sequence~$( g _i)_{i\geq0}$ in~$\bar{G}$ with~$(\bar{g}_i)_{i\geq0}$.

Note that~$M$ , and in particular its subgroup~$ H $, identifies with its image in~$\bar{G}$. We will identify the Lie algebra of~$ H $ and of its image.

As above, we consider the space~${V}:=\Hom_{\Q_S}(\lie{h},\lie{g})$, with its element~$\iota$ and we consider moreover its variant~$\bar{V}:=\Hom_{\Q_S}(\lie{h},\bar{\lie{g}})$, both as representations of~$ G $. We denote with~$\bar{\iota}$ the ``identity'' map~${\lie{h}}\to\bar{\lie{g}}$. The representation~$\bar{V}$ is a quotient of the representation~${V}$, via a quotient map sending~$\iota$ to~$\bar{\iota}$.
 
By hypothesis~\eqref{P2}, the sequence~$({g_i}\cdot\bar\iota)_{i\geq0}$ is bounded in~${V}$. It follows that~$(\bar{g_i}\cdot\bar\iota)_{i\geq0}$ is bounded in~$\bar{V}$. We apply the previous case to~$\bar{G}$, its subgroup~$ H $, and the sequence~$(\bar{g_i}\cdot\bar\iota)_{i\geq0}$, and deduce that the sequence
\[
(\bar{g_i})_{i\geq0}\text{ is of class~}O(1)Z_{\bar{G}}(H),
\] 
where~$Z_{\bar{G}}(H)$ is the centraliser of~$H$ in~$\bar{G}$. 

By~\ref{} we deduce the sequence
\[
({g_i})_{i\geq0}\text{ is of class~}O(1)U\cdot H^\prime,
\] 
We can write~$g_i=b_i u_i z_i$ where~$(b_i)_{i\geq0}$  is a bounded sequence, the~$u_i$ belong to~$U$ and the~$z_i$ to~$H^\prime$.
We get
\[g_i\cdot \iota=b_i\cdot u_i\cdot z_i \cdot \iota= b_i\cdot u_i\cdot \iota.\]
Thus the boundedness of~$(g_i\cdot \iota)_{i\geq 0}$ in~$V$ reduces to
that of~$(u_i\cdot \iota)_{i\geq 0}$. But the orbit of unipotent group in a linear representation is closed, by the Kostant-Rosenlicht theorem \cite[2.4.14]{SpringerLAG}. The orbit map~$U/(U\cap H^\prime) \to U\cdot\iota$ is proper. Thus the sequence~$(u_i)_{i\geq 0}$ is actually bounded in~$U/(U\cap H^\prime)$. Finally we get more precisely
\[
({g_i})_{i\geq0}\text{ is of class~}O(1)(O(1)(U\cap H^\prime))\cdot H^\prime=O(1)H^\prime.
\]

%Note that~$ H ^\prime$ decomposes analogously~$( H ^\prime\cap M)\cdot ( H ^\prime\cap U)$. It can be seen at the Lie algebra level
%(the characteristic is~$0$). As the representations of~$ H $ are fully reducible, we can use~\ref{}.

%Let us decompose~$ g _i=  u_i m_i$ following a Levi decomposition~$ G =M\cdot U$, where the Levi factor~$M$ identifies with~$\bar{G}$, and is assumed to contain~$ H $.
%We deduce that~$(\bar{m_i}\cdot\iota)_{i\geq0}$ is of class~$O(1) Z_{M}( H )$ where~$Z_{M}( H )$ is the centraliser of~$ H $ in~$M$. We can write~$m_i=b_i z_i$ with bounded~$b_i$ and~$z_i$ in~$Z_{M}(H)$. Thus
%$$ g _i=  u_i m_i = u_i b_i z_i = b_i v_i z_i$$
%with~$v_i:= b_i^{-1}\cdot u_i \cdot b_i$. As the unipotent radical~$U$ is normal in~$ G $, the element~$v_i$ belongs to~$U$.
%
%By hypothesis the sequence~$(\hat{g_i}\cdot\iota)_{i\geq0}$ is bounded in~$\hat{V}$. We can rewrite
%$$ g _i \cdot\iota=  b_i v_i z_i \cdot\iota = b_i v_i \cdot\iota$$
%as~$z_i$ belongs to~$Z_{M}( H )$, which is contained in~$Z_{ G }( H )$, which is the stabiliser of~$\iota$. As~$b_i$ is bounded, we deduce that
%$$ v_i \cdot\iota$$
%is bounded.
%
%But the orbit of unipotent group in a linear representation is closed. Consequently the orbit map~$U/Z_U( H )\to V$ is a closed immersion, and is thus a proper map. 
%As a consequence,~$v_i \cdot\iota$ is bounded if and only if the sequence~$v_i$ is~$O(1)Z_U( H )$. 
%
%Finally, we can write~$ g _i$ as~$b_iO(1)Z_U(H)Z_L(H)=O(1)Z_{ G }$. 
This conclude the proof of the implication~$\eqref{P1}\Leftarrow\eqref{P2}$ in the general case, and conclude the proof of~Proposition~\ref{Propobddcrit}.

\end{proof}

Let us note the following corollary, which, in a way, summarise the part of information inside Proposition~\ref{Propobddcrit} which is the deepest.
\begin{corollary}\label{corostabilityrichardson} The orbit
$$ G \cdot \iota $$
is closed in~$\hat{V}$. Equivalently, the map
$$ \left. G \middle/{ H ^\prime}\right.\xrightarrow{g\mapsto g\cdot \iota}\hat{V}$$
is proper.
\end{corollary}
Our proof, relying on~\cite{Lemmanew}, actually uses convexity properties for symmetric spaces (and Bruhat-Tits buildings), similar to the
work of Kempf-Ness \cite{KempfNess}, with a generalised
Cartan decomposition of Mostow (and an analogue for buildings)

\subsubsection{Centralizer and cosets} We give some lemmas about centraliser and quotient groups. This is an adaptation of arguments which can be found
for instance in~\cite[{\S}5]{EMSAnn}.

\begin{lemma}Let~$G$ be a group of rational points of an algebraic group~$\underline{G}$ over a local field, and let~$H$ and~$L$ be two algebraic subgroups of~$G$.

 Write, after~\cite[{\S}5]{EMSAnn}
$$Z(H,L)=\left\{g\in G~\bigl|\forall h\in H,~[g,h]\in L\right\},$$
the ``\!\!\!$\pmod L$-centraliser'' of~$H$ in~$G$.

\begin{enumerate}
\item We remark that~$Z(H,L)$ is invariant by right translations under~$H^\prime$.
\item Assume~$L$ is normalised by~$H$. Then~$Z(H,L)$ is invariant by left translations under~$L$.
\item Under the latter assumption,~$Z(H,L)$ is moreover a finite union of double cosets~$LgH^\prime$. 
\end{enumerate}
\end{lemma}

\begin{proof} We can write the condition~$[g,h^{-1}]\in L$ as~${}^ghL=hL$. The element~$h$ is ``centralised by~$g$ mod~$L$'' (i.e., if~$L$
is normal in~$G$, the cosets~$gL$ and~$hL$ commute in~$G/L$).

If~$z$ belongs to~$H^\prime$, namely if~${}^zh=h$ for any~$h$ in~$H$, then
$${}^{gz}hL={}^{g}({}^zh)L={}^{g}hL=hL$$
for any~$g$ in~$Z(H,L)$. Hence the first remark in the conclusion.

If~$l$ belongs to~$L$, then, for any~$g$ in~$Z(H,L)$, and~$h$ in~$H$
$${}^{lg}hL={}^{l}({}^gh)L=l{}^{g}hl^{-1}L=l{}^{g}hL=lhL=h\cdot l^hL.$$
Hence~$l\cdot g$ belongs to~$Z(H,L)$ as soon as each of the~$l^h$ belong to~$L$. Hence the second remark in the conclusion.

We now prove the third point. Write~$\underline{L}$ the algebraic group associated with~$L$. We first note that condition~$\forall h\in H,~[g,h]\in \underline{L}$ defines an algebraic subvariety~$\underline{Z}(H,L)$ of~$G$, whose rational points are given by~$Z(H,L)$. As such, it is made of finitely many irreducible components. Consider an irreducible component~$\underline{V}$ of~$\underline{Z}(H,L)$. Write~$\underline{H^\prime},\underline{G}$ for the algebraic varieties associated with~$H^\prime$ and~$G$. We will show that for any point~$g$ in~$\underline{V}$, the map
$$\underline{L}\times \underline{H^\prime}\to \underline{G}:(l,z)\mapsto lgz$$
is submersive above~$g$. It will then show that it is submersive everywhere, and that~$\underline{V}$ is actually a double coset~$\underline{L}g\underline{H^\prime}$. We follow an argumentation from~\cite[{\S}5]{EMSAnn}.

\end{proof}

\paragraph{}
\begin{proposition} Let~$G$ be an algebraic group over a characteristic zero field, and~$N$ a normal
algebraic subgroup. Let~$\lie{h}$ be a reductive Lie subalgebra in that~$\lie{g}$ of~$G$, and 
denote~$\bar{\lie{h}}$ its image in that of~$G/N$. 

Then the induced map
\[\lie{h}^\prime\to\bar{\lie{h}}^\prime\]
from the centraliser of~$\lie{h}$ in~$G$ to the centraliser of~$\bar{\lie{h}}^\prime$ in~$G/N$
is submersive (equivalently, open).  
\end{proposition}
\paragraph{Centralisers at the Lie algebra level}
\begin{lemma}
Let
\begin{equation}\label{decompoLie}
\lie{g}=\lie{m}+\lie{l}
\end{equation}
be decomposition of a Lie algebra~$\lie{g}$ over some field
into a direct sum by two linear subspaces.

Let~$E$ be a subset of~$\lie{g}$ whose adjoint action preserve 
the decomposition~\eqref{decompoLie}:
\[\forall~e\in E,~[e,\lie{m}]\subseteq\lie{m}\text{ and }~[e,\lie{l}]\subseteq\lie{l}.\]
Then its centraliser~$E^\prime$
in~$\lie{g}$ decomposes accordingly into a direct sum
\begin{equation}\label{decompoLie}
E^\prime=E^\prime\cap\lie{m}+E^\prime\cap\lie{l}.
\end{equation}
\end{lemma}
\begin{proof}We prove the identity by double inclusion.
The equations defining~$E^\prime$ are
\[[X,e]=0,\]
 for~$e$ in~$E$. These are linear equations in~$X$, and hence~$E^\prime$ is a linear subspace, and consequently we have the 
 trivial inclusion
\[
E^\prime\supseteq E^\prime\cap\lie{m}+E^\prime\cap\lie{l}.
\] 
 Decomposing~$X=m+l$ with~$m\in\lie{m}$ and~$l\in\lie{l}$, we get~$[m,e]\in\lie{m}$ and~$[l,e]\in\lie{l}$
 by hypothesis on~$E$. As~$\lie{m}$ and~$\lie{l}$ are in direct sum, the equation~$[X,e]=[l+m,e]=[m,e]+[l,e]=0$ reduces to the conjunction of~$[m,e]=0$ and~$[l,e]=0$. In other words~$X\in E^\prime\cap\lie{m}+E^\prime\cap\lie{l}$. This proves the non trivial inclusion.
\end{proof}
\begin{corollary} Assume moreover~$\lie{l}$ is an ideal of~$\lie{g}$. We have a sort exact subsequence
\[0\to E^\prime\cap \lie{l} \to E^\prime \to (E\pmod{\lie{l}})^\prime\to 0\]
from
\[0\to\lie{l}\to\lie{g}\to\lie{g}/\lie{l}\to0\]
where~$(E\pmod{\lie{l}})^\prime$ is the centraliser in~$\lie{g}/\lie{l}$ of the image of~$E$.
\end{corollary}
\begin{proof} It is sufficient to check that~$(E\pmod{\lie{l}})^\prime$ is the image of~$E^\prime$ in~$\lie{g}/\lie{l}$.

\end{proof}
\subparagraph{Special cases}
\begin{itemize}
\item
Let~$\lie{h}$ be a subalgebra in a Levi factor~$\lie{m}$ of~$\lie{g}$ and take for~$\lie{l}$ the radical of~$\lie{g}$. Then the Levi decomposition
\[\lie{g}=\lie{m}+\lie{l}\]
is stable under~$\lie{h}$ (the radical is stable under~$\lie{g}$ and~$\lie{m}$ is stable under itself). We get
\[
\lie{h}^\prime= \lie{h}^\prime\cap\lie{m}+\lie{h}^\prime\cap\lie{l}.
\] 
\item More generally,~$\lie{l}$ can be any ideal of~$\lie{g}$ for which there is a supplementary Lie algebra~$\lie{m}$. This occurs for the Lie
algebra of a semidirect product Lie group~$G=M\times L$ and a Lie subgroup~$H\leq M$.
\item If~$\lie{h}$ is a Lie subalgebra which is \emph{reductive in~$\lie{g}$} (in characteristic~$0$, following~\cite[I{\S}{\S}6,8]{BBKLie1}) or more generally
which acts fully reducibly in~$\lie{g}$, then any~$\lie{h}$-invariant
subspace~$\lie{l}$ of~$\lie{g}$ admits a supplementary~$\lie{h}$-invariant subspace~$\lie{m}$. An example for~$\lie{l}$ is the radical of~$\lie{g}$, and, in characteristic~$0$, any Levi factor containing~$\lie{h}$ is an example for~$\lie{m}$.
\end{itemize}

\begin{lemma} Let~$\lie{g}$ be a Lie algebra over some field, denote~$\lie{u}$ its unipotent radical,
and let~$\lie{h}$ be a sub-algebra of~$\lie{g}$ whose adjoint action on~$\lie{g}$ is fully reducible.
(\cite{} says~$\lie{h}$ is “reductive in~$\lie{g}$” in a characteristic~$0$ setting.)

If~$\lie{m}$ be a Levi factor containing~$\lie{h}$
\end{lemma}

\begin{proposition}[Lifting centraliser of a reductive subgroup over a local field] Let~$G$ be a group of rational points of an algebraic group~$\underline{G}$ over a local field of characteristic~$0$,~$H$ a Zariski connected reductive subgroup in~$G$ and~$L$ a normal subgroup of~$G$.

Let~$H^\prime$ (resp.~$Z$) be the centraliser of~$H$ in~$G$ (resp. of~$HL/L$ in~$G/L$),
and~${H^\prime}^0$ (resp.~$Z^0$) its Zariski neutral component.

Then~$Z^0$ is the image of~${H^\prime}^0$ in~$G/L$.
\end{proposition}

\begin{proposition}Let~$G$ be a group of rational points of an algebraic group~$\underline{G}$ over a local field of characteristic~$0$, 
with normal algebraic subgroup~$N$, and an algebraic almost semidirect product decomposition
\[G=M\cdot N.\]
Let~$H$ be Zariski connected reductive subgroup in~$G$ normalising~$M$. Then the almost semidirect decomposition descend to its centraliser
\[H^\prime=(H^\prime \cap M)\cdot (H^\prime\cap N).\]
\end{proposition}

%\newpage
\appendix

% SECTION THOMAS
%\section{Review on Ratner's theorems and some variants}%\label{AppRatner}

\section{Linearisation. A Review in the $S$-arithmetic setting}\label{AppA}

The linearisation method of Dani and Margulis is the second essential ingredient to our proof. Specifically, our proof relies on an $S$-arithmetic version of Proposition~3.13 of \cite{EMSAnn} (refer to equation~\ref{eq0}). In this appendix, we prove this analogue, namely Proposition~\ref{AProp313}.

The proof we present follows the original proof of \cite{EMSAnn}, but makes use of simplifications introduced in \cite{KT} using the Besicovich covering property and $(C,\alpha)$-good functions. With these simplifications, the multidimensional case is treated equally to the one-dimensional case.

\subsection{Double fibration of $G/\Gamma_N$}
Recall that we have a double fibration of $G/\Gamma_N$ arising from the natural projection $\phi: G/\Gamma_N\to G/\Gamma$, and from the map $\eta_L:G\to V$ given by $g\mapsto g\cdot p_L$ (see~\ref{subsection linearisation}), which factors through a map $\bar{\eta}_L$ from $G/\Gamma_N$. 

Note that in our setting, the image of $\Gamma$ in $GL(V)(\Q_S)$ is commensurable with its intersection with $GL(V)(\Z[S^{-1}])$, and that $p_L$ is $\Q_S$-proportional to a rational vector. It follows that $\Gamma\cdot p_L$ is discrete. 

There are two important consequences for us. Firstly, $\phi\times\bar{\eta}_L:G/\Gamma_N\to G/\Gamma\times V$ is proper. Secondly, for any compact set $E$ of $V$, the function $\chi_E:G/\Gamma\to\Z_{\geq 0}$ given by $g\Gamma\mapsto\#\phi^{-1}(g\Gamma)\cap\bar{\eta}_L^{-1}(E)$ is upper semi-continuous (see~\eqref{semicontinu} and the argumentation there).

\subsection{Linearisation of singular sets} Following the terminology of~\cite[\S 3]{DM}, the subset of $(L,\Gamma_N)$-self-intersection of a set $X$ in $G$ is the set where $\phi$ restricted to $X\Gamma_N/\Gamma_N$ fails to be bijective.
%is said to have no $(L,\Gamma_N)$-self-intersection if the restriction of $\phi$ to $A\Gamma_N/\Gamma_N$ is bijective. 

%Borel Zariski density~\cite{BorelCRELLE} and $S$-arithmetic analogue (\cite[Lemma 3.1]{MargulisTomanov}).

\begin{proposition}[\emph{c.f.}~\cite{DM} Prop.~3.3] For any~$L$ in~$\Rat$, the set of points of~$(L,\Gamma_N)$-self-intersection of~$X(L,W)$ is contained in~$\bigcup_{\tilde{L}\in \Rat, \dim(\tilde{L})<\dim(L)} X(\tilde{L},W).$
\end{proposition}
\begin{proof} We follow the proof of~\cite{DM} Prop.~3.3. Let $\L\in\Rat$ and $\gamma\in\Gamma$. Since $L^{++}\Gamma$ and $\gamma L^{++}\Gamma$ are closed, so is $(L^{++}\cap\leftexp{\gamma}{L}^{++})\Gamma=(L\cap\leftexp{\gamma}{L})^{++}\Gamma$. 
Let~$\tilde{\L}$ be the~$\Q$-Zariski closure of~$(L\cap\leftexp{\gamma}{L})^{+}$.
Then it is a subgroup of class~$\Rat$. Remark that $\tilde{\L}$ is a $\Q$-subgroup of $\L\cap\leftexp{\gamma}{\L}$.

%Let $L_\gamma$ be the smallest closed subgroup of $L^{++}$ containing $(L\cap\leftexp{\gamma}{L})^{+}$ and such that $L_\gamma\Gamma$ is closed. Then $L_\gamma\cap\Gamma$ is a lattice in $L_\gamma$\footnote{This can be seen from Ratner's Orbit Closure Theorem [...], or by an $S$-arithmetic adaptation of [Shah Prop2.3]}. We can then apply [tomanov, thm3] to conclude that the Zariski closure $\tilde{\L}$ of $L_\gamma\cap\Gamma$ is of class $\Rat$ and is such that $L_\gamma$ is of finite index in $\tilde{L}=\tilde{\L}(\Q_S)$. Remark that $\tilde{\L}$ is a $\Q$-subgroup of $\L\cap\leftexp{\gamma}{\L}$.

If $g\in X(L,W)$ is a point of~$(L,\Gamma_N)$ intersection, then there exists $\gamma\in \Gamma\setminus \Gamma_N$ such that $g^{-1}Wg\subset L\cap \leftexp{\gamma}{L}$. Since $W=W^{+}$, in fact $g^{-1}Wg\subset (L\cap \leftexp{\gamma}{L})^+\subset L_\gamma\subset \tilde{L}$. Therefore, $g\in X(\tilde{L},W)$.

Finally, since $\gamma$ does not normalise $L$, the group $L\cap\leftexp{\gamma}{L}$ is a proper subgroup of $L$. Hence, $\tilde{\L}$ is a proper $\Q$-subgroup of $\L$. As $\L$ is $\Q$-connected, it follows that $\tilde{\L}$ is of stricly lower Krull dimension than $\L$ is. 
\end{proof}

Recall that $A_L$ is the $\Q_S$-submodule generated by $X(L,W)\cdot p_L$.
\begin{proposition}[\emph{c.f.}~\cite{DM} Cor.~3.5]\label{linearisation of neighbourhood}
Let~$D$ be a compact subset of~$A_L$. Let~$Y_H$ be the subset of $(L,\Gamma_N)$-self-intersection points of~$\eta_L^{-1}(D)$. Let~$K$ be a compact subset of~$G\smallsetminus Y_H\Gamma.$ Then there exists a neighbourhood~$\Phi$ of~$D$ in~$V$ such that~$\eta_L^{-1}(\Phi)\cap(K\Gamma)$ has no point of~$(L,\Gamma_N)$-self-intersection.
\end{proposition}
\begin{proof} Let $\mathfrak{N}$ be the collection of compact neighbourhoods of $D$ in $V$ and for $\Psi\in\mathfrak{N}$, let $U_\Psi$ denote the locus of $\chi_\Psi< 2$. By upper semi-continuity of $\chi_\Psi$, $U_\Psi$ is open.

Note that $\eta_L^{-1}(D)\cap(K\Gamma)$ has no point of $(L,\Gamma_N)$-self-intersection. In other words, for any $x\in K\Gamma/\Gamma$, $\chi_D(x)< 2$. Together with the fact that $\chi_D(x)=\inf\{\chi_\Psi(x)|\Psi\in\mathfrak{N}\}$, it implies that the collection $\{ U_\Psi|\Psi\in\mathfrak{N}\}$ covers the compact set $K\Gamma/\Gamma$. Pick a finite subcovering $\mathfrak{N}_f\subset\mathfrak{N}$, and let $\Phi=\cap_{\mathfrak{N}_f}\Psi$. As $\chi_\Phi\leq\chi_\Psi$ for any $\Psi\in\mathfrak{N}_f$, it follows that for any $x\in K\Gamma/\Gamma$, $\chi_\Phi(x)<2$. In other words, $\eta_L^{-1}(\Phi)\cap (K\Gamma)$ has no $(L,\Gamma_N)$-self-intersection.
\end{proof}

The proof above is essentially the same as the original proof of~\cite{DM}. However, our argument proves a little more, as the number $2$ could be any arbitrary number.

\subsection{Besicovich Property and Good Functions}%, after D.~Kleinbock, G.~A.~Margulis and G.~Tomanov}

%Here we recall, via references, the setting of~$(C,\alpha)$-good functions and the Besicovich property used in~\cite{KT}. We also refer to~\cite[\S~\,3 and~5]{LemmaA} which explain how to apply this setting to our Borel space~$(\Omega,\mu|_\Omega)$ inside~$H$.

Following~\cite{KT}, the Besicovich covering lemma and good functions are used in conjunction to cleanly generalize the one-dimensional real case of~\cite[Prop.3.13]{EMSAnn} to the multi-dimensional $S$-arithmetic setting.

In this appendix, any module $\prod_{v\in S}{\Q_v}^{d_v }$ for nonnegative integers $d_v$ is made in a measure metric space by using the metric induced from the max-norm of each component in $\Q_v$ and a measure $\lambda$ which is a product of Haar measures on $\Q_v$. So defined, it is doubling: there exists a constant $c_d$ such that for any ball $B$, 
\begin{equation}\label{doubling}
\lambda(3B)\leq c_d\lambda(B),
\end{equation}
where $3B$ denotes the ball with same center and three times the radius as $B$.

\subsubsection{Good Functions} Let $X$ be a measure metric space with a nowhere zero locally finite positive Borel measure~$\nu$. We fix two constants $C>0$ and $\alpha>0$. For a function $f:X\to \R$ and a ball~$B$ in~$X$, we denote $\Nm{f}_B$ the supremum of~$\abs{f}$ on~$B$. Then a nonzero real continuous function~$f$ on~$X$ is~\emph{$(C,\alpha)$-good} if for any ball~$B$ in~$X$ and for any $\varepsilon>0$,
\begin{equation}\label{def good} \nu\left(\left\{ x\in B\,\middle|\, \abs{f(x)}<\varepsilon \Nm{f}_B\right\}\right) \leq C\varepsilon^\alpha\nu(B).
\end{equation}
We remark that due to the strict inequality on the left hand side, the zero function is $(C,\alpha)$-good. 

%From a dynamical point of view, $(C,\alpha)$-goodness quantifies the notion that on time intervals, if the function is very close to zero for a significant portion of the time, then it is close to zero \emph{all} the time. %balls, the set where the function is relatively small has itself relatively small measure.
%
%Due to the lack of a global parametrisation of $\Omega$ in the ultrametric case, we will however need to consider a slightly modified class of subsets
% 
%For integers $d_v$ and $e_v$, a map $\phi:\prod_{v\in S}{\Q_v}^{d_v }\to\prod_{v\in S}{\Q_v}^{e_v }$ is analytic if each component $\phi_v:\Q_v^{d_v}\to\Q_v^{e_v}$ is analytic. An analytic function on a bounded set is good, and we will need that composing by linear maps preserves its goodness. 

\begin{proposition}\label{goodness} %Let $\rho:G\to GL(V)(\Q_S)$ be a finite dimensional Lie group representation of $G$ and let $\Nm{\cdot}$ be a norm on $V(\Q_S)$. Then there exist constants $C$ and $\alpha$ such that  for any endomorphism $\Lambda\in\operatorname{End}(V)(\Q_S)$, the function $\Nm{\Lambda\circ\rho}$ is $(C,\alpha)$-good on $\big(\rho(\Omega),\rho_*\mu\big)$.

 Let $U\subset\prod_{v\in S}{\Q_v}^{d_v }$ be a bounded neighbourhood of zero and let $\phi:U\to \prod_{v\in S}{\Q_v}^{e_v }$ be analytic. Then there exist $C$, $\alpha$ and a neighbourhood $V\subset U$ such that for any $\Q_S$-linear map $\Lambda:\prod_{v\in S}{\Q_v}^{e_v }\to\prod_{v\in S}{\Q_v}^{f_v }$, the function $\Nm{\Lambda\circ\phi}$ is $(C,\alpha)$-good on $V$.
\end{proposition}

\begin{proof} Since the norm is the max-norm and the supremum of $(C,\alpha)$-good functions is itself $(C,\alpha)$-good, it is enough to consider a single place $S=\{v\}$ and linear functionals $\Lambda:\Q_v^{e}\to\Q_v$.
% For $1\leq i\leq f_v$, let $\pi_{v,i}$ be the projection of $\Q_v^{e_v}$ onto the $i^{\text{th}}$-component, extended to be zero at other places. Note that $\Nm{L\circ\phi}=\max\abs{\pi_{v,i}\circ L\circ\phi}_v$.

Let $\phi=(\phi_1,...,\phi_{e})$. If $\phi\equiv 0$, then we are done since the zero function is $(C,\alpha)$-good for any constants $C$ and $\alpha$. Otherwise, restricting the maps to $\operatorname{span}\phi(U)$, we can assume that $\phi(U)$ spans $\Q_v^{e }$. Since $\phi$ is analytic, it implies that it is nondegenerate at zero. By~\cite[Thm.4.3]{KT},  there are $C$, $\alpha$ and a neighbourhood $V\subset U$ of zero such that for any constants $c_i\in\Q_v$, $\abs{c_1\phi_1+...+c_e\phi_e}$ is $(C,\alpha)$-good on $V$. In other words, for any linear functional $\Lambda:\Q_v^{e}\to\Q_v$, $\abs{\Lambda\circ\phi}$ is $(C,\alpha)$-good on $V$, finishing the proof.
%First consider linear functionals $L:\Q_v^{e_v}\to\Q_v$, extended to be zero at the other places. Since $\phi$ is analytic, $L\circ\phi$ is non degenerate at zero if and only if it is nonzero. By~\cite[Thm.4.3]{KT}, there exists $C$ and $\alpha$ and a neighbourhood $V\subset U$ of zero on which $L\circ\phi$ is $(C,\alpha)$-good.
%
% By the previous paragraph, there exists a neighbourhood $V_{v,i}\subset U$ and constants $(C_{v,i},\alpha_{v,i})$ such that $\abs{\pi_{v,i}\circ L\circ\phi}_v$ is $(C_{v,i},\alpha_{v,i})$-good on $V_{v,i}$.  
%
\end{proof}
%\paragraph{Good functions --  {\cite[\S~\,2]{KT}}} Let~$X$ be a metric space, with a nowhere zero locally finite positive Borel measure~$\mu$, and fix constants~$C,\alpha>0$. A real continuous function:~$X\to \R$ is~$(C,\alpha)$-good if for any ball~$B$ in~$X$,
%$$\forall \varepsilon >0, \mu\left(\left\{ x\in B\middle| \abs{f(x)}<\varepsilon \right\}\right) \leq C\left(\frac{\varepsilon}{\Nm{f|_B}}\right)^\alpha\mu(B).$$
%In rough words, this properties implies that a function which is small on a ball on big subset, is actually uniformly small. A big subset means a subset containing enough proportion of the measure.

\subsubsection{Good Parametrisation}\label{parametrisation} In the presence of non-archimedean places, a good parametrisation of $\Omega$ would exist only locally, due to the lack of convergence of the exponential map. However, since it is bounded, a given local parametrisation has positive measure, and this will be enough for our purposes.

  Let~$H$ be a Lie subgroup of~$G=\G(\Q_S)$\footnote{i.e. locally a product of Lie subgroups of $\G(\Q_v)$ at each places of $S$}. In particular, it is locally analytic. We pick any coordinate chart $\Theta:U\to H$ around a point of $h\in\Omega$, where $U$ is a neighbourhood of $0$ in $\prod_{v\in S}{\Q_v}^{d_v }$ for nonnegative integers $d_v$ for each $v\in S$. Also, for any representation $\rho:G\to GL(V)(\Q_S)$ of Lie groups, possibly shrinking $U$, the map $\rho\circ\Theta$ is analytic on $U$.

We fix such a representation $\rho$ of $G$ and a neighbourhood $U$ on which $\rho\circ\Theta$ is analytic. There is a smaller neighbourhood $V$ on which Proposition~\ref{goodness} holds. Let $B$ be a ball centered at zero and of radius $r$ small enough such that:
\begin{enumerate}
\item the ball $3B$ with same center and three times the radius is contained in $V$,
\item the projection $G\to G/\Gamma$ is injective on $\Theta(B)$.
\item $\Theta(B)$ is contained in $\Omega$.
\end{enumerate}
Then $\Theta$ provides our desired parametrisation from $B$ onto a bounded open set $\Omega_0=\Theta(B)$ in $H$.

The measure $\lambda$ on $\prod_{v\in S}{\Q_v}^{d_v }$ is associated to an invariant top differential form $\omega$. Then $\Theta_*\omega|_{B}$ is a continuous top differential form on $\Omega_0\subset H$. Therefore, the associated measure $\Theta_*\lambda|_{B}$ is absolutely continuous with respect to any Haar measure on $H$ restricted to $\Omega_0$, hence in particular to the probability measure $\mu|_{\Omega_0}$. Moreover, since $\Omega_0$ is bounded, both measures are comparable: 
\begin{equation}\label{comp meas}
\exists\, c_m\geq 1\text{ such that } c_m^{-1}\mu|_{\Omega_0}\leq \Theta_*\lambda|_{B}\leq c_m\,\mu|_{\Omega_0}.
\end{equation}

%\paragraph{Good parametrisation -- \cite[\S~\,3 -- 5]{LemmaA}} Let~$H$ be a Lie subgroup of~$G=\G(\Q_S)$, with Lie algebra~$\h=\prod_{v\in S}\h_v.$ Endow~$\h$ with a product norm, write~$0$ for the origin, and~$B(0,r)$ for the open ball of radius~$r$ centred at the origin. Note that forsmall enough~$r$,  the exponential map~$\g$ is normally convergent on~$B(0,3r)$, and injective and immersive from~$B(0,r)\to G$.
%
%Write~$\Omega_r=\exp(B(0,r)\cap \h)$. Then, for small enough~$r$, it is an open bounded subset of~$H$. 

\subsubsection{Besicovich Covering Property} It is proved in~\cite[\S~\,1.1]{KT} that $X=\prod_{v\in S}{\Q_v}^{d_v }$ (with the max-norm metric) satisfies the Besicovich covering property: for any bounded subset~$A$ of~$X$, and for any family~$\mathcal{B}_0$ of nonempty open balls such that any point in~$A$ is the centre of some ball of~$\mathcal{B}_0$, there is an at most countable subfamily~$\mathcal{B}\subseteq\mathcal{B}_0$ which covers~$A$ and has multiplicity at most~$N_X$: in terms of characteristic functions, 
\begin{equation}\label{besicovich}
 1_A \leq \sum_{B\in\mathcal{B}}1_B\leq N_X\cdot1_X.
 \end{equation}

%\paragraph{Besicovich covering property -- {\cite[\S~\,1]{KT}}} After~\cite[\S~\,1.1]{KT}, a metric space~$X$ is \emph{Besicovich} if there is a constant~$N_X$  such that the following hold: for any bounded subset~$A$ of~$X$, and for any family~$\mathcal{B}_0$ of nonempty open balls, in~$X$ such that
%$$ \text{any point in~$A$ is the centre of some ball of~$\mathcal{B}$},$$
%there is a finite or countable subfamily~$\mathcal{B}\subseteq\mathcal{B}_0$ which covers~$A$ and has multiplicity at most~$N_X$ everywhere on~$X$: in terms of characteristic functions, 
%$$ 1_A \leq \sum_{B\in\mathcal{B}}1_B\leq N_X\cdot1_X.$$
%
%We use the following: with the product metric,~$\Q_S$ is a Besicovich metric space. More generally: for any finite dimensions~$\left(d_v\right)_{v\in S}$, the space~$\prod_{v\in S}{\Q_v}^{d_v }$, is Besicovich (with ? metric, \cite[\S~\,1.6]{KT}).

\subsection{Linearisation of focusing, after A.~Eskin, S.~Mozes and N.~Shah}

\begin{proposition}[\emph{c.f.}~{\cite[Prop.~3.8, Prop.~3.12]{EMSAnn}}]\label{EMS312}Let $K$ be a compact set contained in $X^*(L,W)\Gamma/\Gamma$ and let $\eps>0$. For any compact $D_0\subset A_L$, there exists a compact $D\subset A_L$ such that for any neighbourhood $\Phi$ of $D$ in $V$, there exists a neighbourhood $\Psi$ of $D_0$ in $V$ such that for any $g\in G$, at least one of the following holds:
\begin{enumerate}
\item There exists $\gamma\in\Gamma$ such that $g\Omega_0\gamma\cdot p_L\subset \Phi$.
\item $\mu\left( \left\{\omega \in\Omega_0 \bigl| g\cdot\omega\Gamma \in K, g\cdot\omega\Gamma\cap\Psi\neq\emptyset \right\}\right)<\eps.$
\end{enumerate}
\end{proposition}

\begin{proof} %Choose a finite collection $\mathcal{L}$ of $\Q_S$-linear functionals with values in a single $\Q_p$ such that 
%\[ \bigcap_{f\in\mathcal{L}}f^{-1}(0) = A_L. \]
Let $\Lambda$ be a $\Q_S$-linear map with kernel exactly $A_L$. Choose $M$ a large real number to be fixed later and $R>\Nm{D_0}$. Let 
\[ D=\{v\in A_L\ |\ \Nm{v}\leq MR\}\]
and let $\Phi$ be a given neighbourhood of $D$. If the conclusion of the proposition holds for a neighbourhood of $D$ in $V$, then it necessarily holds for any larger neighbourhood. Thus, possibly shrinking $\Phi$, we can assume by Proposition~\ref{linearisation of neighbourhood} that for any $y\in K$, $\chi_\Phi(y)\leq 1$. Since $\Lambda$ is continuous, there is a constant $b$ such that
\[ \{v\in V\ |\ \Nm{v}\leq MR\text{ and }\Nm{\Lambda(v)}<b\}\subset \Phi.\] 
%\[ \{v\in V\ |\ \Nm{v}\leq MR\text{ and }\forall f\in\mathcal{L},\ |f(v)|<b\}\subset \Phi.\] 
Now define a neighbourhood $\Psi$ of $D_0$ to be
\[ \Psi = \{v\in V\ |\ \Nm{v}< R\text{ and }\Nm{\Lambda(v)}<\frac{b}{M}\}.\]
%\[ \Psi = \{v\in V\ |\ \Nm{v}\leq R\text{ and }\forall f\in\mathcal{L},\ |f(v)|<\frac{b}{M}\}.\]
We have to show that for these choices of $D$, $\Phi$ and $\Psi$, the alternative is exhaustive. To do so, let $\rho$ be the adjoint representation of $G$ on its Lie algebra $\mathfrak{g}$, and let $\Theta:B\to\Omega_0$ be a parametrisation associated to $\rho$ as in \S\ref{parametrisation}.
%$(\operatorname{dim}L)^{\text{th}}$-exterior power of the 
For $g\in G$, define 
\[ E=\{t\in B\ |\ \exists\gamma\in\Gamma,\ g\Theta(t)\gamma\Gamma/\Gamma\in K\text{ and }g\Theta(t)\gamma\cdot p_L\in\Psi\}.\]
Since $\Psi\subset\Phi$ and $\chi_\Phi(y)\leq 1$ on $K$, the element $\gamma$ appearing in the definition of $E$ is unique. For $s\in E$, let us denote it $\gamma_s$. %Let $3B$ denotes the ball with same center as $B$ and radius three times the one of $B$. 
For $s\in E$, denote $A_s$ the largest ball in $3B$ centered at $s$ such that for all $t\in A_s$, $g\Theta(t)\gamma_s\cdot p_L\in\Phi$.  

If there exists $s\in E$ such that the boundary $\partial A_s$ of $A_s$ intersects the boundary $\partial (3B)$ of $3B$, then $A_s$ contains $B$, and we are in the first alternative of the proposition: $g\Omega_0\gamma_s\cdot p_L\subset\Phi$. Therefore, we assume it is not the case: the closure of $A_s$ is contained in $3B$ and for any $s\in E$, there exists $t\in\partial A_s$ such that $g\Theta(t)\gamma_s\cdot p_L\in\partial\Phi$.

%\[ \mu(\Theta(E))\leq c\lambda(E)\leq\sum_{s\in E} \lambda(A_s\cap E).\] 

Now consider for any $\gamma\in G$ the following maps:
\begin{align*}
T_{\gamma}&: t\mapsto\Lambda(g\Theta(t)\gamma\cdot p_L)\text{ and}\\
N_{\gamma}&:t\mapsto g\Theta(t)\gamma\cdot p_L.
\end{align*}
Note that these maps are obtained from $\rho\circ\Theta$ by post-composing by linear maps into~$V$.
From our choice of $\Theta$ and from Proposition~\ref{goodness}, there exist constants $(C,\alpha)$ such that~$\Nm{T_\gamma}$ and $\Nm{N_\gamma}$ are all $(C,\alpha)$-good on $3B$. 

Let $s\in E$. The set $\Phi$ was so defined that $g\Theta(t)\gamma_s\cdot p_L\in\partial\Phi$ implies that either $\sup_{A_s}T_{\gamma}\geq b$ or $\sup_{A_s}N_\gamma \geq MR$. On the other hand, the set $\Psi$ was so defined that $g\Theta(t)\gamma_s\cdot p_L\in\Psi$ implies that $\sup_{A_s\cap E}T_{\gamma}\leq \frac{b}{M}$ and $\sup_{A_s\cap E}N_\gamma \leq R$. Therefore, by the $(C,\alpha)$-good property,
\begin{equation}\label{eq good} \text{for any }s\in E\quad \lambda(A_s\cap E) \leq CM^{-\alpha}\lambda(A_s).\end{equation}

 By the Besicovitch Property~\eqref{besicovich}, there exists a constant $N$ and a $E'\subset E$ such that $1_E\leq \sum_{s\in E'} 1_{A_s}\leq N$. Therefore,
\begin{equation}\label{eq besicovich} \lambda(E)\leq \sum_{s\in E'} \lambda(A_s\cap E)\quad\text{and}\quad \sum_{s\in E'} \lambda(A_s)\leq N\lambda(\cup_{s\in E'} A_s)\leq N\lambda(3B).\end{equation}

%Recall from~\eqref{comp meas}  that there is a constant $c$ such that $\mu(\Theta(E))\leq c\lambda(E)$.
Combining equations~\eqref{doubling}, \eqref{comp meas}, \eqref{eq good} and \eqref{eq besicovich}, we obtain
\begin{align*}
\mu(\Theta(E))&\leq c_m\lambda(E)\leq c_m\sum_{s\in E'} \lambda(A_s\cap E)\\ 
&\leq c_mCM^{-\alpha}\sum_{s\in E'}\lambda(A_s)\leq c_mNCM^{-\alpha}\lambda(3B)\\ 
&\leq c_dc_mNCM^{-\alpha}\lambda(B).
\end{align*}
Choosing $M$ greater than $(\varepsilon^{-1}c_dc_mNC\lambda(B))^{1/\alpha}$ yields the second alternative of the proposition:
\[  \mu(\Theta(E))= \mu\left( \left\{\omega \in\Omega_0\, \bigl|\, g\cdot\omega\Gamma \in K,\, g\cdot\omega\Gamma\cap\Psi\neq\emptyset \right\}\right)<\eps.\]
This completes the proof.
\end{proof}

\begin{proposition}[\emph{c.f.}~{\cite[Prop.~3.13]{EMSAnn}}]\label{AProp313}
Assume~$\mu_\infty^{[L]}$ is concentrated on the Borel set $X^*(L,W)\Gamma/\Gamma$, that is 
%$\mu_\infty\left(X(L,W)\Gamma/\Gamma\right)=\mu_\infty\left( X^*(L,W)\Gamma/\Gamma\right)$.
$\mu_\infty\left(X(L,W)\Gamma/\Gamma\right)=\mu_\infty\left( X^*(L,W)\Gamma/\Gamma\right)$.
 
Then there is a compact subset~$D$ in~$A_L$, and a sequence~$\left(\gamma_i\right)_{i\geq0}$ in~$\Gamma$, such that, for any neighbourhood~$\Phi$ of~$D$ in~$V$, 
\[\forall i\gg 0,\ g_i\cdot\Omega\cdot\gamma_i\cdot p_L\subseteq \Phi.\]
\end{proposition}

\begin{proof} Fix a parametrisation $\Theta: B\to\Omega_0\subseteq\Omega$ associated to the exterior power of the adjoint representation as in \S\ref{parametrisation}. Note that $X^*(L,W)\Gamma/\Gamma$ is a Polish space, as it is the complement of a countable union of closed sets of the closed set $X(L,W)\Gamma/\Gamma$. Therefore, for any arbitrarily small $\delta>0$, there exists a compact set $K\subset X^*(L,W)\Gamma/\Gamma$ such that $\mu_\infty(K)>1-\delta$. We can lift $K$ to a compact set $K'$ in $X^*(L,W)$. Let $D_0=K'\cdot p_L$, which is a compact subset of $A_L\subset V$.

%Fix a parametrisation $\Theta: B\to\Omega_0$ as in ... Pick finitely many balls $B_1,...,B_r$ contained in $B$ and  finitely many elements $h_1,...,h_r$ in $H$ such that the sets $h_i\Theta(B_i)$, $1\leq i\leq r$, are disjoint subsets of $\Omega$ covering it up to a set of measure less than $\delta$.

Choose~$0<\eps<\mu(\Omega_0)$ and a compact set $D\subset A_L$, let $\Phi$ be an arbitrary neighbourhood of $D$ in $V$ and choose a neighbourhood $\Psi$ of $D_0$ in $V$ such that Proposition~\ref{EMS312} holds for $\Omega_0$, $K$ and $\eps$. Furthermore, $\Psi$ is chosen small enough so that the neighbourhood $\phi\circ\eta^{-1}(\Psi)$ of $K$ is disjoint from $S(D)$ (the locus of $\chi_D\geq 2$). 
%As $\Psi$ can be chosen arbitrarily small and $K$ is a compact set disjoint from the closed locus $S(D)$ of $\chi_D\geq 2$, we can assume the neighbourhood $\pi\circ\phi^{-1}(\Psi)$ of $K$ is also disjoint from $S(D)$. 

Suppose for contradiction that for infinitely many $i>0$, the second case of Proposition~\ref{EMS312} holds: $\mu\left( \left\{\omega \in\Omega_0 \bigl| g_i\omega\Gamma \in K, g_i\omega\Gamma\cdot p_L\cap\Psi\neq\emptyset \right\}\right)<\eps.$ Then 
\[ \mu\left( \left\{\omega \in\Omega \bigl| g_i\omega\Gamma \in K, g_i\omega\Gamma\cdot p_L\cap\Psi\neq\emptyset \right\}\right)<\eps+\mu(\Omega)-\mu(\Omega_0)=1+\eps-\mu(\Omega_0).\]
%The set $\pi\circ\phi^{-1}(\Psi)$ is a neighbourhood of the compact set $K$, the latter being disjoint from the locus $S(D)$ of $\chi_D\geq 2$. From ..., $S(D)$ is also closed. Therefore, shrinking $\Psi$ if necessary, we can assume $\pi\circ\phi^{-1}(\Psi)$ is disjoint from $S(D)$. 
Hence, for infinitely many $i>0$, $g_i\mu_{\Omega}(\phi\circ\eta^{-1}(\Psi))<1+\eps-\mu(\Omega_0)$. This however contradicts $g_i\mu_\Omega \to\mu_\infty$ and $\mu_\infty(K)>1-\delta$, whenever $\delta<\mu(\Omega_0)-\eps$.

Therefore the first case of Proposition~\ref{EMS312} holds: $\forall i\gg 0,\ g_i\cdot\Omega_0\cdot\gamma_i\cdot p_L\subseteq \Phi$. In other words, the limit points of $g_i\Omega_0\gamma_i\cdot p_L$ are contained in $D$. However, as $\Omega$ is included in the span of $\Omega_0$ in $\operatorname{End}(V)$ and $\Omega$ is bounded, the limit points of $g_i\Omega\gamma_i\cdot p_L$ are contained in a possibly larger compact subset $D$ of the linear space $A_L$, concluding the proof.

%Therefore for all but finitely many $i>0$, the first case of Proposition~\ref{EMS312} holds, which is the conclusion of the proposition.
\end{proof}

\section{Review on Ratner's theorems and some variants}\label{AppRatner}

Confusingly, Ratner's theorem can refer to various theorems, depending on the author. These theorems have variants, in different degree of generality. And, there is more, a same statement may be known under different names. 

Hopefully, the review article~\cite{RatnerICM} explicits and precises many such statements; with comment on the variants found in the literature. (And some insight on the history of their proof, their evolution from weaker theorem, the logical links between the statements, etc.)

Generally speaking, \emph{Ratner's theory} can be understood as 
$$\text{``rigidity'' properties of ``unipotent'' ``flows'' on ``homogeneous spaces''.}$$
Let us explain these terms. \emph{Rigidity} is an informal term (cf.~\cite{WhatisMeasureRig}) which means here that objects of a (``richer'') algebraic nature (the homogeneous subsets) are actually ubiquitous at the a priori ``weaker'' topological level (questions about orbit closure), or the ``even weaker'' measure-theoretic level (ergodicity of finite Borel measures).\footnote{Cartan's theorem about analyticity in Lie group theory, or Margulis arithmeticity theorem, Mostow-Prasad rigidity theorem may bring a grasp on the rigidity phenomenon.} 
 \emph{Unipotent group} can correspond to various kind of groups: one parameter, monogenous, higher dimensional algebraic unipotent, $\Ad$-unipotent, connected, disconnected, discontinuous, generated by unipotents or~$\Ad$-unipotent. \emph{Flows} is a dynamical image to indicate the underlying action (of the unipotent group), which may be more suitable for one parameter group (the parameter being seen as a time parameter). An \emph{Homogeneous space}~$G/\Gamma$ is understood in the theory of topological groups~$G$ and~$\Gamma$, but belonging to some quite general class, which encompass at least the connected real Lie groups, but with more general groups involved as the $S$-arithmetic regular Lie groups of~\cite{RatnerICM}; the stabiliser~$\Gamma$ is often chosen to be a lattice (sometimes assumed to be arithmetic, \cite{TomanovOrbits}), but not always (see~\cite{MT3}.)

Mostly three kind of theorems are pinpointed.
\begin{enumerate}
\item \emph{Ratner's orbit closure}. A topological kind of theorem, about the closure of orbits of unipotent groups, which is the original Raghunathan's conjecture. (\cite[Conjecture~1, Theorem~3, {\S}4 Theorem~S2]{RatnerICM})
\item \emph{Ratner's classification}. A measure-theoretic or ergodic-theoretic kind, about the algebricity of the measures with unipotent invariance. This is sometimes known as Raghunathan's measure conjecture, as Ratner's (strict) measure rigidity, or as Ratner's classification theorem.(\cite[Conjectures~2,3, Theorems~1,2, {\S}4 Theorem~S1]{RatnerICM} see also~\cite[Theorems~10,13,14, Remark after Theorems~14]{RatnerICM})
\item \emph{Ratner's equidistribution}. A dynamic kind, about the asymptotic distribution of a unipotent trajectory in the closure of their orbit. This is sometimes known as Ratner's uniform distribution theorem, or as distribution rigidity. One also refers to uniform distribution using the term ``equidistribution''. (\cite[Theorems~6,8, {\S}4 Theorem~S3]{RatnerICM} see also~\cite[Theorem~10]{RatnerICM}))
\end{enumerate} 
 
In classification statements, algebricity of a probability measure~$\mu$ has a group theoretic meaning, namely, in an homogeneous space~$G/\Gamma$:
$$ \text{the support of~$\mu$ is a closed orbit of a closed group~$L$ of~$G$ under which~$\mu$ is invariant.} $$
This ``algebraicity" is sometimes refered as ``homogeneity" by recent authors. This notion is implicit. It is sometimes useful to have a more explicit descriptions of the algebraic measures, namely of the groups~$L$ which can be involved in the statement above. In particular in the case~$\Gamma$ is an arithmetic or $S$-arithmetic lattice (when this definition makes sense for the considered group~$G$), we expect these groups to be algebraic and defined over~$\Q$. This can be found in~\cite[Proposition~3.2]{Shah91} and~\cite[Theorems~1,~2]{TomanovOrbits}.

Here, we will mostly rely on the measure-theoretic classification theorem, in the~$S$-arithmetic setting for algebraic groups, as can be found in~\cite{MargulisTomanov}. More particularly for groups generated by unipotents as in~\cite[Theorem~2]{MargulisTomanov}.

In the remaining of this appendix~\ref{AppRatner}, we will add some complement to~\cite[Theorem~2]{MargulisTomanov}. We need to explicit the class of algebraic measures involved in~\cite{MargulisTomanov}, in terms of algebraic group; we were unable to find such precisions in the available litterature, when the context of~\cite[Theorem~2]{MargulisTomanov} is concerned. In the archimedean case~$\Q_S\ciso\R$, these complements are more simply stated, and well known. The knowledgeable reader interested in the archimedean case only can skip what remains of this appendix.

\subsection{On groups of Ratner class in the $S$-arithmetic case} When applying the classification theorem, one faces the following class of subgroups. Recall that, for a subgroup~$L$ of~$G$, we denote by~$L^+$ the group generated by the algebraic unipotent subgroup contained in~$L$. 

\begin{definition}	 Let~$\Lscr$ be the class of closed subgroups~$L$ of~$G$ fulfilling the following conditions.
\begin{enumerate}
\item[1a.] The subset~$L\cdot\Gamma$ is closed. Equivalently, the orbit~$L\cdot\Gamma/\Gamma$ is closed in~$G/\Gamma$.
\item[1b.] (stronger than 1a) The intersection~$\Gamma_L:=\Gamma\cap L$ is a lattice in~$L$. Namely: the orbit~$L\cdot\Gamma/\Gamma$ is the support of a~$L$-invariant probability measure. We will denote~$\mu_L$ the mentioned probability measure (though it depends on~$\Gamma$ as well). 
\item[2.] (assuming 1b.) The probability~$\mu_L$ is ergodic under the action of~$L^+$.
\end{enumerate}
Note that the condition~2. is actually equivalent to the seemingly weaker one: there exists an algebraic unipotent subgroup in~$L$ which acts ergodically on~$\mu_L$.
\end{definition}

From now on, we consider only the case ofan $S$-arithmetic lattices~$\Gamma$ in an algebraic group over~$\Q_S$. In such case the previous class~$\Lscr$ is closely related with the following more explicit, and algebraically defined, one. 
\begin{definition}Let~$\RatQ$ be the class of algebraic subgroups~$\L$ of~$G$ defined over~$\Q$ such the following conditions.
\begin{enumerate}
\item The algebraic group~$\L$ is Zariski connected over~$\Q$.
\item The radical of~$\L$ is unipotent. Equivalently,~$\L$ admits a semi-simple Levi factor (instead of reductive). Equivalently,~$\L$ admits no character, absolutely.
\item The Levi factors of~$\L$ over~$\Q$ are of non compact type: none of their quasi factor defined over~$\Q$ is anisotropic over~$\Q_S$.
\end{enumerate}
\end{definition}
From Levi decomposition for connected linear algebraic groups, one easily check that the conjonction of these three conditions is actually equivalent to the following.
\begin{equation}
\text{The subgroup~$\L(\Q_S)^+$ is $\Q$-Zariski dense in~$\L$.}
\end{equation}

The relation between the two classes is the following. One passes from~$\Lscr$ to~$\RatQ$ by associating to~$L$ its Zariski closure (over~$\Q$) of a neighbourhood?. Conversely, one passes from~$\RatQ$ to~$\Lscr$ by the following construction. Write~$\Gamma_\L$ for~$\Gamma\cap\L(\Q_S)$, and write~$\L^+$ for~$\L(\Q_S)^+$. Then~$\L^+$ is a subgroup of~$\L(\Q_S)$ invariant under all algebraic automorphisms. It is then a normal subgroup, and is normalised by~$\Gamma_\L$. Hence~$\Gamma_\L\cdot\L^+$ is a group. We write~$\\Lpp$ for its closure. It is a group of class~$\Lscr$.

\begin{proposition} Let~$W$ be a subgroup of~$G$ such that~$W^+=W$. Then the Zariski closure of~$W$ over~$\Q$ is of class~$\Rat$.

Every group of class~$\Rat$ can be achieved in such a way.
\end{proposition}

\begin{proposition}\label{++finiteindex}
 The group~$\\Lpp$ is closed, open, Zariski dense and of finite index in~$\L(\Q_S)$; it is of class~$\Lscr$.
\end{proposition}
Lets prove the first statement.
\begin{proposition} Let~$\L$ be a connected linear algebraic group over~$\Q$. Consider~$L=\L(\Q_S)$, and, for an arbitrary~$S$-arithmetic lattice~$\Gamma$ in~$L$,
\[
\Lpp=\overline{\Gamma_L\cdot L^+}.
\]
Assume~$\L$ is of non compact type and its radical is unipotent: equivalently
\begin{itemize}
\item that~$\Gamma_L$ is $\Q_S$-Zariski dense in~$L$; (Borel-Wang density~\cite{Wang})
\item that~$L^+$ is $\Q$-Zariski in~$\L$.
\end{itemize}

Then~$\Lpp$ is an open subgroup of finite index in~$L$.
\end{proposition}
\begin{proof}[Sketch of proof] By construction~$\Lpp$ is closed. It is a subgroup of~$L$ as~$L^+$ is normal in~$L$. 
It suffices to prove the finiteness of the index. The openness will follow.

Without loss of generality, we can replace~$\Gamma_\L$ by a subgroup of finite index. 

Note that the radical~$R$ of~$\L$ is unipotent, hence is contained in~$L^+$. 
Consider the projection~$\L\to\L/R$. The image of~$\Gamma_\L$ is a~$S$-arithmetic subgroup of~$\L/R$. 
The induced map~$L\to L/R(\Q_S)$ is surjective (additive Hilbert~90 for a perfect field~\cite[II~Prop.1, III~Prop.6]{SerreLNM5}.)
We can reduce to the case~$\L$ is semisimple. And even adjoint. And even $\Q$-simple.

By~\cite[Proof of Proposition~6.14]{BorelTits},~$L^+$ is open and of finite index in the
isotropic $\Q_S$-factors of~$L$. It suffices to prove that the projection of~$\Gamma$ on the 
anisotropic $\Q_S$-factors has a closure which is a finite index subgroup~$K$. The Lie algebra
of this closure is invariant under~$\Gamma$, hence under~$L$ by Borel-Wang Zariski density.

It need to be a product of Lie algebras of some anisotropic factors.
But this projection is non trivial and non discrete on every anisotropic factor, again by Zariski density.

The Lie algebra of~$K$ is that of the product of the anisotropic factors.
Hence~$K$ is open in the product of the anisotropic factors. The latter is compact.
Hence~$K$ is also of finite index.
\end{proof}
Lets prove the second statement.
\begin{proof} Find references.

\end{proof}

\begin{proposition}[Kneser-Tits for characteristic~$0$ local fields, \cite{TitsBBK}{\cite[2.3.1 (d)]{Margulis}}] Let~$\H$ be an algebraic simply connected semi-simple group over~$\Q_S$ without~$\Q_S$-anisotropic quasi-factor.

 Then~$\H(\Q_S)=\H(\Q_S)^+$.
\end{proposition}
%\begin{corollaire} If~$H$ is no more necessarily simply connected,~$\H(\Q_S)^+$ is open of finite index in~$\H(\Q_S)$.\end{corollaire}
\begin{proposition}[Borel Zariski density, $S$-arithmetic variant of {\cite{Wang}, \cite[{\S}I.3.2]{Margulis}}]\label{BorelWang} Let~$\H$ be an algebraic simply connected semi-simple group over~$\Q$ without~$\Q_S$-anisotropic $\Q$-quasi-factor.

Then any $S$-arithmetic subgroup is a Zariski dense lattice.
\end{proposition}

Let's define the Ratner class as lattice + unipotent acts ergodically. Arithmetic case implies open in zariski closure, and even finite index (depending on Gamma for anisotropic factors). (Kneser tits simplement connexe? facteur localement anisotropes: image de Gamma)

%\subsection{Complement to~{\cite[Theorem~2]{MargulisTomanov} DONE IN~\cite{TomanovOrbits}}}
%%TODO Expertiser les besoin et l’énoncé
%Here we state and derive a more explicit variant of~\cite[Theorem~2]{MargulisTomanov}. 
%\begin{theorem} Let~$W$ be a subgroup of~$G$ such that~$W^+=W$.
%
% And let~$\mu$ be a locally finite $W$-invariant measure on~$G/\Gamma$.
% 
% Then~$\mu$ is finite, and can be written as a countable sum
% $$ .$$
% Where~$\mu^{[L]}$ gather the~$W$-ergodic components of~$W$ which are translates of~$\mu_{\Lpp}$.
% 
% Moreover,~$\mu^{[L]}$ is the restriction~$\mu|_{X^*(L,W)}$ where ...
% 
% Finally~$X^*(L,W)$ is the subset ...
%  
%
%\end{theorem}
%\input{finarticle.tex}

\bibliographystyle{smfart/smfalpha}
\bibliography{bib}

\end{document}